\documentclass[12pt]{amsart}
 \usepackage{amsmath}
\usepackage{amssymb}
\usepackage{amsfonts}
     \usepackage{xcolor}
    \makeatletter
     \def\section{\@startsection{section}{1}%
     \z@{.7\linespacing\@plus\linespacing}{.5\linespacing}%
     {\bfseries \normalfont\scshape
     \centering
     }}
     \def\@secnumfont{\bfseries}
     \makeatother
\scrollmode
   \newtheorem{theorem}{Theorem}[section]
\newtheorem{lemma}[theorem]{Lemma}
\newtheorem{corollary}[theorem]{Corollary}
\newtheorem{proposition}[theorem]{Proposition}

\theoremstyle{definition}

\newtheorem{problem}[theorem]{Problem}
\newtheorem{example}[theorem]{Example}

\newtheorem{remark}[theorem]{Remark}

\newtheorem{comments}[theorem]{Comments}
\numberwithin{equation}{section}

\def \a{{\alpha}}
\def \b{{\beta}}

\def \d{{\delta}}
\def \e{{\varepsilon}}

\def \g{{\gamma}}

\def \k{{\kappa}}

\def \o{{\omega}}

\def \t{{\vartheta}}
\def \m{{\mu}}
\def \s{{\sigma}}

\def \qq{{\qquad}}
 
\def \noi{{\noindent}}
\def \dd{{\rm d}}
 \def\cc{{\color{red}{\rm [$\v c$]}\ }}

\def\E{{\mathbb E \,}}

\def\P{{\mathbb P}}

\def\R{{\mathbb R}}

\def\Z{{\mathbb Z}}

\def\N{{\mathbb N}}

 at 9,5pt

 \font\sevenrm= cmr10 at 7 pt
   
  \font\gsec= cmb10 at 11,5 pt  
\def\beq{\begin{equation}}
\def\eeq{\end{equation}}
  
\def\ben{\begin{eqnarray}}
\def\een{\end{eqnarray}}

\def\ddate {\sevenrm \ifcase\month\or January\or
February\or March\or April\or May\or June\or July\or
August\or September\or October\or November\or December\fi\! {\the\day}, \!{\sevenrm\the\year}}

  \title[\rm 
 ASLLT for sums of independent random variables]
  {Almost sure local limit theorems 
   for sums of independent random variables using
the Bernoulli  decomposition
    of a random variable
}
 \title[\rm 
   ASLLT for sums of independent random variables]  {An application of the Bernoulli  decomposition
    of a random variable  to the Almost sure local limit theorem  in the independent case} 
      \author{Michel J.\,G. WEBER}

\title[\rm 
   A general correlation inequality and application to   ASLLT]  {A general correlation inequality  for level sets 
    of     sums of independent  random variables using 
 the Bernoulli part with  applications to  the almost sure local limit theorem}  
 
\address{IRMA, UMR 7501, Universit\'e
Louis-Pasteur et C.N.R.S.,   7  rue Ren\'e Descartes, 67084
Strasbourg Cedex, France.
   E-mail:    {\tt  michel.weber@math.unistra.fr}}

\begin{document}
 
%%%%%%%

\renewcommand{\thefootnote}{} {{
\footnote{2010 \emph{Mathematics Subject Classification}: Primary: 60F15, 60G50 ;
Secondary: 60F05.}
\footnote{\emph{Key words and phrases}: Local limit theorem, almost sure version,  independent random variables,     
 lattice distributed random variables,  Bernoulli-part of a random variable, square integrability, correlation inequality, uniform asymptotic distribution, quasi-orthogonal system, i.i.d. random variables,  Cram\'er model, a.s. convergent series.}
 \renewcommand{\thefootnote}{\arabic{footnote}}
\setcounter{footnote}{0}
  \begin{abstract} Let $X=\{X_j , j\ge 1\}$ be  a sequence of  independent, square integrable  variables taking values in a common lattice $\mathcal L(v_{
0},D )= \{v_{ k}=v_{ 0}+D k ,  k\in \Z\}$. 
 Let 
$S_n=X_1+\ldots +X_n$,  $S_0=0$,    $a_n= {\mathbb E\,} S_n$, and $\s_n^2={\rm Var}(S_n)\to \infty$ with $n$.  Assume that for each $j$, $\t_{X_j}  =\sum_{k\in \Z}{\mathbb P}\{X_j=v_k\}\wedge{\mathbb P}\{X_j=v_{k+1}\}>0$. 
 
%\begin{theorem} \label{t2.abs}   
\vskip 2 pt Using the Bernoulli part, we prove a general sharp correlation inequality extending the one we obtained in  the i.i.d.\,case in \cite{W3}:  Let $0<\t_j\le \t_{X_j}$  and assume that $
\nu_n =\sum_{j=1}^n \t_j \, \uparrow \infty$, $n\to \infty$.
  Let  $\k_j\in \mathcal L(jv_0,D)$, $j=1,2,\ldots$   be a  sequence of integers   such that 
 \begin{equation*} {\rm(1)}\ \ \frac{\kappa_j-a_j}{\s_j}=\mathcal O(1 )  \qquad
   {\rm(2)}\ \ {\it for \, all}\  j\ge i\ge0,\  (\s_j^2-\s_i^2)^{1/2}  \,{\mathbb P}\{S_j-S_i=\kappa_j-\kappa_i\}  ={\mathcal O}(1). \end{equation*} 
    Then there  exists a constant  
$C $  
  such that for all $1\le m<n$,
\begin{align*} 
 \s_n&\s_m \, \Big|{\mathbb P}\{S_n=\k_n, S_m=\k_m\}- {\mathbb P}\{S_n=\k_n \}{\mathbb P}\{  S_m=\k_m\} \Big|
 \cr &      \,\le \,    \frac{C}{D^2}\, \max \Big(\frac{\s_n }{\sqrt{\nu_n}},\frac{\s_m }{\sqrt {\nu_m}} \Big)^3  \,\bigg\{  \nu_n^{1/2} \prod_{j=m+1}^n\vartheta_j    
+ {\nu_n^{1/2}  \over
 (\nu_n-\nu_m) ^{3/2}}+{ 1\over    \sqrt{{\nu_n\over \nu_m}}-1} \bigg\}.
 \end{align*}

 %\end{theorem}
\vskip 2 pt 
  We derive  a nearly optimal  almost sure local limit theorem: Assume   
that $M(t)=  \sum_{1\le n<t} {1\over \s_n\sqrt{\nu_n} }\uparrow \infty$ with $t$. Given any $R>1$, let $\mathcal M=\{M_j=M(R^{j}), j\ge 1\}$, $I_l=[R^l,R^{l+1}[$, $l\ge 1$.
 
  We prove under moderate and simple conditions, that the series 
  $$
 \sum_{k\ge 1\atop 
I_k\cap \mathcal M\neq\emptyset} 
  \frac{1}{R^k}\sup_{j\ge 1\atop   M_j  
\in I_k} \Big|\sum_{1\le n<R^{j}}{   \vartheta_n(  {\bf 1}_{\{S_n=\kappa_n\}}-{\mathbb P}\{S_n=\kappa_n\} )
  \over 
 \sqrt{\nu_n}}\Big| 
        $$
  converges almost  surely. Further 
 if $  \lim_{n\to \infty}  \s_n {\mathbb P}\{S_n=\k_n\}=
 {D\over  \sqrt{ 2\pi} }e^{-{ \k ^2\over  2   } }
 $, $M(t)$ is slowly varying near infinity,       then the    ASLLT holds, 
\begin{equation*} 
   \lim_{N\rightarrow\infty}  {1\over {\sum_{1\le n<N}{\t_m\over \s_n\sqrt{\nu_n} }}}\ \sum_{1\le n<N}{   {\bf 1}_{\{S_n=\kappa_n\}}  \over   \sqrt{\nu_n}} \, \buildrel{\rm a.s.}\over ={D\over  \sqrt{ 2\pi} }e^{-{ \k ^2\over  2   } }
   .\end{equation*}
Applications are given, notably to the Cram\'er model.
   \end{abstract}
 \maketitle

%%%%%%%%%%%%%%%%%
%%%%%%%%%%%%%%%%%

\section{\gsec INTRODUCTION}\label{s1}
%%%%%%%%%%%%%%%%%
\vskip 10 pt  
%\subsection{Preamble}
Throughout  let $X=\{X_i , i\ge 1\}$ denotes a sequence of  independent  variables taking values in a common lattice $\mathcal L(v_{
0},D )$, namely defined by the
 sequence $v_{ k}=v_{ 0}+D k$, $k\in \Z$, where
 $v_{0} $ and $D >0$ are   real numbers, $D$ is called the  {\it span} of the lattice. 
   Let 
$S_n=X_1+\ldots +X_n$, $n\ge 1$. Then  $S_n $   takes values in the lattice
$\mathcal L( v_{ 0}n,D )$. 
  We assume that the random variables $X_i$ are square integrable, and that $\s_n^2={\rm Var}(S_n)\to \infty$ with $n$. Let also    $a_n= {\mathbb E\,} S_n$, for each $n$.   
In this work, we consider the problem of estimating the probability of the level sets 
$$ \P\{S_n=N\}$$
 $N\in\mathcal L( v_{ 0}n,D )$. %$$ \P\{S_n\in F\}$$ where $F$ is a  possibly infinite  subset  of $\mathcal L( v_{ 0}n,D )$. 
This is obviously a quite important problem, which attracted a lot of attention.
The solution, when it  exists, is described by a famous limit theorem,
 the local limit theorem (LLT).  The sequence $  X$ satisfies a {\it local limit theorem} (in the usual sense) if
 \begin{equation}\label{llt}    \sup_{N=v_0n+Dk }\Big|\s_n {\mathbb P}\{S_n=N\}-{D\over  \sqrt{ 2\pi } }e^{-
{(N-a_n)^2\over  2 s_n^2} }\Big| = o(1), \qq \quad n\to\infty.
\end{equation}

   Remark that the series 
 \begin{equation}\label{def.llt.indep.sum}    \sum_{N=v_0n+Dk }  \Big( {\mathbb P}\{S_n=N\}-{D\over  \sqrt{ 2\pi } B_n}e^{-
{(N-M_n)^2\over  2 B_n^2} } \Big),
\end{equation}
is  obviously convergent, whereas  
%there is no reason for the other series $ \sum_{n }\frac{1}{B_n}$ to be so; thus 
nothing can be deduced concerning its   order  from the very definition of the local limit theorem.     However one can draw from Poisson summation formula that  \begin{equation}\label{def.llt.indep.sum.}    \sum_{N=v_0n+Dk } 
 \Big(  {\mathbb P}\{S_n=N\}-{D\over  \sqrt{ 2\pi } B_n}e^{-
{(N-M_n)^2\over  2 B_n^2} }\Big)\,=\, \mathcal O(D/B_n),
\end{equation}
  see \cite{WM1},  Remark 1.1.  \vskip 1 pt
In the  iid  case the  probability 
 $\P\{S_n=N\}$ 
%, $N\in\mathcal L( v_{ 0}n,D )$.  
can be efficiently estimated by using Gnedenko's local limit theorem \cite{G}, which asserts that  \begin{equation}\label{llt.iid}   
 \sup_{
 %N=v_0n+Dk 
 N\in  \mathcal L( v_{ 0}n,D )}\Big|  \s \sqrt{n}\, {\mathbb P}\{S_n=N\}-{D\over  \sqrt{ 2\pi } }e^{-
{(N-n\m)^2\over  2 n\s^2} }\Big| = o(1),
\end{equation} 
if and only if the span $D$ is maximal,
(there are no  other real numbers
$v'_{0}
$ and
$D' >D$ for which
${\mathbb P}\{X
\in\mathcal L(v'_0,D')\}=1$).
 See  \cite{P} (Theorem 1 and proof of Theorem
2,  p.\ 193--195).
\vskip 3 pt
 The study of this  fine  limit theorem  is intrinsically more complicated than the one of the central limit theorem;   conditions on arithmetical properties of the support of a random variable are for instance always present.   % The   local limit theorem has many interfaces, with Number Theory notably, Ergodic Theory and    applications, for instance in asymptotic enumeration, allocation problems, asymptotics for coefficient polynomials, random permutations, random mappings.
      It was investigated by numerous authors, we only list some recent, and refer to  our  monograph  \cite{SW}   jointly written with Szewczak ($\approx$ 240 references). 
  We cite    % The LLT for sums of independent, identically distributed random variables: 
       % Gnedenko's Theorem, Ibragimov and Linnik characterization of the speed of convergence under moments conditions, stronger forms, Galstyan results, strong LLT's with convergence in variation, versions for densities,   the case of  weighted sums of i.i.d. random variables, local large deviations, Nagaev's result,  Tkachuk and Doney's results,  Diophantine measures and LLT, Breuillard, 
  Gamkrelidze \cite{G2a,Gam80,Gam3,Gam4} (convergence in variation, counterexamples), Mukhin \cite{Mu2, Mu1,Mu} (necessary and sufficient condition, with a recent correction in  Weber \cite{WM2}), Doney \cite{D1,D,D2} (local large deviations), %Shepp, Stone, 
       % LLT and Edgeworth expansions,
        Breuillard \cite{Br} (diophantine measures and Edgeworth expansions),  %, Feller, LLT's under arithmetical conditions. 
%  The LLT for sums of independent  random variables: Prokhorov's theorem, Richter LLT's  and large deviations, Maejima's LLT's with remainder term, LLT's with convergence in variation, 
% Rozanov's necessary condition and uniform asymptotic  distribution, Mitalauskas' LLT  for random variables having stable limit distribution,   Azlarov and  Gamkrelidze's counterexamples,  structural characteristics, 
   MacDonald \cite{M,M1},
 Dabrowski and McDonald \cite{DMD} (Bernoulli part extraction), Giuliano and Weber \cite{GW2,GW3} (effective rate), Macht and Wolf \cite{MW} (using   H\"older-Continuity), R\"ollin  and Ross   \cite{R} (using Landau-Kolmogorov inequalities),      Jacod, Kowalski and Nikeghbali \cite{JKN} (using  Mod-$ \phi$ convergence),  Dolgopyat \cite{Dolgo}, Dolgopyat and  Hafouta \cite{DH}, Hafouta and   Kifer   \cite{HK16,HK18} (under  tightness assumptions), and 
 %The LLT for ergodic sums (expanding maps), by  
 Rousseau-Egele \cite{Rou}, Broise \cite{Bro}, Calderoni, Campanino and Capocaccia \cite{CCC},   Gou\"ezel \cite{Gou} (LLT for ergodic sums of expanding maps).
 %These results mostly concern sufficient conditions for the validity
% of the LLT and its interesting variant forms:
%strong LLT, strong LLT with convergence in variation.
 %Quite importantly are necessary conditions, and the results obtained are sparse,
%essentially: Rozanov's necessary condition, Gamkrelidze's necessary condition, and, almost isolated among the
% flow of results, Mukhin \cite{Mu2} necessary and sufficient condition, see also the recent correction in  Weber \cite{WM2}. In spite of considerable efforts,  the question concerning conditions of validity of the local limit theorem, has up to now no satisfactory solution.

%Extremely useful and instructive  are the counter-examples due to Azlarov and Gamkrelidze,
%as well as necessary and sufficient conditions obtained for a class of random variables,
 %such as Mitalauskas' characterization of the LLT in the strong form for random variables
%   having stable limit distributions. The method of characteristic functions and the Bernoulli   part extraction
% method,  are presented and compared.
    
 \medskip \par
 %The   study
% of almost sure  versions of the local limit theorem is more recent, and is one motivation of this paper. was   instilled in  Denker and Koch 2002 \cite{DK}.  

%The   almost sure local limit theorem (ASLLT)   is comparatively much more recent,    its study  is  naturally   more delicate and there are the results up to day are few, and only concern  the iid case. 
 %These two forms of local limit theorems are introduced in the next sub-section. 

 \vskip 10 pt

 The   study
 of almost sure  versions of the local limit theorem is more recent, and is one     motivation  of this paper.   This notion  was introduced in 2002 by Denker and Koch in
\cite{DK} (sections 1,2), in analogy with the usual almost sure central limit
theorem:
 \lq\lq {\it A stationary sequence of random variables $\{X_n, n\ge 1\}$ taking values in $\R$ or $\Z$ with partial sums $S_n =X_1= \ldots +X_n$ satisfies an almost
sure local limit theorem, in short ASLLT, if there exist sequences
$\{a_n, n\ge 1\}$ in
$\R$ and
$\{b_n, n\ge 1\}$ in $\R^+$ satisfying $b_n\to \infty$, such that 
\begin{equation}\label{asllt}\lim_{N\to \infty}{1\over \log N}\sum_{n=1}^N {b_n\over n} \chi\{ S_n\in k_n+I\}\buildrel{a.s.}\over{ =}g(\kappa)|I| \quad {\rm
as}\quad {k_n-a_n\over b_n}\to
\kappa, 
\end{equation}  where $g$ denotes some density and $I\subset \R$ is some bounded interval.  Further $|I|$ denotes the length of the interval $I$
in the case where $X_1$ is real valued and the counting measure of $I$ otherwise.}\rq\rq 

\vskip 3 pt 
The above  definition is  however   incomplete, 
   as 
    remarked in Weber   \cite{W3}, Section 4. Assume
that 
  $\P\{X_i\in \mathcal L(v_0,D)\}=1$, for each $i$.
 %We also assume that $\s^2= \E X_1^2<\infty$ and let  
%$\m=\E X_1$. 
As $g$ is a density, there are reals $\kappa$ such that
$g(\kappa)\not= 0$.  Now if
$\{ k_n, n\ge 1\}$ is such that  ${k_n-a_n\over
b_n}\to
\kappa$, then   any  sequence $\{\k_n, n\ge 1\}$, $\k_n=k_n +u_n$ where $u_n$ are uniformly bounded  also satisfies this. But we can
arrange the $u_n$ so that $\k_n\notin \mathcal L( v_0,D)$ for all $n$. 
%Therefore $\P\{S_n=\k_n\}\equiv 0$. 
Picking $I=[-\d, \d]$ with $ \d<1/2$, we get   
$$ \lim_{N\to \infty}{1\over \log N}\sum_{n=1}^N {b_n\over n} \chi\{ S_n\in k_n+I\}\buildrel{{\rm a.s.}}\over{ =}0\not=g(\kappa)|I| ,
$$ 
hence a contradiction. It is  thus necessary  to    assume 
\begin{equation}\label{addasllt} \k_n \in \mathcal L(nv_0,D), \qq n= 1,2, \ldots,
\end{equation}
%Then   $k_n+I\subset L(nv_0,D)$ if and only if $ I\subset L(0,D)$. 
  also  to  change $|I|$ for $\#\{ I\cap
\mathcal L( v_0,D)\}$. Then (\ref{asllt})  modified   becomes,
\begin{equation}\label{asllt1}\lim_{N\to \infty}{1\over \log N}\sum_{n=1}^N {b_n\over n} \chi\{ S_n\in k_n+I\}\buildrel{{\rm a.s.}}\over{
=}g(\kappa)\#\{ I\cap \mathcal  L(  v_0,D)\}, \  {\rm as}\  {k_n-a_n\over b_n}\to
\kappa,  
\end{equation}
where $I$ is a bounded interval.  
\vskip 3 pt 
Translating it in the independent case, we consider ($b_n=\s_n$, $g(x)= {D\over  \sqrt{ 2\pi } }e^{-
{x^2/  2 } }$), this means that 
the ASLLT holds, by definition, if  \begin{equation}\label{asllt2}\lim_{N\to \infty}{1\over \log N}\sum_{n=1}^N {\s_n\over n} \chi\{ S_n= k_n \}\buildrel{{\rm a.s.}}\over{
=}{D\over  \sqrt{ 2\pi } }e^{-
{x^2/  2 } } , \qq {\rm  whenever}  \quad  {k_n-a_n\over \s_n}\to
\kappa.  
\end{equation}
However, as we shall see in Section \ref{s4}, (Theorem \ref{t1[asllt].}) that formulation is not appropriate, and the right one we prove turns up to be  more complicated,   involving notably the sequence of parameters $\{\t_n,n\ge 1\}$. One must nevertheless admit that the right formulation is difficult to guess, the i.i.d. case being generally weakly informative of the independent non identically distributed case. As mentioned by the authors  in \cite{DK}, p.146, the existence of almost sure local limit theorems is of fundamental
interest.  
%A recent application to a problem of representation of integers is given in Giuliano-Weber \cite{GW2.}. 
It seems reasonable to expect     applications, notably at the interface with Number Theory.
\vskip 3 pt\vskip 3 pt
The inherent second order study,
 which has its own interest,  is much more difficult than for establishing the almost sure central limit theorem.
 The   almost sure local limit theorems are very recent and  already cover  the i.i.d. case, the stable case,
 Markov chains,   the model of the Dickman function.
 %,  and  the independent case, with almost sure convergence of related series. 
 When the random variables are identically distributed, the ASLLT states as follows.
\begin{theorem} \label{t1[asllt]..} Let $X$ be a square integrable  random variable taking values on the lattice $\mathcal L(v_0,D)= \{v_0+kD, k\in \mathbb Z\}$  with maximal span $D$. Let $\m ={\mathbb E\,} X$,
$\s^2={\rm Var} (X)>0$.  Let also $ \{X_k, k\ge 1\}$ be independent copies of
$X$, and put
$S_n=X_1+\ldots +X_n$, $n\ge 1$.   Then
$$ \lim_{ N\to \infty}{1\over    \log N } \sum_{ n\le
N}  {  1 \over \sqrt n} {\bf 1}_{\{S_n=\kappa_n\}} \buildrel{a.s.}\over {=}{D\over
\sqrt{ 2\pi}\s}e^{-  {\k^2/ ( 2\s^2 ) } },$$
  for any  sequence of integers $\k_n\in \mathcal L(nv_0,D)$, $n=1,2,\ldots$      such that
  \begin{equation}\label{eq2}  \lim_{n\to \infty} { \k_n-n\m   \over    \sqrt{  n}  }= \k >0.
\end{equation}
 \end{theorem}

Note that by    \eqref{llt.iid},  if
$ \k_n \in \mathcal L(nv_0,D)$ is a sequence which   verifies condition (\ref{eq2}), then
 \begin{equation}\label{llt1}  \lim_{n\to \infty}  \sqrt n {\mathbb P}\{S_n=\k_n\}={D\over  \sqrt{ 2\pi}\s}e^{-
{ \k ^2\over  2   \s^2} } .
\end{equation}

Theorem \ref{t1[asllt]..} was announced  in Denker and Koch \cite{DK} (Corollary 2). The succint  proof given however contains a gap, see   Weber   \cite{W3}. A complete proof was given in  \cite{W3}.  Later, in the well-written paper \cite{GS}  adressing the  same problem in the stable i.i.d. case,  more precisely, for  specific classes of  stable i.i.d. random variables,  Giuliano and Szewczak,  recovered  that result   as a particular case.   One interesting aspect of the  approach used in  \cite{GS} is that it is  based on Fourier analysis, and is  thus   different from ours. See also Giuliano and Szewczak \cite{GS13}, for a result of this kind concerning Markov chains. See Section \ref{sub2..1}  for  a detailed exposition and a new improvment. 
\vskip 3 pt 
%The following remark,  Remark 3 in \cite{W3}, is instructive in the present context. We observe that if $\kappa_n$  satisfies \eqref{eq2},
%   $$ \lim_{n\to \infty}\sqrt n \P\{S_n=\ell_n\}= {D\over  \sqrt{ 2\pi}\s}e^{- 
%{\k^2\over  2  \s^2} }, \qq (\ell_n\equiv \k_n\ {\rm or}\ \ell_n\equiv\k_n+D)  . 
%$$
%Then for some $n_\k<\infty$, $\P\{S_n=\k_n\}\wedge \P\{S_n=\k_n+1\}>0 $ if $n\ge n_\k$.   
%Changing $X$ for   $X'=S_{n_\k}$, we see that $X'$ satisfies (\ref{basber}).   
\vskip 5 pt An ASLLT with rate    was primarily  proved in Giuliano-Weber \cite{GW3}.  
 \begin{theorem}[\cite{GW3}, Theorem 1] \label{tb}Assume that  $\E X^{2+\e}<\infty$  for some positive $\e$. Then,   
$$ \lim_{ N\to \infty}{1\over    \log N } \sum_{ n\le
N}  {  1 \over \sqrt n} {\bf 1}_{\{S_n=\kappa_n\}} \buildrel{a.s.}\over {=}{D\over  \sqrt{ 2\pi}\s}e^{-
{ \k ^2\over  2   \s^2} },$$
  for any  sequence of integers $\{\k_n, n\ge 1\}$     such that \eqref{llt1}  holds. 
 Moreover, if \eqref{eq2} is sharpened 
as follows,
$$    { \k_n-n\m   \over    \sqrt{  n}  } = \k + \mathcal O_\eta\big(  (\log n)^{-1/2+\eta}        \big),
$$then 
$${1\over    \log N } \sum_{ n\le
N}  {  1 \over \sqrt n} {\bf 1}_{\{S_n=\kappa_n\}}  \buildrel{\rm a.s.}\over {=}{D\over  \sqrt{ 2\pi}\s}e^{-
{ \k ^2\over  2   \s^2} } +\mathcal O_\eta\big(  (\log N)^{-1/2+\eta}      
 \big) \Big).$$ 
\end{theorem}

\vskip 10 pt  The  proof of Theorem \ref{t1[asllt]..} mainly relies upon    on sharp correlation inequalities  for the associated level sets,
\begin{equation}\label{corr.asllt}|\P\{S_m=k_m, S_n= k_n\}-\P\{S_m=k_m\}\P\{S_n=k_n\}|
 .
 \end{equation}
which are also established in \cite{W3}, and   are much harder to get than  the correlation inequalities 
 \begin{equation}\label{corr.asclt}|\P\{S_m<k_m, S_n<k_n\}-\P\{S_m<k_m\}\P\{S_n<k_n\}|,  \end{equation}
needed for the proof of global a.s.\ central limit theorems, see e.g.\ Lacey  and Philipp \cite{LP}. 
%\bigskip \par  
\subsection{\bf General problem investigated}\label{GS} It is a well-known  in probability theory that the iid case and the independent case are always quite a different matter; one   for instance can take the example of the law of the iterated logarithm (\cite{P}). As announced,    we consider the independent, non necessarily identically distributed case, and our main objective   will be  to establish  in this very large setting  a general correlation inequality, next to apply it to prove the corresponding ASLLT. In fact we will prove  new general stronger forms of it, valid    in a more larger setting than  of the LLT,   this one being not assumed, and we also show that the validity conditions we found are nearly optimal. The search of that correlation inequality has revealed   new facts,  which result from the investigation of a more general, and in the same time, more complex  case.  

%We shall already explain a singularity that occured in the course of the study.
 \vskip 3 pt   
\subsection{\bf Characteristic  of a random variable} Let $Y$ be
  a random variable   such that  ${\mathbb P}\{Y
\in\mathcal L(v_0,D)\}=1$. We do not   assume  that the  span $D$ is maximal.  
  Introduce the following characteristic,  
\begin{eqnarray}\label{vartheta}  \t_Y =\sum_{k\in \Z}{\mathbb P}\{Y=v_k\}\wedge{\mathbb P}\{Y=v_{k+1}\} ,
\end{eqnarray}
 where $a\wedge b=\min(a,b)$.
Obviously $\t_Y$   depends on $D$.
 Note    that we \underline{always} have the relation
   \begin{eqnarray}\label{vartheta1}
0\le \t_Y<1 .\end{eqnarray}
When $Y$ has finite  variance
$\sigma^2$,   we have the following important liaison inequality,
 \begin{eqnarray}  \label{tD}\s^2 \,\ge\,  \frac{ D^2   }{4}\, \t_Y,
\end{eqnarray}
from which it follows that (since $X_i$ are independent, ${\rm Var}(X_1+\ldots + X_n)={\rm Var}(X_1)+\ldots + {\rm Var}(X_n))$,
 \begin{eqnarray}  \label{tDn}\qq \qq \qq \qq\  \s_n^2-\s_m^2\,\ge  \,\frac{ D^2   }{4}\ \sum_{j=m+1}^n \t_{X_j}, \qq \quad n\ge m\ge 1.
\end{eqnarray}    
For the proofs of \eqref{vartheta1}, \eqref{tD}, we refer the reader to the article Giuliano-Weber \cite{GW3}, subsection\,2.1, or to the recent monograph  Szewczak-Weber \cite{SW}, see (1.136) and after, also on the equivalence with the 
  \lq\lq smoothness\rq\rq characteristic  \begin{eqnarray}\label{delta}  \d_X =\sum_{m\in \Z}\big|{\mathbb P}\{X=m\}-{\mathbb P}\{X=m-1\}\big|,
\end{eqnarray}   
 introduced by Gamkrelidze in \cite{G2}. \bigskip \par

 \subsection{\bf General  Assumption}\label{nun} Throughout  this work   we assume that  
 \begin{equation}\label{nun1} \t_{X_j}>0, \qq \hbox{for each $j\ge 1$}.
  \end{equation}  
  
  We select $\t_j$ so that 
    \begin{equation}\label{nun2}
   \quad 0<\t_j\le \t_{X_j}  \qq \text{for each}\ j\ge 1,
 \end{equation}
 and we assume that \begin{equation}\label{nun3}
   \nu_n :=\sum_{j=1}^n \t_j \, \uparrow \infty, \qq \text{as}\  n\to \infty.
 \end{equation} 

%Assumption \eqref{nun} is really the  {basic} requirement to make in order to investigate the problem studied in the general setting described in  subsection \ref {GS}.  
The sequences    $\{{\rm Var}(X_j), j\ge 1\}$ and $\{\t_{X_j}, j\ge 1\}$ we consider,  are otherwise {\it arbitrary}, and  are the main parameters in this study.
%\begin{remark}\label{bbl}\rm    
Condition  (\ref{nun1}) is   natural in our setting. Assume the local limit
theorem (\ref{llt}) to be applicable to the sequence $X$. Let    $\{\k_n, n\ge 1\}$ be a sequence such that
 \begin{equation}\label{2..}  \lim_{n\to \infty} { \k_n-a_n   \over    \s_n  }= \k .
\end{equation}  
The local limit
theorem implies
\begin{equation}\label{2...}  \lim_{n\to \infty}  \s_n {\mathbb P}\{S_n=\k_n\}=
 {D\over  \sqrt{ 2\pi} }e^{-{ \k ^2\over  2   } }.
  \end{equation}
And so   $$ \lim_{n\to \infty} \s_n\P\{S_n=\ell_n\}= {D\over  \sqrt{ 2\pi} }e^{- 
{\k^2\over  2   } }, \qq (\ell_n\equiv \k_n\ {\rm or}\ \ell_n\equiv\k_n+D)  . 
$$
Then for some $n_\k<\infty$, $\P\{S_n=\k_n\}\wedge \P\{S_n=\k_n+1\}>0 $ if $n\ge n_\k$. Changing $X$ for   $X'=S_{n_\k}$, we see that the new sequence $X'$ satisfies
(\ref{nun1}).   This was used in the course of the proof of the ASLLT in \cite{W3}. It is true in general,   and   has some degree of importance, at the light of    assumption \eqref{nun1}.
%\end{remark}
%\bigskip \par 
%   \vskip 3 pt
   The second order theory of a probabilistic system  is a key   element   of its study. Combined with criteria of almost everywhere convergence,  it  allows one  to prove almost sure convergence results. This is the standard approach for treating these questions.
In this case here, we study the correlations properties of the system of   set's     indicators 
  $$\mathcal T(\k)= \Big\{\big({\bf 1}_{\{S_n=\k_n\}}-{\mathbb P}\{S_n=\k_n\} \big),\quad n\ge 1\Big\},$$ 
  where  $ \k=\{\k_n, n\ge 1\}$, and  $\k_n\in \mathcal L(nv_0,D)$ for each $n$.   %The works made in \citet{W3}, \cite{GW3} are based on a coupling method called the  \lq\lq {\it Bernoulli part extraction}\rq\rq\, of a random variable, which is essentially  due to McDonald \cite{M}. In this paper, that method is applied to study  the almost sure local limit theorem in the independent non identically distributed case.
 %\bigskip \par 
  \vskip 3 pt
The  
  general   correlation inequality as well as some corollaries are stated in the next Section. The proof is technically complicated and is given in Section \ref{s3}. An obvious  reason  is that  in addition to the  sequence of   variances   ${\rm Var}(X_j)$,  the  sequence of characteristics $\t_{X_j}$  is   involved in {\it all} estimates. 
In section \ref{s2.}, preliminary results are collected. We prove in Section \ref{s4} that for independent square integrable random variables,  the almost sure local limit theorem  still holds, under fairly reasonable conditions, but the proof  is  more involving.  %(closely related to the smoothness characteristic of Gamkrelidze \cite{G1})
  We also prove    in Section \ref{sub2..1}, an almost sure local limit theorem with speed of convergence   in the i.i.d. square integrable case, and show the almost sure convergence of   tightly related random series. 

\medskip \par

%%%%%%%%%%%%
%\subsection{The correlation inequality}
%%%%%%%%%%%%

\vskip 5 pt {\it Notation.} Throughout the paper,  the letter $C$ denotes a universal constant whose value may change at each occurence, and   $C_{\a, \b, \ldots}$ a denotes a constant depending only on the parameters ${\a, \b, \ldots}$.

%%%%%%%%%%%%%%%%%
%%%%%%%%%%%%%%%%%

\vskip 25 pt

\section{\gsec A GENERAL CORRELATION INEQUALITY}\label{s1a}
\vskip 10 pt 
%%%%%%%%%%%%%%%%%

\vskip 4 pt  We   prove the following result in which a new fact is that we do not  assume a local limit theorem to be applicable, neither   an integral limit theorem   to hold. 
%This was in fact    implicit in \cite{W3}.

 \begin{theorem}[{\bf Main Result}] \label{t2[asllt]} 
 %Assume that
 %   \begin{equation}\label{nun}
 % {\rm (1)}\quad 0<\t_j\le \t_{X_j}  \ \text{for each}\ j,\qq  {\rm (2)}\quad  \nu_n =\sum_{j=1}^n \t_j \, \uparrow \infty, \ \text{as}\  n\to \infty.
 %\end{equation} 
 Let  $\k_j\in \mathcal L(jv_0,D)$, $j=1,2,\ldots$,     be   such that 
 \begin{eqnarray}\label{nun.ka}{\rm(1)}& &\frac{\kappa_j-a_j}{\s_j}=\mathcal O(1 ),
 \cr {\rm(2)}&& {\it for \, all}\  j\ge i\ge0,\  (\s_j^2-\s_i^2)^{1/2}  \,{\mathbb P}\{S_j-S_i=\kappa_j-\kappa_i\}  ={\mathcal O}(1).
 \end{eqnarray} 
   
   Then there  exists a constant  
$C $  
  such that for all $1\le m<n$,
\begin{align} \label{t2[asllt]enonce1}
 \s_n&\s_m \, \Big|{\mathbb P}\{S_n=\k_n, S_m=\k_m\}- {\mathbb P}\{S_n=\k_n \}{\mathbb P}\{  S_m=\k_m\} \Big|
 \cr &      \,\le \,    \frac{C}{D^2}\, \max \Big(\frac{\s_n }{\sqrt{\nu_n}},\frac{\s_m }{\sqrt {\nu_m}} \Big)^3  \,\bigg\{
 \nu_n^{1/2}  \Big(\prod_{j=m+1}^n\vartheta_j    
+  {1  \over (\nu_n-\nu_m) ^{3/2}}  \Big)
+{ 1\over    \sqrt{{\nu_n\over \nu_m}}-1} \bigg\}.
 \end{align}
  \end{theorem}

  The proof   is   delicate and   long. 
  Some comments are in order.

\begin{comments}\label{comments.t2[asllt]}\vskip 2 pt (1)   It is necessary to observe that $\t_{X_j}$ cannot be too small. This follows from the  compensation effect  existing between assumptions \eqref{nun2} and \eqref{nun3}. 
\vskip 2 pt (2) Condition \eqref{nun3}   implies that $\{S_n, n\ge 1\}$ is   asymptotically uniformly distributed, a.u.d. in short. See Weber \cite{WM1}, Th.\,3.3. It is known that the local limit theorem is applicable to the sequence $\{S_n, n\ge 1\}$ only if the a.u.d. property is satisfied. %An interesting question is whether  condition \eqref{nun}  implies that $\{S_n, n\ge 1\}$   satisfy an integral limit theorem. 

  \vskip 2 pt (3) It is also known that if $ X$ satisfies a local limit theorem in the strong form, it is necessary  that Rozanov's condition be fulfilled, namely that
\begin{equation}\label{roz}
\sum_{k=1}^\infty \ \min_{0\le m<q} {\mathbb P}\big\{ X_k\not\equiv m\, ({\rm mod} \ q) \big\}= \infty,\qq \hbox{for all integers $q\ge 2$.}
\end{equation}
 That condition is also sufficient in some important examples, in particular if    $X_j$ have stable limit distribution, see Mitalauskas  \cite{Mit}. 
 % On the other hand  it is known that the local limit theorem always implies the integral limit theorem.   It was believed that    the integral limit theorem plus the a.u.d. property for partial sums is equivalent to the local limit theorem.   Gamkrelidze  built in \cite{Gam} a remarkable counter-example.
  
\vskip 2 pt 
(4) %By   \eqref{vartheta1}, $0\le \t_{X_j}<1$, for each $j\ge 1$. In assumption \eqref{unif.bound.varthetaj}, we require the second inequality  to   being true uniformly over all  $j\ge 1$.
%\vskip 2 pt (2) 
 % Since we  only required that $0<\t_j\le \t_{X_j}$, we may choose then so that   assumption \eqref{unif.bound.varthetaj} be    satisfied. But in that case, the factors $ {\s_n }/{\sqrt{\nu_n}}$ get bigger, as well as ${1 /(\nu_n-\nu_m) ^{3/2}}$, thereby giving rise to a larger bound, up to some extend, on a case-by-case basis. 
 Choosing  $\t_j$ smaller (see (1) however)  may increase  the size of the correlation bound, up to some extend, and on a case-by-case basis. 
   
\vskip 2 pt  (5) Although $0< \t_{X_j}<1$  for each $j$,   we may have   $\t_{X_j}$   arbitrary close to 1, for infinitely many $j$, and so assumption \eqref{unif.bound.varthetaj} below  is no longer true  if we pick $\t_j$ near $\t_{X_j}$, for each $j$ (according to previous remark).     It is interesting to
  estimate   
 $\prod_{j=m+1}^n\t_{X_j}  $ in this case on an  example. 
 % \vskip 2 pt Consider the following   example. 
 \begin{example}\label{example1}  Let $X_j$ be independent random variables, each defined as follows, 
 \begin{equation*} 
\P\{ X_j= v_m\}= \begin{cases} 0 &\quad\hbox{if\ }m\le n_j \hbox{\ or if\ } m>n_j+b_j,\cr
\frac{1}{b_j} & \quad\hbox{if\ } m=n_j+1, \ldots, n_j+b_j, \end{cases}
 \end{equation*}
where $n_j,b_j$ are positive integers,   $n_j$ can be all equal, $b_j\to \infty$ with $j$,  and the series $\sum_{j\ge 1} \frac1{b_j}$ diverges. 
%We do not exclude the case $n_j,b_j$ are all the same for each $j$.
%\vskip 2 pt We assume that  for some $b\ge 3/2$, \begin{equation}\label{ass.ex1}  \sum_{j=m+1}^n \frac1{b_j}\ge b\,\log (n-m). \end{equation} 
%Thus  the series $\sum_{j\ge 1} \frac1{b_j}$ diverges.  For instance this holds if $b_j=j^{-\e}$ with $1<\e<1$. 
We have  $\t_{X_j }= \sum_{k\in \Z}{\mathbb P}\{X_j=v_k\}\wedge{\mathbb P}\{X_j=v_{k+1}\} =\sum_{m=n_j+1 }^{n_j+b_j-1}{\mathbb P}\{X_j=v _m\}\wedge{\mathbb P}\{X_j=v_{m+1}\}$, and so $\t_{X_j }=1-\frac1{b_j}$, for each $j$. 
Thus for some $b>0$,
$$\prod_{j=m+1}^n\t_{X_j} =\prod_{j=m+1}^n\big(1-\frac1{b_j}\big)%\, \approx\,  e^{-b\sum_{j=m+1}^n \frac1{b_j} }
% \, \le\,   \frac{1}{(\sum_{j=m+1}^n\t_{ j})^b } 
 .$$
 If for instance $b_j =\frac1j$, this product contributes   for $\frac{m}{n}$. Choosing $\t_j$ sufficiently close to $ \t_{X_j}$, for each $j$, we have $\prod_{j=m+1}^n\t_{j}\approx\,\frac{m}{n}$,   which, for $m=o(n)$, is bigger   than $ {1  \over (\nu_n-\nu_m) ^{3/2}}\sim {1  \over ( n- m) ^{3/2}}$ in the right-term of inequality \eqref{t2[asllt]enonce1}. 
%It follows that estimate \eqref{t2[asllt]enonce2}  remains valid in this case either. By pushing the calculations a little bit more we also get$$   {\rm Var}(X_j)= D^2 \bigg\{\frac{   5n_j^2+4n_jb_j  +b_j ^2-b_j-3n_j }{6 } -\frac{( b_j - 1 )^2}{4}\bigg\}$$
 \end{example} 
\end{comments}
\vskip 5 pt The following Corollaries are immediate consequences of Theorem \ref{t2[asllt]}.

\begin{corollary}\label{c1[cor.est]}Under   assumptions of Theorem \ref{t2[asllt]}, suppose that $\t_n$ are \underline{chosen} so that 
   \begin{equation}\label{unif.bound.varthetaj}\tau:=\sup_{j\ge 1} \t_{ j}<1.\end{equation} 
%  \begin{equation}\label{unif.bound.varthetaj}\tau_X:=\sup_{j\ge 1} \t_{X_j}<1.\end{equation} 

 Then we  have the simplified bound,
 \begin{align}\label{t2[asllt]enonce2} 
 \s_n&\s_m \, \Big|{\mathbb P}\{S_n=\k_n, S_m=\k_m\}- {\mathbb P}\{S_n=\k_n \}{\mathbb P}\{  S_m=\k_m\} \Big|
 \cr &  \, \le        \,   \frac{C_{\tau }}{D^2}\, \max \big(\frac{\s_n }{\sqrt{\nu_n}},\frac{\s_m }{\sqrt {\nu_m}} \big)^3
 \bigg\{\,  { 1\over    \sqrt{{\nu_n\over \nu_m}}-1}+\,{\sqrt{\nu_n} \over
 (\nu_n-\nu_m) ^{3/2}} \bigg\}.
\end{align}
 \end{corollary} 
 Note that we do \underline{not} require the below different condition 
 \begin{equation}\label{unif.bound.varthetaj.}\tau_X:=\sup_{j\ge 1} \t_{X_j}<1,\end{equation} 
 to be hold. See also at this regard   (4) in Comments \ref{comments.t2[asllt]}.
Condition  \eqref{unif.bound.varthetaj} trivially holds in the  i.i.d. case, choosing  $\vartheta=\t_{X_1}>0$. Thus the previous estimate  contains the correlation estimate in the i.i.d. square integrable case, with weaker requirements; in particular we do not assume a  LLT to be applicable unlike   in \cite{W3}. 
 
 \vskip 2 pt We state:

\begin{corollary}[{\bf i.i.d.\! case}]\label{c2[cor.est.i.i.d]} Let $X$ be a square integrable  random variable taking values on the lattice $\mathcal L(v_0,D)= \{v_0+kD, k\in \mathbb Z\}$  with maximal span $D$. Let $\m ={\mathbb E\,} X$,
$\s^2={\rm Var} (X)>0$ and assume that $\t_{X}>0$.  Let also $ \{X_k, k\ge 1\}$ be independent copies of
$X$, and put
$S_n=X_1+\ldots +X_n$, $n\ge 1$.  Let  $\k_j\in \mathcal L(jv_0,D)$, $j=1,2,\ldots$   be   such that \begin{equation}\label{nun.ka.iid}{\rm(1)}\qquad\frac{\kappa_j- j\m}{ \s\sqrt j}={\mathcal O}  (1), \qq\quad  {\rm(2)}\qquad   \s \sqrt j \,{\mathbb P}\{S_j=\kappa_j\}  ={\mathcal O}  (1).
 \end{equation}
 Then  for all $1\le m<n$,
\begin{align} \label{t2[asllt]enonce1.iid}
\s^2\sqrt {nm}&  \, \Big|{\mathbb P}\{S_n=\k_n, S_m=\k_m\}- {\mathbb P}\{S_n=\k_n \}{\mathbb P}\{  S_m=\k_m\} \Big|
% \cr &      \,\le \,    \frac{C}{D^2}\,    \,\bigg\{ n^{1/2}  \Big( \vartheta^{n-m}    +  {1  \over (\vartheta(n-m)) ^{3/2}}  \Big)+{ 1\over    \sqrt{{ n\over  m}}-1} \bigg\}
 \cr &      \,\le \,    \frac{C_\vartheta}{D^2}\,    \,\Big\{
         {n^{1/2}   \over  (n-m)  ^{3/2}}  
+{ 1\over    \sqrt{{ n\over  m}}-1} \Big\}.
 \end{align}
\end{corollary}
   % \vskip 2 pt  
Now we pass to another Corollary, which the proof is given in Section \ref{s3}.

 \begin{corollary}\label{cort1}      Let $0<c<1$. Under assumptions of Theorem \ref{t2[asllt]} and \eqref{unif.bound.varthetaj},  the following assertions are fulfilled:
 \vskip 2 pt {\rm (1)} There   exists a constant
$C_{\tau, c}  $  such that for all $1\le \nu_m\le c\nu_n$,
\begin{align*} 
 \s_n \s_m \, \, \Big|{\mathbb P}\{S_n=\k_n, S_m=\k_m\}&- {\mathbb P}\{S_n=\k_n \}{\mathbb P}\{  S_m=\k_m\} \Big|
   \cr &      \,\le \,   \frac{C_{\tau, c}}{D^2}\, \max \big(\frac{\s_n }{\sqrt{\nu_n}},\frac{\s_m }{\sqrt {\nu_m}} \big)^3  \,\sqrt{\nu_m\over \nu_n}.
    \end{align*}

\vskip 2 pt {\rm (2)} Further  if, 
\begin{equation}\label{snlenun} \s_n = \mathcal O( \sqrt {\nu_n}),\end{equation}
then for some other constant $C'_{\tau, c}$, 
we have for all $1\le \nu_m\le c\nu_n$, 
\begin{equation}\label{c.bound} 
 \s_n \s_m \, \, \Big|{\mathbb P}\{S_n=\k_n, S_m=\k_m\} - {\mathbb P}\{S_n=\k_n \}{\mathbb P}\{  S_m=\k_m\} \Big|
        \,\le \,   \frac{C'_{\tau, c}}{D^2}\,  \sqrt{\nu_m\over \nu_n}.
    \end{equation}
    \end{corollary}
\begin{remark}\label{remark.c.bound} Assumption \eqref{snlenun} implies in view of \eqref{tDn} that $\s_n\asymp \sqrt {\nu_n}$.
\vskip 2 pt Thus for the case \begin{equation}\label{c.bound.}\s_n\asymp \sqrt {\nu_n},\qq  \sup_{j\ge 1} \t_{ j}<1.
\end{equation}
estimate  \eqref{c.bound} holds, and  provides  a very handable correlation bound. 
 \end{remark}%{{}Case $X_j=Y_j\chi\{|Y_j|\le a_n\}$, $a_n\uparrow\infty$ where $Y,Y_1,\ldots, Y_n$ are i.i.d.: It is necessary to assume $\t_{Y}>0$.  Then $\t_{X_j}\to \t_{Y}$, $\nu_n\sim \t_{Y}n$, $\s_j\to \s_Y$}

  %  \vskip 8 pt\vskip 3 pt  %The  proof is characteristic functions free;   in place it uses the Bernoulli part extraction method.   

\begin{remark}\label{simple.correlation.bound} We indicate here a   simple correlation bound, involving only the sequence $\{\s_n,  n\ge 1\}$. 
For $n>m\ge 1$,
\begin{align*}
%\label{basic.corr.bd}
   \s_m\s_n\Big|{\mathbb P} \{S_m =\kappa_m\,,\,S_{n } =\kappa_n\}-{\mathbb P}\{ S_{ m} = \kappa_m\}  {\mathbb P} \{S_n
=\kappa_n
\}\Big|    \le  C\, \Big(\frac {\s_n}{\s_{n-m}}+1\Big).
\end{align*}
The  proof is elementary, see \eqref{basic.corr.bd}. \end{remark}

%%%%%%%%%%%%%%%%%
%%%%%%%%%%%%%%%%%

\vskip 25 pt 
%%%%%%%%%%%%%%%%%

%%%%%%%%%%%%%%%%%

 \section{\gsec  PRELIMINARIES}\label{s2.}
 \vskip 10 pt 
%%%%%%%%%%%%%%%%%

 \subsection{The Bernoulli part extraction of a random variable.}\label{sub2.1}

  Let $X$ be
  a random variable   such that  ${\mathbb P}\{X
\in\mathcal L(v_0,D)\}=1$. 
We assume  that 
\begin{equation}\label{varthteta2}\t_X>0.
\end{equation}
Let $ f(k)= {\mathbb P}\{X= v_k\}$, $k\in \Z$. Let also $0<\t\le\t_X$. One can associate to $\t$ and $X$  a
sequence $  \{ \tau_k, k\in \Z\}$     of   non-negative reals such that
\begin{equation}\label{basber0}  \tau_{k-1}+\tau_k\le 2f(k), \qq  \qq\sum_{k\in \Z}  \tau_k =\t.
\end{equation}
Take for instance  
 $\tau_k=  \frac{\t}{\nu_X} \, (f(k)\wedge f(k+1))  $.
   Now   define   a pair of random variables $(V,\e)$   as follows:
  \begin{eqnarray}\label{ve} \qq\qq\begin{cases} {\mathbb P}\{ (V,\e)=( v_k,1)\}=\tau_k,      \cr
 {\mathbb P}\{ (V,\e)=( v_k,0)\}=f(k) -{\tau_{k-1}+\tau_k\over
2}    .  \end{cases}\qq (\forall k\in \Z)
\end{eqnarray}
   By assumption     this is   well-defined,   and the margin  laws verify
\begin{eqnarray}\begin{cases}{\mathbb P}\{ V=v_k\} &= \  f(k)+ {\tau_{k }-\tau_{k-1}\over 2} ,
\cr
 {\mathbb P}\{ \e=1\} &= \ {\vartheta} \ =\ 1-{\mathbb P}\{ \e=0\}   .
\end{cases}\end{eqnarray}
Indeed, ${\mathbb P}\{ V=v_k\}= {\mathbb P}\{ (V,\e)=( v_k,1)\}+ {\mathbb P}\{ (V,\e)=( v_k,0)\}=f(k)+ {\tau_{k }-\tau_{k-1}\over 2} .$
 Further  ${\mathbb P}\{ \e=1\}  =\sum_{k\in\Z} {\mathbb P}\{ (V,\e)=( v_k,1)\}=\sum_{k\in\Z} \tau_{k }={\vartheta}  $.

 \begin{lemma} \label{bpr} Let $L$
be a Bernoulli random variable    which is independent of  $(V,\e)$, and put
 $Z= V+ \e DL$.
We have $Z\buildrel{\mathcal D}\over{ =}X$.
\end{lemma}
\begin{proof}[Proof]
 Plainly,
\begin{eqnarray*}{\mathbb P}\{Z=v_k\}&=&{\mathbb P}\big\{ V+\e DL=v_k, \e=1\}+ {\mathbb P}\big\{ V+\e DL=v_k, \e=0\} \cr
&=&{{\mathbb P}\{ V=v_{k-1}, \e=1\}+{\mathbb P}\{
V=v_k, \e=1\}\over 2} +{\mathbb P}\{ V=v_k, \e=0\}
\cr&=& {\tau_{k-1}+ \tau_{k }\over 2} +f(k)-{\tau_{k-1}+ \tau_{k
}\over 2}
= f(k).
\end{eqnarray*}
\end{proof} Consider now independent random variables  $ X_j,j=1,\ldots,n$,    and assume that
\begin{equation}\label{basber.j}  \t_{X_j}>0, \qq \quad  j=1,\ldots, n.
\end{equation}
Let $0<\t_j\le \t_{X_i}$, $j=1,\ldots, n$.  
 One can  associate to them a
sequence of independent vectors $ (V_j,\e_j, L_j) $,   $j=1,\ldots,n$  such that
 \begin{eqnarray}\label{dec0} \big\{V_j+\e_jD   L_j,j=1,\ldots,n\big\}&\buildrel{\mathcal D}\over{ =}&\big\{X_j, j=1,\ldots,n\big\}  .
\end{eqnarray}

Further the sequences $\{(V_j,\e_j),j=1,\ldots,n\}
 $   and $\{L_j, j=1,\ldots,n\}$ are independent.
For each $j=1,\ldots,n$, the law of $(V_j,\e_j)$ is defined according to (\ref{ve}) with $\t=\t_j$.  And $\{L_j, j=1,\ldots,n\}$ is  a sequence  of
independent Bernoulli random variables. Set
 \begin{equation}\label{dec} S_n =\sum_{j=1}^n X_j, \qq  W_n =\sum_{j=1}^n V_j,\qq M_n=\sum_{j=1}^n  \e_jL_j,  \quad B_n=\sum_{j=1}^n
 \e_j .
\end{equation}

\begin{proposition}\label{lmd}We have\begin{eqnarray*} \{S_k, 1\le k\le n\}&\buildrel{\mathcal D}\over{ =}&  \{ W_k  +  DM_k, 1\le k\le n\} .
\end{eqnarray*}
And  $M_n\buildrel{\mathcal D}\over{ =}\sum_{j=1}^{B_n } L_j$.
 \end{proposition}

We also will need the following local limit theorem for Bernoulli sums.
\begin{proposition} \label{lltber}Let
$\mathcal B_n=\b_1+\ldots+\b_n$, $n=1,2,\ldots$  where
$
\b_i
$ are i.i.d.  Bernoulli r.v.'s (\,${\mathbb P}\{\b_i=0\}={\mathbb P}\{\b_i=1\}=1/2$).  There exists a numerical constant $C_0$ such that for all positive $n$
 \begin{eqnarray*}  \sup_{k}\, \Big|  {\mathbb P}\big\{\mathcal B_n=k\}
 -\sqrt{\frac{2}{\pi n}} e^{-{ (2k-n)^2\over
2 n}}\Big|\ \le \
  \  \frac{C_0}{n^{3/2}}  .
  \end{eqnarray*}
\end{proposition}

 \subsection{A concentration inequality for sums of independent random variables.}\label{sub2.2}
We also need  the following Lemma  (\cite{di}, Theorem 2.3).
\begin{lemma} \label{di.1}
Let $X_1, \dots, X_n$     be independent random variables, with $0 \le X_k \le 1$ for each $k$.
Let $S_n = \sum_{k=1}^n X_k$ and $\mu = \E S_n$. Then for any $\epsilon >0$,
 \begin{eqnarray*}
{\rm (a)}\quad  &&
 \P\big\{S_n \ge  (1+\epsilon)\mu\big\}
    \le  e^{- \frac{\epsilon^2\mu}{2(1+ \epsilon/3) } } ,
    \cr {\rm (b)} \quad & &\P\big\{S_n \le  (1-\epsilon)\mu\big\}\le    e^{- \frac{\epsilon^2\mu}{2}}.
 \end{eqnarray*}
\end{lemma}
\subsection{An estimate for quadratic forms}\label{sub2.3.qf}
 We refer to Weber  \cite{W1a}, inequality (2.19).
  \begin{lemma} \label{qfb}   For any system of complex numbers $\{x_i\}$ and  $ \{\a_{i,j}\}$,
\ben\label{qfba} \Big| \sum_{1\le i<j\le n}  x_ix_j\a_{i,j}\Big|  &\le&\frac1{2}
\sum_{i=1}^n |x_i|^2 \Big(\sum_{i<\ell\le n } |\a_{i,\ell}
|+\sum_{1\le  \ell <i }  |\a_{ \ell ,i} |\Big)  .
\een  
 Also, \ben\label{qfbb}   \Big|  \sum_{ 1\le i,j\le n\atop i\not= j}  x_ix_j\a_{i,j} \Big|\le \frac1{2}  \sum_{ i=1}^n |x_i|^2\Big(  \sum_{\ell
=1\atop
\ell\not = i }^n (|\a_{i,\ell} | +|\a_{
\ell ,i} |) \Big)  .
\een 
In particular, if $\a_{i,j}=\a_{j,i}$, then 
\ben\label{qfbaa} \Big| \sum_{1\le i,j\le n\atop i\not= j}  x_ix_j\a_{i,j}\Big|  &\le& 
\sum_{ i=1}^n |x_i|^2\Big(  \sum_{\ell
=1\atop
\ell\not = i }^n |\a_{i,\ell} |   \Big)  .
\een  
   \end{lemma}
 \begin{proof} We have
\ben\Big| \sum_{1\le i<j\le n}  x_ix_j\a_{i,j}\Big|&\le &  \sum_{1\le i<j\le n}  \big(\frac{|x_i|^2+|x_j|^2}{2} \big)|\a_{i,j} |
\cr &\le&\frac1{2}
\sum_{i=1}^n |x_i|^2 \Big(\sum_{i<\ell\le n } |\a_{i,\ell}
|+\sum_{1\le  \ell <i }  |\a_{ \ell ,i} |\Big)  .
\een Operating
similarly for the sum $\sum_{1\le j<i\le n}  x_ix_j\a_{i,j}$ gives
the  result.
  \end{proof}
 
 \vskip 3pt
 \subsection{Background on quasi-orthogonal systems}\label{sub2.3.qos}
We recall some classical facts on the notion of quasi-orthogonality in Hilbert spaces, of much relevance in the present work, and which are taken from our book \cite{W},   p.\,22. Let $(H,\|\cdot\|)$ be a real or complex Hilbert space, and let $\{f_n,   n\ge 1\}$
be orthonormal vectors  in
the inner product space $H$.   The fact that
 $
{1\over n}\sum_{i=1}^n f_i\, \to \, 0$  in $H$  follows for instance from Ky Fan's result. Recall  {Bessel's
inequality}: {\it If $\{e_i, 1\le i\le n\}$ are orthogonal vectors in the
inner product space $H$, then}
$$
\sum_{i=1}^n |\langle x,e_i\rangle |^2\le \|x\|^2\qq \textit{for any $x\in H$.} $$
There are many generalizations of this  inequality. 
In relation with Bessel's inequality,   
Bellman  introduced the
notion of a  {quasi-orthogonal system}  (see Kac, Salem and Zygmund \cite{KSZ}). 
A sequence  $  \underline{f}=\{f_n, n\ge 1\}$   in $H$  is called   a
quasi-orthogonal system  if the quadratic form   on $\ell_2$
defined by
 $\{x_h, h\ge 1\}\mapsto \|\sum_{h }  x_h   f_h\|^2 $  is
bounded. A necessary and sufficient condition for $\underline{f}$ to be
quasi-orthogonal  is that the series   $\sum c_nf_n$ converges in
$H$, for any sequence $\{c_n, n\ge 1\} $ such that  $\sum c_n^2
<\infty$.  

  As  observed in \cite{KSZ},    \lq\lq \textit{every theorem on
orthogonal systems  whose proof depends only on Bessel's inequality,
holds for quasi-orthogonal systems}\rq\rq.

  In particular for $H= L^2(  X,\mathcal A, \mu)$,  $(  X,\mathcal A, \mu)$ a probability space,
  Rademacher--Menchov's Theorem \ref{Ra.Me} applies,  and so   the series $\sum c_nf_n $
  converges almost everywhere, provided
that $\sum c_n^2 \log^2 n<\infty$. This is easily seen
from the fact that    $\underline{f}$ is quasi-orthogonal if and only if there
exists a constant $L$ depending on $\underline{f}$ only, such that
 \begin{equation}\label{quasi.o.caracterisation} \big\|\sum_{i\le n}
y_if_i\big\| \le L\, \big(\sum_{i\le n}|y_i|^2 \big)^{1/2}.
\end{equation}   There is a useful sufficient condition for quasi-orthogonality: {\it In order for $\underline{f}$ to be quasi-orthogonal, it is sufficient that}
  \begin{equation}\label{quasi.o.criterion}
  \sup_{i\ge 1}\Big(\sum_{  \ell\ge 1  } |\langle f_i, f_\ell \rangle  | \Big)  \,<\infty.
\end{equation}
This is  Lemma 7.4.4 in  Weber \cite{W.}, among other sources.   %  the classical weighted estimate for quadratic forms (Weber  \cite{W1a}, inequality (2.19)):
%{\it  For any system of complex numbers $\{x_i\}$ and  $ \{\a_{i,j}\}$},
% \begin{eqnarray}\label{wqe} \Big|  \sum_{ 1\le i,j\le n\atop i\not= j}  x_ix_j\a_{i,j} \Big|\le  \frac1{2}\sum_{ i=1}^n |x_i|^2\Big(  \sum_{\ell
%=1\atop
%\ell\not = i }^n (|\a_{i,\ell} | +|\a_{
%\ell ,i} |) \Big) ,
%\end{eqnarray}
  Indeed,    \eqref{qfbaa} in Lemma \ref{qfb} implies
\beq\Big| \sum_{1\le i,j\le n }  x_ix_j\a_{i,j}\Big|   \le  
\sum_{i=1}^n |x_i|^2 \Big(\sum_{1\le \ell\le n } |\a_{i,\ell}
| \Big)  .
\eeq 
Thereby,
\beq \label{bound.quasi.o}
\big\|\sum_{i\le n}
y_if_i\big\|^2 \le    
\sum_{i=1}^n |y_i|^2 \Big(\sum_{1\le \ell\le n } |\langle f_i, f_\ell \rangle  | \Big)
\le    
  \Big(\sup_{i\ge 1} \sum_{  \ell \ge 1 } |\langle f_i, f_\ell \rangle  | \Big)\cdot \sum_{i=1}^n |y_i|^2.\eeq 
Thus if
\eqref{quasi.o.criterion} is fulfilled,   \eqref{quasi.o.caracterisation} readily follows from the previous calculation.
 \subsection{Rademacher--Menshov's theorem}\label{sub2.3.}
  This   well-known result states as follows.  
\begin{theorem}\label{Ra.Me}
{\rm  a)} Let $\{a_k, k\ge 1\}$ be a  sequence of reals   such that
       $$\sum_{k\ge 1} | a_k  |^2_{2 } \log^2 k <
\infty.
$$
Then  for any  orthonormal sequence $\{\xi_k, k\ge 1\}$
in $L^2(\P)$,  $(\Omega, {\mathcal{ B}}, \P)$  a probability space, the
series $\sum_{k\ge 1}a_k\xi_k$ converges $\P$-almost surely.
\smallskip

{\rm b)} If $\{ w(n), n\ge 1\}$ is an arbitrary  positive monotone
increasing sequence of numbers with $w(n)=o(\log n)$, then there
exists an everywhere divergent orthonormal series $\sum_{k\ge
1}a_k\psi_k$ whose coefficients satisfy the condition
$$\sum_{k\ge 1}c_k^2w(k)^2=\infty.  $$
\end{theorem}
We refer to   Alexits  \cite{A}, (see p.\,80), Kashin and Saakyan \cite{KS}, Olevskii \cite{O} concerning orthogonal systems in Fourier analysis, and  Weber \cite{W}, \cite{W.}. Less known, and less applied, are the refinements of this result, obtained notably by Tandori, see \cite{KS}, and those obtained using the majorizing measure method, which are nevertheless of a self-evident practical interest in the study of limit theorems in Probability Theory. There are many   different proofs available. See also, more specifically Weber \cite{W4} and Ch.\,7 in   \cite{W.}, and references therein.
\subsection{A useful almost sure convergence criterion.}\label{sub2.3}For the proof of  the ASLLT (Theorem \ref{t1[asllt].}), the following almost sure convergence criterion  is of particular relevance.
This one has the advantage, comparatively to other known criteria, to allow  one to control   convergence of   series of random variables inherent to questions of this kind.
%, which  for instance  cannot be obtained by means  of the more used    G\'al-Koksma criterion.

 \begin{theorem}  \label{8.2} Let $\xi=\{\xi_l, l\ge 1\}$ be a sequence of real random variables with partial sums  $S_n=\sum_{l=1}^n \xi_l$, $n\ge 1$. Assume that    the following assumption is satisfied: 
 \vskip 2 pt For some $\g >1$, 
 \begin{equation}\label{8.12} \E \big|\sum_{i\le l\le j}\xi_l\big|^\g \le  \Phi(\sum_{  l\le j}
u_l)\Psi(\sum_{i\le l\le j} u_l)  \qq (\forall 1\le
i\le j<\infty),
\end{equation} where $\Phi ,  \Psi:\R^+\rightarrow\R^+$ are non-decreasing,   $\Psi(x)/x^{1+\eta}$ is
non-decreasing for some $\eta\ge 0$ and $u=  \{u_i,
i\ge 1\}$ is a sequence of non negative reals such
that  $U_j=\sum_{l=1}^j u_l\uparrow \infty$.   Further assume that  for some real $M>1$,   the series   
\begin{equation}\label{little.s}s^\g=s_{M}^\g =\sum_{l\ge 1\atop [M^l,M^{l+1}[\cap \mathcal U\neq\emptyset} \frac{\Psi(M^{ l  })
\,\Phi(M^{l })\,L(l)^{\e(\eta)}}{  M^{ \g l  }} 
\end{equation} converges, where   $\mathcal U= \{U_j, j\ge 1\}$, 
$I(l)=I_{M}(l)= [M^l,M^{l+1}[$, 
$L(l)=L_M(l)= 1+ \log \big( \sharp\{[M^l,M^{l+1}[\cap \mathcal U\}\big)$, 
 $l\ge 1$, and $\e(\eta)=0$ or $1$, according to  $\eta>0$ or  $\eta=0$. 
 \vskip 2pt Put,
\begin{equation}\label{big.s}{\mathcal S}^\g={\mathcal S}_{M}^\g=\sum_{k\ge 1\atop [M^k,M^{k+1}[\cap \mathcal U\neq\emptyset}
 \, \sup_{ 
M^k\le U_j<M^{k+1} }\Big({|S_j|\over U_j}\Big)^\g.
\end{equation}   Then,
  \begin{equation}\label{8.13} \|{\mathcal S}_M\|_\g \le  M  K_\g \,s ,\qq \hbox{and    in  particular }\quad\P \big\{\lim_{j\to\infty}\ {S_j\over U_j}= 0\big\}=1,
\end{equation}
where $K_\g $ is a constant depending on $\g$ only.
\vskip 3 pt 
%\color{blue}
{Let $1<M_1\le M$. Then
\beq \|{\mathcal  S}_{M_1}\|_\g^\g\,\le \, 2^{\g}\,\|{\mathcal  S}_{M}\|_\g^\g.
\eeq}
\end{theorem}
This is a slight variant of  the criterion given in Weber \cite[Th.\,8.2]{W1} generalizing G\'al-Koksma's, in which the reals $u_i$ are assumed to be integers.

\vskip 3 pt Several remarks are in order.
 %the first concerns the finiteness of $s_{M}^\g$ for different values of $M$.    

 % \begin{remark}[$\boldsymbol{ s_{M}^\g}$]\label{smg}. PlacÃ© aprÃšs \end{document}

 \begin{remark}
  The above Theorem provides with \eqref{8.13}, fine estimates of  ${|S_j|/   U_j}$.   Functions $\Phi, \Psi $  which are present in assumption \eqref{8.12},   appear   in the course of the proof where increment's assumptions are  combined with Lemma \ref{metrical.bounds} to control intermediate local maxima, but not    in \eqref{big.s}. 
  %which may surprise the reader   not familiar with criteria of that sort. 
  However assumption \eqref{little.s}, which is used to achieve the proof, completely relies  on these functions, and appear in the first part of \eqref{8.13}.

  \end{remark}

 \begin{remark}
    It is quite interesting to observe that the bound obtained {\it tightly depend} on  the way the sequence $\mathcal U$ is asymptotically  distributed, which is reflected by \eqref{little.s}. Indeed, only those $k$ such that $[M^k,M^{k+1}[\cap \mathcal U\neq\emptyset$ have to be  counted. This seems to have been overlooked, in particular by the author.  
This remark also applies to similar  criteria.  \end{remark}

\begin{remark}
An analog bound can   be derived from the proof for partially observed sequences $\{S_n, n\in \mathcal N\}$, $\mathcal N$ a growing sequence of integers.
\end{remark}

\begin{remark}[bounded case] \label{U.bounded}When $\mathcal U=  \{U_i,
i\ge 1\}$ is a bounded sequence, assumption \eqref{8.12} takes the simpler form: for some 
$\g >1$, 
\begin{equation}\label{8.12.U.bounded} \E \big|\sum_{i\le l\le j}\xi_l\big|^\g \le  C\Psi(\sum_{i\le l\le j} u_l)  \qq (\forall 1\le
i\le j<\infty).
\end{equation}
This readily implies that  $S_n=\sum_{l=1}^n \xi_l$, $n\ge 1$ is a Cauchy sequence in $L^\g(\P)$,  thus converging to some element $S\in L^\g(\P)$.
Further, assuming a little more than the convergence of the series $\sum_{l=1}^\infty m_l$, the series $\sum_{l=1}^\infty \xi_l$ also converges almost surely. More precisely we have the implication,
\begin{equation}\label{8.12.U.bounded.as.cv}
\sum_{l=1}^\infty u_l(\log l)^\g<\infty\  \Longrightarrow \  \hbox{the series
$\sum_{l=1}^\infty \xi_l$ converges almost surely.}
\end{equation} 
The case $\g =2$ contains Rademacher-Menshov's Theorem. See Remark 8.3.5 and (8.3.27) in Weber \cite{W}, and  Chapter 8 for detailed study  of these questions. See also Remark \ref{opt.midiv} where this is applied.

\end{remark}For the sake of completeness, we give a detailed proof of Theorem \ref{8.2}. \begin{proof}  
 Let
$\kappa=\kappa(M)= \{\kappa_p(M),p\ge 1\}$ be the sequence defined by $\kappa_p=\kappa_p(M)=
k$, if $I_k$ is the $p$-th interval  such that $I_k\cap \mathcal{ U}
\not= \emptyset$. Let $\mathcal{ L}_p=\mathcal{ L}_p(M)$ be the set of indices defined
by $L\in \mathcal{ L}_p\Leftrightarrow U_L\in I_{\kappa_p} $. Pick
arbitrarily some index in $\mathcal{ L}_p$, which we write $ L^*_p=L^*_p(M) $.
\vskip 2 pt
On the one hand
\beq\label{generic.seq}
\sum_p{\|S_{L^*_p}\|_\g^\g\over  M^{ \g(\kappa_p+1)  }  }  \le
\sum_p {\Psi(M^{ \kappa_p+1  }) \Phi(M^{ \kappa_p+1  })\over   M^{
\g(\kappa_p+1)  }  }\,.
 \eeq
  \vskip 2 pt
 And on the other,  for $i,j\in \mathcal{ L}_p$, $i\le j$, 
 \begin{equation}\label{generic.seq.}\E  {\big| S_j-S_i\big|^\g \over
\Phi(M^{ \kappa_{p} +1  }) }  \le  \frac{\Psi(U_j-U_i)}{(U_j-U_i)^{1+\eta}}(U_j-U_i)^{1+\eta}\,\le \,\frac{\Psi(M^{ \kappa_{p} +1  })}{(M^{ \kappa_{p} +1  })^{1+\eta}}\,\big(U_j-U_i\big)^{1+\eta}
%\cr & =  &\Psi(M^{ \kappa_{p} +1  })\,\Big(\frac{U_j  }{M^{ \kappa_{p} +1}  }-\frac{ U_i}{M^{ \kappa_{p} +1 } } \Big)^{1+\eta} 
.
 \end{equation}
 Thus
 \begin{eqnarray}\label{generic.seq..}\E   \big| S_j-S_i\big|^\g  & \le & \Phi(M^{ \kappa_{p} +1  })\Psi(M^{ \kappa_{p} +1  })\,\Big(\frac{U_j  }{M^{ \kappa_{p} +1}  }-\frac{ U_i}{M^{ \kappa_{p} +1 } } \Big)^{1+\eta} .
 \end{eqnarray}
 
We use the following Lemma based on metric entropy  chaining.
 \begin{lemma}[\cite{W1}, Lemma 3.4]\label{metrical.bounds}
Let $\g>1$, $0<\beta\le 1$ and consider a finite
collection of random variables $ \bigl( X_1,\dots , X_N \bigr)
\subset L^\g(\P)$,  and reals $0\le t_1\le t_2\le \dots \le t_N \le
1$ such that
 \begin{equation}\label{metrical.bounds1}
\|X_j-X_i\|_\g  \le (t_j -t_i)^{\beta }  \qq (\forall 1\le i\le j\le
N)
.
 \end{equation} 

Then, there exists a constant $K_{\beta,\g}$
depending on $\beta,\g$ only, such that
 \begin{equation}\label{metrical.bounds2}
\bigl\|  {\sup_{1\le i,j\le N}   |X_i -X_j |}  \bigr\|_\g
\,\le  \,
\begin{cases}
K_{\beta,\g}   &\text{if $\beta\g>1$,}
\\
K_{\beta,\g} \log  N&\text{if $\beta\g= 1$,}
\\
   K_{\beta,\g} N^{{1\over \g}-\beta}&\text{if $\beta\g<1$.}
\end{cases}
 \end{equation}
\end{lemma}

We deduce from \eqref{generic.seq..} and Lemma \ref{metrical.bounds},
\begin{equation}\label{appl.metrical.bounds}
\big\|  \sup_{i,j\in \mathcal{ L}_p}|S_i -S_j |  \big\|_\g  \le
\begin{cases}
 K_{\eta,\g} \Phi(M^{ \kappa_{p} +1  })^{1/\g}\Psi(M^{ \kappa_p+1  })^{1/\g}
&\text{if $\eta>0$,}
\\
K_{\eta,\g}\Phi(M^{ \kappa_{p} +1  })^{1/\g}\Psi(M^{ \kappa_p+1  })^{1/\g} \log \sharp(\mathcal{ L}_p)
&\text{if $\eta=0$,}
\end{cases} \end{equation}
where $K_{\eta,\g}$ depend on $ \eta,\g $ only.
 \vskip 3 pt Assume that $\eta=0$. Then
\beq\label{oscillation.gen.seq}
\sum_p{\big\|\sup_{i,j\in \mathcal{ L}_p}|S_i-S_j|\big\|_\g^\g\over
M^{ \g(\kappa_p+1)  }} \le K_{ \g} \sum_p \frac{
\Phi(M^{\kappa_p+1  }) \Psi(M^{\kappa_p+1}) (1+\log \sharp(\mathcal{ L}_p))}{  M^{ \g(\kappa_p+1)  }}  
.
\eeq
 From \eqref{generic.seq} and \eqref{oscillation.gen.seq}     and  the elementary inequality $(a+b)^y\le 2^{y-1}(a^y+b^y)$,   $a\ge 0,b\ge 0$ and $y\ge 1$ (which is a plain application of H\"older's inequality),  we deduce that,  
 \beq\label{oscillation.gen.seq.}
\sum_p{\big\|\sup_{j\in \mathcal{ L}_p}| S_j|\big\|_\g^\g\over
M^{ \g(\kappa_p+1)  }} \le K_{ \g} \sum_p \frac{ \Phi(M^{
\kappa_p+1  }) \Psi(M^{\kappa_p+1}) (1+\log \sharp(\mathcal{ L}_p))}{  M^{ \g(\kappa_p+1)  }} .
\eeq

Observing for $L\in \mathcal{ L}_p$ that $ M^{\kappa_p+1}\le M U_L$,   we obtain,  in view of the definition of $\mathcal{ L}_p$, the following bound,
\begin{eqnarray*}\|{\mathcal  S}_{M}\|_\g^\g &=&
 \sum_{k\ge 1\atop [M^k,M^{k+1}[\cap \mathcal U\neq\emptyset}\Big\|\, \sup_{ 
M^k\le U_j<M^{k+1} }\Big({|S_j|\over U_j}\Big)\Big\|_\g^\g
\cr &= &\sum_{p}\Big\|\, \sup_{ j\in \mathcal{ L}_p }\Big({|S_j|\over U_j}\Big)\Big\|_\g^\g\,\le\,M\,\sum_p{\big\|\sup_{j\in \mathcal{ L}_p}| S_j|\big\|_\g^\g\over
M^{ \g(\kappa_p+1)  }} 
\cr &\le &M\,K_{ \g} \sum_p{ \Phi(M^{
\kappa_p+1  }) \Psi(M^{\kappa_p+1}) }(1+\log \sharp(\mathcal{ L}_p))/ {  M^{ \g(\kappa_p+1)  }}\cr &\le &
  M  K_{ \g} \,s^\g.\end{eqnarray*}
 Further
  ${\sup_{j\in
\mathcal{ L}_p}  |S_j|/ M^{\kappa_p} }$ tends to $0$ almost surely, as
$p$ tends  to infinity. And so \eqref{8.13} follows. The case $\eta>0$ is identical.
\vskip 3 pt
 Let $1<M_1\le M$. For any integer $k\ge 1$, there exists an integer $l\ge 1$ such that   $[M_1^{k},M_1^{k+1}[\subset [M^l,   M^{l+2}[$. Moreover if $[M_1^k,M_1^{k+1}[\cap \mathcal U\neq\emptyset$, then   $[M^l, M^{l+1}[\cap \mathcal U\neq\emptyset$ and/or  $[M^{l+1}, M^{l+2}[\cap \mathcal U\neq\emptyset$.
 
  From the following inequality,
  \ben\Big\|\, \sup_{ 
M_1^{k}\le U_j<M_1^{k+1} }\Big({|S_j|\over U_j}\Big)\Big\|_\g^\g&
\le & \bigg\{\Big\|\, \sup_{ 
M^l\le U_j<M^{l+1} }\Big({|S_j|\over U_j}\Big)\Big\|_\g\, + \Big\|\, \sup_{ 
M^{l+1}\le U_j<M^{l+2} }\Big({|S_j|\over U_j}\Big)\Big\|_\g\bigg\}^\g\
\cr &
\le &2^{\g-1} \bigg\{\Big\|\, \sup_{ 
M^l\le U_j<M^{l+1} }\Big({|S_j|\over U_j}\Big)\Big\|_\g^\g  + \Big\|\, \sup_{ 
M^{l+1}\le U_j<M^{l+2} }\Big({|S_j|\over U_j}\Big)\Big\|_\g^\g\bigg\},\een
and   definition of ${\mathcal  S}_{M}$, it follows that 
\beq \|{\mathcal  S}_{M_1}\|_\g^\g\,\le \, 2^{\g}\,\|{\mathcal  S}_{M}\|_\g^\g.
\eeq
This implies by the same argument used before that   ${\sup_{j\in
\mathcal{ L}_p}  |S_j|/ M_1^{\kappa'_p} }$ tends to $0$ almost surely, as
$p$ tends  to infinity, $\kappa'_p=\kappa_p(M_1)$ being defined with respect to the sequence $\{I_{M_1}(l), l\ge 1\}$ and $\mathcal U$. This achieves the proof. \end{proof}

%%%%%%%%%%%%%%%%%
%%%%%%%%%%%%%%%%%

\vskip 25 pt

 \section{\gsec  PROOFS  OF THEOREM \ref{t2[asllt]} AND OF COROLLARY   \ref{cort1} }\label{s3}
\vskip 5 pt 
%%%%%%%%%%%%%%%%%
 
Assumption \eqref{nun1} implies that condition  \eqref{basber.j}  is realized.
 Let $0<\t_j\le\t_{X_j}$.  By Proposition \ref{lmd}, we
may
 associate to $ \{X_k, k\ge 1\}$ a
sequence $\{(V_j,\e_j, L_j), j\ge 1\}$ of independent copies of  $(V ,\e , L )$  such that
 $$ \{V_j+\e_jD L_j, j\ge 1\}\buildrel{\mathcal D}\over{ =}\{X_j, j\ge 1\}  . $$
Further $\{(V_j,\e_j), j\ge 1\}
 $   and $\{L_j, j\ge 1\}$ are independent sequences.
 And $\{L_j, j\ge 1\}$ is  a sequence  of independent Bernoulli random variables. Set
\begin{equation}\label{dec1} S_n =\sum_{j=1}^n X_j, \qq  W_n =\sum_{j=1}^n V_j,\qq M_n=\sum_{j=1}^n \e_jL_j,  \qq B_n=\sum_{j=1}^n \e_j .
\end{equation}
 We   notice that $M_n$ is   a sum of exactly $B_n$ Bernoulli random variables.
  We have the
representation
$$ \{S_n, n\ge 1\}\buildrel{\mathcal D}\over{ =}  \{ W_n  + DM_n, n\ge 1\} .$$
And  $M_n\buildrel{\mathcal D}\over{ =}\sum_{j=1}^{B_n } L_j$. 
    
      We denote by $\E_{(V,\e)}$, ${\mathbb P}_{(V,\e)}$ (resp.\ $\E_{\!L}$, ${\mathbb P}_{\!L}$) the expectation and probability symbols
relatively to the
$\s$-algebra generated by the sequence   $\{(V_j,\e_j), j=1, \ldots, n\}$ (resp.\ $\{L_j , j=1, \ldots, n\}$). These algebra are independent.
\bigskip \par Put \begin{equation}\label{Y}Y_n= \s_n \big({\bf 1}_{\{S_n=\kappa_n\}}-{\mathbb P}\{S_n=\kappa_n\} \big)  .
\end{equation}
Let $n>m\ge 1$.   We have a simple expression for the correlation 
\begin{equation}\label{basic.0}  \s_m\s_n\Big({\mathbb P} \{S_m =\kappa_m\,,\,S_{n } =\kappa_n\}-{\mathbb P}\{ S_{ m} = \kappa_m\}  {\mathbb P} \{S_n
=\kappa_n
\}\Big)= {\mathbb E\,} Y_nY_m, 
\end{equation} 
which we can     more conveniently rewrite as follows,
\begin{equation}\label{basic}{\mathbb E\,} Y_nY_m=\s_m{\mathbb P} \{S_m =\kappa_m\}\,\s_n\Big({\mathbb P}\{S_n-S_m =\kappa_n -\kappa_m\}- {\mathbb P} \{S_n
=\kappa_n
\}\Big). 
\end{equation}
  
 Since  by \eqref{nun.ka}-(2), 
 \begin{equation*}   (\s_j^2-\s_i^2)^{1/2}  \,{\mathbb P}\{S_j-S_i=\kappa_j-\kappa_i\}  ={\mathcal O}(1),
 %\s_m \, {\mathbb P}\{S_m=\kappa_m\}={\mathcal O}(1),
 \end{equation*}
we have the bound 
\begin{align*}     &   \s_m {\mathbb P}\{S_m=\kappa_m\}   
\ \s_n \Big|{\mathbb P}\{S_n-S_m=\kappa_{n}-\kappa_m \}- {\mathbb P}\{S_n=\kappa_n\}\Big| \\
&\le  C \, \Big(\frac {\s_n}{(\s_n^2-\s_m^2)^{1/2} }\Big)  (\s_n^2-\s_m^2)^{1/2} \,{\mathbb P}\{S_n-S_m=\kappa_{n}-\kappa_m \}+\s_n {\mathbb P}\{S_n=\kappa_n\}\Big)
\le  C \Big(\frac {\s_n}{(\s_n^2-\s_m^2)^{1/2} }+1\Big).
\end{align*}

Whence for $n>m\ge 1$,
\begin{align}\label{basic.corr.bd}  |{\mathbb E\,} Y_nY_m|     \le  C\, \Big(\frac {\s_n}{(\s_n^2-\s_m^2)^{1/2}}+1\Big).
\end{align}

 % $$====$$(i) We write
%\begin{align*}  &\s_m\s_n
%\Big|{\mathbb P}\{S_m=\kappa_m\,,\,S_n=\kappa_n \}- {\mathbb P}\{S_m=\kappa_m\}{\mathbb P}\{S_n=\kappa_n\}\Big|\\ &= \Big\{\s_m {\mathbb P}\{S_m=\kappa_m\}\Big\} \cdot
%\Big\{\s_n \Big|{\mathbb P}\{S_{n-m}=\kappa_{n}-\kappa_m \}- {\mathbb P}\{S_n=\kappa_n\}\Big|\Big\}\\
%&\le  C\frac {\s_n}{\s_{n-m}} \cdot \s_{n-m} \Big|{\mathbb P}\{S_{n-m}=\kappa_{n}-\kappa_m \}- {\mathbb P}\{S_n=\kappa_n\}\Big|
%\le  C \Big(\frac {\s_n}{\s_{n-m}}+1\Big)
%\end{align*}
%by \eqref{qq}.

%$$=====$$

We therefore have to estimate $|A|$ where 
 \begin{eqnarray}\label{llt.iid.appl.}
A &:=& \s_n\Big({\mathbb P}\{ S_{n }-S_{ m} =\kappa_n -\kappa_m\}- {\mathbb P} \{S_n =\kappa_n \}\Big)
\cr
& =& \s_n{\mathbb E\,} \big({\bf 1}_{\{B_n\le \theta\nu_n \}}+{\bf 1}_{\{B_n>\theta\nu_n\}}\big)\Big({\bf 1}_{\{ S_{n }-S_{ m} =\kappa_n -\kappa_m\}}-
{\bf 1}_{
\{S_n =\kappa_n \}}\Big).
\end{eqnarray}
Further,
 \begin{equation}\label{basic0}{\mathbb E\,} Y_n^2=\s_n^2 \,{\mathbb P}\{S_n=\kappa_n\}\big(1-{\mathbb P}\{S_n=\kappa_n\} \big) ={\mathcal O}(\s_n).
\end{equation}
 \vskip 5pt
 Let $0<\e<1$ and set $\rho=e^{- \frac{\epsilon^2 }{2}}$, $\theta=1-\e$. By Lemma \ref{di.1},
 \begin{equation}\label{nun1.}\P\{B_n\le \theta\nu_n\big\}\le \rho^{ \nu_n}\ ,\   \P\{B_n-B_m\le \theta(\nu_n-\nu_m)\big\}\le   \rho^{ \nu_n-\nu_m}
,\end{equation}
for all $n\ge m\ge 1$. 

\vskip 3 pt  We can write in view of Proposition \ref{lmd} 
\begin{align}\label{dep} &\s_n\,{\mathbb E\,} 
 {\bf 1}_{\{B_n>\theta\nu_n\}} \Big({\bf 1}_{\{ S_{n }-S_{ m} =\kappa_n
-\kappa_m\}}- {\bf 1}_{
\{S_n =\kappa_n \}}\Big)\cr&=  \s_n\,\E_{(V,\e)} {\bf 1}_{\{B_n>\theta\nu_n\}} \Big({\mathbb P}_{\!L}
\Big\{D  \sum_{j=m+1}^n \e_jL_j =\kappa_n-\kappa_m-(W_{n  }-W_m)\Big\}
\cr &  \  -  {\mathbb P}_{\!L}
\big\{D \sum_{j= 1}^n \e_jL_j  =\kappa_n-W_n
\big\}\Big) .
\end{align}  

To treat the independent case, it is necessary to introduce the sets 
$$A(n,m)= \{ {\color{black}\e_j}=0,\ m<j\le n\}.$$ On $A(n,m)$ we have $ \sum_{j=m+1}^n \e_jL_j =0$. Thus
$$\Big\{D  \sum_{j=m+1}^n \e_jL_j =\kappa_n-\kappa_m-(W_{n  }-W_m)\Big\}=\Big\{W_{n
}-W_m =\kappa_n-\kappa_m \Big\}.$$
So that (\ref{dep}) may be continued with
\begin{eqnarray}\label{dep1}  &= & \s_n \Big\{\E_{(V,\e)} {\bf 1}_{\{(B_n>\theta\nu_n)\cap A(n,m)\}}\big({\bf 1}_{\{  W_{n  }-W_m=\kappa_n-\kappa_m   \} }
\cr & &\    - {\mathbb P}_{\!L}
\big\{D \sum_{j= 1}^n \e_jL_j  =\kappa_n-W_n
\big\}\big) \Big\}
\cr
&  &+\s_n\Big\{\E_{(V,\e)} {\bf 1}_{\{B_n>\theta\nu_n)\cap A(n,m)^c\}} \Big[{\mathbb P}_{\!L}
\Big\{D  \sum_{j=m+1}^n \e_jL_j =\kappa_n-\kappa_m-
\cr & &\ (W_{n  }-W_m)\Big\}
   -  {\mathbb P}_{\!L}
\big\{D \sum_{j= 1}^n \e_jL_j  =\kappa_n-W_n
\big\}\Big] \Big\}
\cr &:=& A'+ A''.
\end{eqnarray}

   We have the easy bound 
\begin{equation}\label{e1}|A'| \le \s_n  {\mathbb P}
\big\{ A(n,m)\}  =\s_n\, {\color{black} \prod_{j=m+1}^n\vartheta_j  }
%\le \big( \frac{\s_n}{\sqrt {\nu_n}}\big)\, \sqrt {\nu_n}\,{\color{red} \prod_{j=m+1}^n\vartheta_j  }
. \end{equation}
%since $0<\t_j\le\t_{X_j}\le 1$, and so $\nu_n-\nu_m\le n-m$. 
Concerning  $A''$, we note that 
$$\sum_{j= 1}^n \e_jL_j\buildrel{\mathcal D}\over{ =}\sum_{j=1}^{B_n } L_j,\qq \sum_{j=m+1}^n \e_jL_j
 \buildrel{\mathcal D}\over{ =}\sum_{j=B_m+1}^{B_n } L_j.$$
   By applying Proposition \ref{lltber}   we obtain,
 \begin{eqnarray}\label{lltber1} \sup_{z}\, \Big|\sqrt N\, {\mathbb P}\big\{\sum_{j=1}^{N } L_j=z\} \big\} -{2\over \sqrt{2\pi}}e^{-{(z-(N/2))^2\over (N/2)}}\Big| =o\Big({1\over N}\Big).
 \end{eqnarray} 
 It follows that 
$$ \bigg|   {\mathbb P}_{\!L} \Big\{D\sum_{j=1}^{B_n } L_j =\kappa_n-W_n \Big\} -{2e^{-{(\kappa_n-W_n-(B_n/2))^2\over D^2(B_n/2)}}\over
\sqrt{2\pi B_n}} \bigg| =o\Big({1\over B_n^{3/2}}\Big). $$
Further since on the set $A(n,m)^c$, {\color{black} $\e_j =1$ for some $m<j\le n$},  we have   $ B_n>B_m $, 
 %[hint: $\e_jL_j=1$ for some $m<j\le n$ and thus $\e_j =1$ for some $m<j\le n$], 
 then
\begin{align*} & \bigg|   \, {\mathbb P}_{\!L} \Big\{D\sum_{j=1}^{B_{n }-B_{ m} } L_j =\kappa_n-\kappa_m-(W_{n  }-W_m) \Big\}
 -{2e^{-{(\kappa_n-\kappa_m-(W_{n  }-W_m)-(B_{n }-B_{ m})/2 )^2\over D^2(B_{n }-B_{ m})/2 }}\over
\sqrt{2\pi (B_{n }-B_{ m})} } \bigg|
\cr& = o\Big({1\over
(B_{n }-B_{ m})^{3/2}}\Big). 
\end{align*}
%%%%%%%%%%%%%%%%%%%%%%%%%%%%%%%%%%%%%%%%%%%%%%%%%%%%%%%%%%%%%%%%%%%%%%%%%%%%%%%%%%%%%%%%%%%%%%%%%%%%%%%%%%%%%%%%%%%%%%%%%%%%%%
 Therefore,
 \begin{eqnarray}\label{e2} |A''| &\le & \s_n\, \bigg|\E_{(V,\e)}{\bf 1}_{\{(B_n>\theta\nu_n)\cap A(n,m)^c\}}
\Big\{{2e^{-{(\kappa_n-W_n-(B_n/2))^2\over D^2(B_n/2)}}\over
\sqrt{2\pi B_n}}
\cr
& &\quad- {2e^{-{(\kappa_n-\kappa_m-(W_{n  }-W_m)-(B_{n }-B_{ m})/2 )^2\over D^2(B_{n }-B_{ m})/2 }}\over
\sqrt{2\pi (B_{n }-B_{ m})} }\Big\}\bigg|
\cr
& &\quad+C_0\,\s_n\, \E_{(V,\e)}{\bf 1}_{\{B_n>\theta\nu_n, B_n>B_m\}}\Big\{ {1\over B_{n }^{3/2}}+ \Big({1\over
(B_{n }-B_{ m})^{3/2}}\Big)\Big\}\cr
&:=& A''_1+C_0A''_2,
\end{eqnarray}
 where the constant $C_0$ comes from 
 %the Landau symbol $o$ in 
  \eqref{lltber1}. The second term is easily estimated. Indeed,
\begin{eqnarray}\label{e6}A''_2&= &\s_n \,\E_{(V,\e)}{\bf 1}_{\{B_n>\theta\nu_n, B_n>B_m\}}
 \Big( {1\over B_{n }^{3/2}}+{1\over (B_{n }-B_{ m})^{3/2}}\Big)
\cr &\le&
 2\s_n\, {\mathbb P}\{B_n-B_m\le \theta(\nu_n-\nu_m)\}\cr
& & + \s_n\, \E_{(V,\e)}{\bf 1}_{\{ {\color{black}B_n>\theta\nu_n} ,\, B_n-B_m>  \theta(\nu_n-\nu_m)\}} \Big( {1\over {\color{black}B_n^{3/2}}}+{1\over (B_{n }-B_{ m})^{3/2}}\Big)
\cr &\le&
 C\, \s_n \Big\{ \rho^{ \nu_n-\nu_m }  +  {1\over (\theta\nu_n)^{3/2}}+ {1\over
(\theta(\nu_n-\nu_m))^{3/2}}\Big\}.
\end{eqnarray}

We now   estimate $A''_1$, which we    bound   as follows:
%$$A''_1=\s_n\, \bigg|\E_{(V,\e)}{\bf 1}_{\{(B_n>\theta\nu_n)\cap A(n,m)^c\}}
%\Big\{{2e^{-{(\kappa_n-W_n-(B_n/2))^2\over D^2(B_n/2)}}\over
%\sqrt{2\pi B_n}}$$$$ - {2e^{-{(\kappa_n-\kappa_m-(W_{n  }-W_m)-(B_{n }-B_{ m})/2 )^2\over D^2(B_{n }-B_{ m})/2 }}\over\sqrt{2\pi (B_{n }-B_{ m})} }\Big\}\bigg|$$
 \begin{eqnarray}\label{e3} A_1''&=&\s_n\, \bigg|\E_{(V,\e)}{\bf 1}_{\{(B_n>\theta\nu_n)\cap A(n,m)^c\}}
\Big\{{2e^{-{(\kappa_n-W_n-(B_n/2))^2\over D^2(B_n/2)}}\over
\sqrt{2\pi B_n}}
\cr & & \qq\qq\qq\qq- {2e^{-{(\kappa_n-\kappa_m-(W_{n  }-W_m)-(B_{n }-B_{ m})/2 )^2\over D^2(B_{n }-B_{ m})/2 }}\over
\sqrt{2\pi (B_{n }-B_{ m})} }\Big\}\bigg|
\cr &
  \le   & C  \ \E_{(V,\e)}{\bf 1}_{\{(B_n>\theta\nu_n)\cap A(n,m)^c\}} \Big\{\Big({\s_n \over \sqrt{B_n}}\Big) \Big[
  \sqrt{  {B_n} \over
   B_{n }-B_{ m}   }-1 \Big]
\cr & & \ \qq\qq   \times e^{-{(\kappa_n-\kappa_m-(W_{n  }-W_m)-(B_{n }-B_{ m})/2 )^2\over D^2(B_{n }-B_{ m})/2 }}\Big\}
\cr &  &\ + C  \ \E_{(V,\e)}{\bf 1}_{\{(B_n>\theta\nu_n)\cap A(n,m)^c\}} \Big({\s_n \over \sqrt{B_n}}\Big) 
\cr & & \qq\qq  \times \Big|   e^{-{(\kappa_n-W_n-(B_n/2))^2\over
D^2(B_n/2)}}
-
   e^{-{(\kappa_n-\kappa_m-(W_{n  }-W_m)-(B_{n }-B_{ m})/2 )^2\over D^2(B_{n }-B_{ m})/2 }}\Big|
 \cr &\le &  C_\theta \Big({\s_n \over\sqrt{\nu_n}}\Big)\,\Bigg\{  \ \E_{(V,\e)}{\bf 1}_{\{B_n>\theta\nu_n, B_n>B_m\}}   \Big[
  \sqrt{  {B_n} \over
   B_{n }-B_{ m}   }-1 \Big]
 \cr & & \ +   \ \E_{(V,\e)}{\bf 1}_{\{B_n>\theta\nu_n, B_n>B_m\}}  
 \cr & &\qq \qq \times \Big|   e^{-{(\kappa_n-W_n-(B_n/2))^2\over
D^2(B_n/2)}}
  -
   e^{-{(\kappa_n-\kappa_m-(W_{n  }-W_m)-(B_{n }-B_{ m})/2 )^2\over D^2(B_{n }-B_{ m})/2 }}\Big| \ \Bigg\}
\cr&:=&\, C_\theta\Big({\s_n \over \sqrt{\nu_n}}\Big)\,\big\{A''_{11}+A''_{12}  \big\}.
\end{eqnarray}

%%%%%%%%%%%%%%%%%%%%%%%%%%%%%%%%%%
%%%%%%%%%%%%%%%%%%%%%%%%%%%%%%%%%%%
 As $\nu_n\uparrow \infty$ with $n$, it follows from Kolmogorov's law of the iterated logarithm that
\begin{equation}\label{thetaj.LIL}
 \limsup_{n\to \infty} \frac{B_n -\nu_n}{\sqrt{\nu_n\log \log \nu_n}} \buildrel{{\rm a.s.}}\over{=} 1.
 \end{equation}
 Thus  
 $$ B_*= \sup_{n\ge 1}\frac{ B_n }{\nu_n}<\infty,$$
almost surely. Moreover $\|B_*\|_p<\infty$ for $1<p<\infty$.
  
Consider $A''_{11}$.  
  Using the inequality $\sqrt x-\sqrt y\le \sqrt{x-y}$ if $x\ge y\ge 0$, we have
 on the set $\{ B_n>B_m,B_n-B_m>\theta(\nu_n-\nu_m)\}$,
 \begin{eqnarray*}   \sqrt{  {B_n} \over
   B_{n }-B_{ m}   }-1 
   &=& 
    {\sqrt{B_n} -\sqrt{B_{n }-B_{ m}}\over \sqrt{B_{n }-B_{ m}}}
\,\le\, {\sqrt{B_m}\over \sqrt{B_{n }-B_{ m}}}\,\le\, B_*^{1/2}\,{\sqrt{ \nu_m}\over \sqrt{\theta(\nu_n-\nu_m)}}
\cr &=&B_*^{1/2}\,{1\over \sqrt \theta \sqrt{\frac{\nu_n}{\nu_m}-1}}\,\le \,B_*^{1/2}\, {1\over \sqrt \theta (\sqrt{{\nu_n\over \nu_m}}-1) }.   \end{eqnarray*}
Thus
 \begin{align}\label{e4}  A''_{11} 
 & =  \E_{(V,\e)}\Big({\bf 1}_{\big\{  \frac{B_n-B_m}{\theta(\nu_n-\nu_m)}> 1\big\}}
 \ +{\bf 1}_{\big\{  \frac{B_n-B_m}{\theta(\nu_n-\nu_m)}\le 1\big\}}\Big){\bf 1}_{\{B_n>\theta\nu_n, B_n>B_m\}} \Big[
  \sqrt{  {B_n} \over
   B_{n }-B_{ m}   }-1 \Big]
\cr &\le  \, { \E_{(V,\e)} B_*^{1/2}\over \sqrt \theta (\sqrt{{\nu_n\over \nu_m}}-1) }+\E_{(V,\e)} {\bf 1}_{\big\{  \frac{B_n-B_m}{\theta(\nu_n-\nu_m)}\le 1\big\}}
 {\bf 1}_{\{  B_n>B_m\}} \Big[{\sqrt{B_m}\over \sqrt{B_{n }-B_{ m}}}\Big]
\cr &\le   { \E_{(V,\e)} B_*^{1/2}\over \sqrt \theta (\sqrt{{\nu_n\over \nu_m}}-1) }+\sqrt{\nu_m} \,\E_{(V,\e)}B_*^{1/2} {\bf 1}_{\big\{  \frac{B_n-B_m}{\theta(\nu_n-\nu_m)}\le 1\big\}}.
\end{align}
By H\"older's inequality, for $\a>1$,
 \begin{eqnarray*}\E_{(V,\e)}B_*^{1/2} {\bf 1}_{\big\{  \frac{B_n-B_m}{\theta(\nu_n-\nu_m)}\le 1\big\}}&\le&\big( \E_{(V,\e)}(B_*)^{\a/2}\big)^{1/\a}\P^{1-1/\a}\Big\{  \frac{B_n-B_m}{\theta(\nu_n-\nu_m)}\le 1\Big\}
 \cr &\le&\big( \E_{(V,\e)}(B_*)^{\a/2}\big)^{1/\a}\rho^{(1-1/\a)(\nu_n-\nu_m)}.
  \end{eqnarray*}
Therefore
\begin{align}\label{Asecond.11}A''_{11}  &\le \,   { \E_{(V,\e)} B_*^{1/2}\over \sqrt \theta (\sqrt{{\nu_n\over \nu_m}}-1) }+\|B_*^{1/2}\|_\a\,\sqrt{\nu_m} \,\rho^{(1-1/\a)(\nu_n-\nu_m)}.\end{align}

  \vskip 5 pt

 We now turn to $A''_{12}$. Put 
 $$\kappa_n'=\kappa_n-W_n-(DB_n/2).$$
Then
\begin{eqnarray*}A''_{12}
&=& \E_{(V,\e)}{\bf 1}_{\{B_n>\theta\nu_n,  B_n>B_m\}} \Big|   e^{-{(\kappa_n-W_n-(B_n/2))^2\over
D^2(B_n/2)}}
\cr & &  \qq\qq\qq\qq\qq\qq-
   e^{-{(\kappa_n-\kappa_m-(W_{n  }-W_m)-(B_{n }-B_{ m})/2 )^2\over D^2(B_{n }-B_{ m})/2 }}\Big|
\cr
 &=&\E_{(V,\e)}{\bf 1}_{\{B_n>\theta\nu_n,  B_n>B_m\}} \Big|   e^{-{{\kappa'_n}^2\over
D^2(B_n/2)}}
  -
   e^{-{(\kappa'_n-\kappa'_m )^2\over D^2(B_{n }-B_{ m})/2 }}\Big| .
\end{eqnarray*}
At first,
  \begin{equation}\label{e5}  \E_{(V,\e)}{\bf 1}_{\{B_n>\theta\nu_n,  0<B_n-B_m\le \theta (\nu_n-\nu_m)\}} \Big|   e^{-{{\kappa'_n} ^2\over
D^2(B_n/2)}}
  -
   e^{-{(\kappa'_n-\kappa'_m )^2\over D^2(B_{n }-B_{ m})/2 }}\Big|
    \,\le \,2\rho^{ \nu_n-\nu_m}.
  \end{equation}
Concerning the integration over  $\{B_n>\theta\nu_n,   B_n-B_m> \theta (\nu_n-\nu_m)\}$, 
 we   get by letting
 $$b_n ={ \kappa_n' \over
 \sqrt B_n }     ,$$ and using the inequality $|e^{-u }-e^{-v }|\le
|u-v|$    for   reals $u\ge 0,v\ge 0$,  
  \begin{eqnarray}
 & &{   D^2\over 2}\, \Big|   e^{-{{\kappa'_n}^2\over
D^2(B_n/2)}}
  -
   e^{-{(\kappa'_n-\kappa'_m )^2\over D^2(B_{n}-B_{m})/2 }}\Big|
\cr & \le &  \Big|
-{(\kappa'_n-\kappa'_m )^2\over  (B_{n }-B_{ m})  }
 + {{\kappa'_n}^2\over
 B_n } \Big|= \Big|  -{ (\sqrt{
B_{n}}b_n-
\sqrt{ B_{m}}b_m)^2  \over B_{n}-B_{m} } + b_n^2 \Big|\cr
& = &  \Big|   {  -B_{n}b_n^2-B_{m}b_m^2+2\sqrt{B_{n}B_{m}}b_nb_m  +B_{n}b_n^2-B_{m}b_n^2\over B_{n}-B_{m} }  \Big|\cr
%& = & aa \Big|   {  -{ B_{n}\over B_{m}}b_n^2- b_m^2+2\sqrt{{ B_{n}\over B_{m}}}b_nb_m  +{ B_{n}\over B_{m}}b_n^2- b_n^2\over{ B_{n}\over B_{m}}-1}  \Big|\cr
  & = &
 \Big|   { -(b_n-b_m)^2+
2b_mb_n(\sqrt{B_{n}\over B_{m}}-1)\over {B_{n}\over B_{m}}-1 }  \Big|
  \cr &\le &    { 2(b_n^2+b_m^2)+ 2|b_m||b_n|(\sqrt{B_{n}\over B_{m}}-1)\over {B_{n}\over B_{m}}-1 }
           . \end{eqnarray}

Hence
\begin{eqnarray}  & &\E_{(V,\e)}{\bf 1}_{\{B_n>\theta\nu_n,   B_n-B_m> \theta (\nu_n-\nu_m)\}} \Big|   e^{-{{\kappa'_n}^2\over
D^2(B_n/2)}}
  -
   e^{-{(\kappa'_n-\kappa'_m )^2\over D^2(B_{n }-B_{ m})/2 }}\Big|
\cr & \le &C\, \E_{(V,\e)}{\bf 1}_{\{B_n>\theta\nu_n,   B_n-B_m> \theta (\nu_n-\nu_m)\}} \bigg\{ {  b_n^2+b_m^2  \over {B_{n}\over B_{m}}-1 }    +{
 |b_m||b_n|  \over \sqrt{B_{n}\over B_{m}}-1  }    \bigg\}.
 \end{eqnarray}

One next establishes the following three estimates:
 \begin{eqnarray}\label{three.est} \E_{(V,\e)} {\bf 1}_{\{B_n>\theta\nu_n,   B_n-B_m> \theta (\nu_n-\nu_m)\}} { |b_n||b_m| \over
\sqrt{B_{n}\over B_{m}}-1   } &\le &   \Big(\frac{\s_n\s_m}{\sqrt{\nu_n\nu_m}}\Big)\, {   C_\theta\over
 \sqrt{\nu_n\over \nu_m}-1   } 
  \cr \E_{(V,\e)} {\bf 1}_{\{B_n>\theta\nu_n,   B_n-B_m> \theta (\nu_n-\nu_m)\}} {   |b_n| ^2\over
 {B_{n}\over B_{m}}-1   } &\le&  \Big(\frac{\s_n^2}{\nu_n}\Big)   { C_\theta\over
    \sqrt {\nu_n\over \nu_m}-1   }
    \cr \E_{(V,\e)} {\bf 1}_{\{B_n>\theta\nu_n,   B_n-B_m> \theta (\nu_n-\nu_m)\}} {   |b_m| ^2\over
 {B_{n}\over B_{m}}-1   } &
\le &  \Big(\frac{\s_m^2}{\nu_m}\Big) { C_\theta\over
    \sqrt {\nu_n\over \nu_m}-1   } .
\end{eqnarray}

 Let $S'_m=W_m+(DB_m/2)$,  we note that 
  $S'_m-a_m= S'_m-\E_{(V,\e)}S'_m$, since   ${\mathbb E\,} S_m
 =
\E_{(V,\e)}S'_m=a_m $.
 Besides by \eqref{nun.ka}-(1),
  $|\kappa_j-a_j|=\mathcal O(\s_j)$; thus 
\begin{equation}\label{boundbj} |b_j|= {\big|\kappa_j-a_j -\big(W_j+(DB_j/2)-a_j\big)\big| \over
 \sqrt  {B_j} }
  \le  {C\over
 \sqrt  {B_j} }  \Big[ \s_j +\big|S'_j-\E_{(V,\e)} S'_j \big| \Big]  .
\end{equation}
  We note that
$ {     1\over
\sqrt{B_{n}\over B_{m}}-1   } \le  {    \sqrt{B_{m}}(\sqrt{B_{n} }+\sqrt{  B_{m}})\over
 \theta(\nu_n-\nu_m)   } $, on the set $ \{B_n>\theta\nu_n,   B_n-B_m> \theta (\nu_n-\nu_m)\} $.
Thus
 \begin{eqnarray*} && { |b_n||b_m| \over
\sqrt{B_{n}\over B_{m}}-1   } \,\le\,   {C\over
 \sqrt  {B_nB_m} } {
\sqrt{B_{m}}(\sqrt{B_{n} }+\sqrt{  B_{m}})\over
 \theta(\nu_n-\nu_m)   }
 \cr & &\qq\qq\qq\qq \qq \ \times \Big[ \s_n +\big|S'_n-\E_{(V,\e)} S'_n \big| \Big]   \Big[ \s_m +\big|S'_m-\E_{(V,\e)} S'_m
\big|
\Big]
\cr
&\le&     {
 2C\s_n\s_m\over
 \theta(\nu_n-\nu_m)   }
  \Big[ 1 +{\big|S'_n-\E_{(V,\e)} S'_n \big| \over \s_n}\Big]    \Big[1 +{\big|S'_m-\E_{(V,\e)}
S'_m
\big|
   \over\s_m }\Big] 
\cr
 &\le&  \Big(\frac{\s_n\s_m}{\sqrt{\nu_n\nu_m}}\Big)\,   {
2 C\sqrt{ \nu_m}(\sqrt{\nu_n }+\sqrt{ \nu_m})\over
 \theta(\nu_n-\nu_m)   }   \Big[ 1 +{\big|S'_n-\E_{(V,\e)} S'_n \big| \over \s_n}\Big]    \Big[1 +{\big|S'_m-\E_{(V,\e)}
S'_m
\big|
   \over\s_m }\Big] 
   \cr    &= &   \Big(\frac{\s_n\s_m}{\sqrt{\nu_n\nu_m}}\Big)\,    {   2C\over
 \theta\big(\sqrt{\nu_n\over \nu_m}-1  \big) } \Big[ 1 +{\big|S'_n-\E_{(V,\e)} S'_n \big| \over \s_n}\Big]    \Big[1 +{\big|S'_m-\E_{(V,\e)}
S'_m
\big|
   \over\s_m }\Big] .
\end{eqnarray*}
By the Cauchy--Schwarz inequality,
\begin{align}&\E_{(V,\e)}  {\big|S'_n-\E_{(V,\e)} S'_n \big| \over \s_n}\,  {\big|S'_m-\E_{(V,\e)}
S'_m
\big|
   \over \s_m }
   \cr  &\le \Big[\E_{(V,\e)}  {\big|S'_n-\E_{(V,\e)} S'_n \big|^2 \over  \s_n^2}\Big]^{1/2}    \Big[ \E_{(V,\e)}{\big|S'_m-\E_{(V,\e)}
S'_m
\big|^2
   \over  \s_m^2  }\Big]^{1/2}\le C .
   \end{align}
%since
%\begin{align}\E_{(V,\e)}  {\big|S'_j-\E_{(V,\e)} S'_j \big| \over \s_j} \le \Big[\E_{(V,%\e)}  {\big|S'_j-\E_{(V,\e)} S'_j \big|^2 \over\s_j^2}\Big]^{1/2}  \le C . \end{align}
Since
$$\E_{(V,\e)}  { |b_n||b_m| \over
\sqrt{B_{n}\over B_{m}}-1   } \,\le  \,\Big(\frac{\s_n\s_m}{\sqrt{\nu_n\nu_m}}\Big)\,    {   2C\over
 \theta\big(\sqrt{\nu_n\over \nu_m}-1  \big) } \E_{(V,\e)} \Big[ 1 +{\big|S'_n-\E_{(V,\e)} S'_n \big| \over \s_n}\Big]    \Big[1 +{\big|S'_m-\E_{(V,\e)}
S'_m
\big|
   \over\s_m }\Big]
$$
we deduce
\begin{equation}  \E_{(V,\e)} {\bf 1}_{\{B_n>\theta\nu_n,   B_n-B_m> \theta (\nu_n-\nu_m)\}} { |b_n||b_m| \over
\sqrt{B_{n}\over B_{m}}-1   } \le \Big(\frac{\s_n\s_m}{\sqrt{\nu_n\nu_m}}\Big)\,    {   2C\over
 \theta\big(\sqrt{\nu_n\over \nu_m}-1  \big) } .
\end{equation}

 Now on the set $\{B_n>\theta\nu_n,   B_n-B_m> \theta (\nu_n-\nu_m)\}$,
  \begin{eqnarray*}  { |b_n| ^2\over
 {B_{n}\over B_{m}}-1   } &\le&  {C\over
   {B_n} }  \Big[ \s_n +\big|S'_n-\E_{(V,\e)} S'_n \big| \Big]^2 { B_{m}\over
  B_{n}    -B_{m}   }
 \cr  &\le&   C  { \s_n^2B_{m}\over
   {B_n}(\theta (\nu_n-\nu_m))   } \Big[ 1 +{\big|S'_n-\E_{(V,\e)} S'_n \big|\over \s_n} \Big]^2
\cr  &\le&   C  \Big(\frac{\s_n^2}{ \nu_n}\Big)  { B_{m}\over
   \theta^2 (\nu_n-\nu_m) } \Big[ 1 +{\big|S'_n-\E_{(V,\e)} S'_n \big|\over \s_n} \Big]^2
   \cr  & \le&   CB_* \Big(\frac{\s_n^2}{\nu_n}\Big)  { \nu_m\over
   \theta^2 (\nu_n-\nu_m) } \Big[ 1 +{\big|S'_n-\E_{(V,\e)} S'_n \big|\over \s_n} \Big]^2
 \cr  &=&   CB_* \Big(\frac{\s_n^2}{\nu_n}\Big)  { 1\over
   \theta^2 (\frac{\nu_n}{\nu_m}-1) } \Big[ 1 +{\big|S'_n-\E_{(V,\e)} S'_n \big|\over \s_n} \Big]^2 
   \cr  &\le&     C_{\theta,\rho}\,B_* \Big(\frac{\s_n^2}{\nu_n}\Big) {  1\over
    \sqrt {\frac{\nu_n}{\nu_m}}-1   } \Big[ 1 +{\big|S'_n-\E_{(V,\e)} S'_n \big|\over \s_n} \Big]^2 .
\end{eqnarray*}
    Therefore
\begin{eqnarray}  & &\E_{(V,\e)} {\bf 1}_{\{B_n>\theta\nu_n,   B_n-B_m> \theta (\nu_n-\nu_m)\}} {   |b_n| ^2\over
 {B_{n}\over B_{m}}-1   } 
\cr &\le& C_{\theta,\rho}\,  \Big(\frac{\s_n^2}{\nu_n}\Big) {  1\over
    \sqrt {\frac{\nu_n}{\nu_m}}-1   }  \E_{(V,\e)}  B_* \Big[ 1 +{\big|S'_n-\E_{(V,\e)} S'_n \big|\over \s_n} \Big]^2
\cr
&\le& C_{\theta,\rho}\,  \Big(\frac{\s_n^2}{\nu_n}\Big) {  1\over
    \sqrt {\frac{\nu_n}{\nu_m}}-1   }   \Big[ 1 +{1\over  \s_n^2}\E_{(V,\e)}  \big|S'_n-\E_{(V,\e)} S'_n \big|^2  \Big] \cr
&\le&  \Big(\frac{\s_n^2}{\nu_n}\Big) { C_{\theta,\rho}\over
    \sqrt {\frac{\nu_n}{\nu_m}}-1   } .
\end{eqnarray}
  Next $|b_m|  \le  {C\over
 \sqrt  {B_m} }  \Big[ \s_m +\big|S'_m-\E_{(V,\e)} S'_m \big| \Big] $, and so on the set $\{  B_n-B_m> \theta (\nu_n-\nu_m)\}$,
 \begin{eqnarray*}  { |b_m| ^2\over
 {B_{n}\over B_{m}}-1   } 
 &\le&  {C\over
   {B_m} }  \Big[\s_m +\big|S'_m-\E_{(V,\e)} S'_m \big| \Big]^2 { B_{m}\over
  B_{n}    -B_{m}   }
 \cr  &\le&   C  {  \s_m^2\over
   \theta (\nu_n-\nu_m)   } \Big[ 1 +{\big|S'_m-\E_{(V,\e)} S'_m \big|\over \s_m}\Big]^2
\cr  &\le&   C\,  \Big(\frac{\s_m^2}{\nu_m}\Big)  {  \nu_m \over
    \theta (\nu_n-\nu_m)    } \Big[ 1 +{\big|S'_m-\E_{(V,\e)} S'_m \big|\over \s_m}\Big]^2
  \cr  &\le&   C\,  \Big(\frac{\s_m^2}{\nu_m}\Big)  {1 \over
    \theta (\frac{\nu_n}{  \nu_m}-1)    } \Big[ 1 +{\big|S'_m-\E_{(V,\e)} S'_m \big|\over \s_m}\Big]^2  \cr  &\le&   \Big(\frac{\s_m^2}{\nu_m}\Big)   { C_{\theta,\rho}\over
    \sqrt {\nu_n\over \nu_m}-1   } \,\Big[ 1 +{\big|S'_m-\E_{(V,\e)} S'_m \big|\over \s_m} \Big]^2 .
\end{eqnarray*}
    Therefore
\begin{eqnarray}\label{ }& &\E_{(V,\e)} {\bf 1}_{\{B_n>\theta\nu_n,   B_n-B_m> \theta (\nu_n-\nu_m)\}} {   |b_m| ^2\over
 {B_{n}\over B_{m}}-1   }\cr
&\le&   \Big(\frac{\s_m^2}{\nu_m}\Big)   { C_{\theta,\rho}\over
    \sqrt {\nu_n\over \nu_m}-1   }   \Big[ 1 +{\big|S'_m-\E_{(V,\e)} S'_m \big|\over \s_m} \Big]^2
\, \le  \, \Big(\frac{\s_m^2}{\nu_m}\Big)   { C_{\theta,\rho}\over
    \sqrt {\nu_n\over \nu_m}-1   }  .
\end{eqnarray}

 We thus arrive at
\begin{equation}\label{e11} \E_{(V,\e)}{\bf 1}_{\{B_n>\theta\nu_n,   B_n-B_m> \theta (\nu_n-\nu_m)\}} \Big|   e^{-{{\kappa'_n}^2\over
D^2(B_n/2)}}
  -
   e^{-{(\kappa'_n-\kappa'_m )^2\over D^2(B_{n }-B_{ m})/2 }}\Big|
  \, \le \, \Big(\frac{\s_n }{  \sqrt {\nu_n}}+\frac{\s_m }{\sqrt {\nu_m}}\Big)^2 { C_{\theta,\rho}\over
    \sqrt {\nu_n\over \nu_m}-1   }  ,
\end{equation}
 
Consequently, 
\begin{equation}\label{e11a}|A''_{12}|\,\le \, 2\rho^{\nu_n-\nu_m}+\Big(\frac{\s_n }{  \sqrt {\nu_n}}+\frac{\s_m }{\sqrt {\nu_m}}\Big)^2  { C_{\theta,\rho}\over
    \sqrt {\nu_n\over \nu_m}-1   }.\end{equation}
 As by \eqref{Asecond.11},
$$|A''_{11}| \,\le\,{ \E_{(V,\e)} B_*\over \sqrt \theta (\sqrt{{\nu_n\over \nu_m}}-1) }+\|B_*\|_\a\,\sqrt{\nu_m} \,\rho^{(1-1/\a)(\nu_n-\nu_m)},$$
it follows that 
\begin{equation}\label{e11b}|A''_1|\,\le \, { \E_{(V,\e)} B_*^{1/2}\over \sqrt \theta (\sqrt{{\nu_n\over \nu_m}}-1) }+\|B_*^{1/2}\|_\a\,\sqrt{\nu_m} \,\rho^{(1-1/\a)(\nu_n-\nu_m)}+ 2\rho^{\nu_n-\nu_m}+ \Big(\frac{\s_n }{  \sqrt {\nu_n}}+\frac{\s_m }{\sqrt {\nu_m}}\Big)^2  { C_{\theta,\rho}\over
    \sqrt {\nu_n\over \nu_m}-1   }.\end{equation}
 We have $$|A    |\le  A'+C_0|A''_2|+ C_\theta\Big({\s_n \over \sqrt{\nu_n}}\Big)\,A''_1.
   $$   Recalling that $   |A'| \le \s_n\, {\color{black} \prod_{j=m+1}^n\vartheta_j  }$, and $|A''_2|\le C\, \s_n \big\{ \rho^{ \nu_n-\nu_m }  +  {1\over (\theta\nu_n)^{3/2}}+ {1\over
(\theta(\nu_n-\nu_m))^{3/2}}\big\}$,  by \eqref{e1} and \eqref{e6} respectively,
we obtain
\begin{eqnarray}\label{e11c} |A |&\le &  \s_n \, {\color{black} \prod_{j=m+1}^n\vartheta_j  }+\,C_0\,C\, \s_n \Big\{ \rho^{ \nu_n-\nu_m }  +  {1\over (\theta\nu_n)^{3/2}}+ {1\over
(\theta(\nu_n-\nu_m))^{3/2}}\Big\} 
\cr & &  +C_\theta\Big({\s_n \over \sqrt{\nu_n}}\Big)\bigg\{ { \E_{(V,\e)} B_*\over \sqrt \theta (\sqrt{{\nu_n\over \nu_m}}-1) } 
     +\,\|B_*\|_\a\,\sqrt{\nu_m} \,\rho^{(1-1/\a)(\nu_n-\nu_m)}
+ 2\rho^{\nu_n-\nu_m} 
    \cr & & + \Big(\frac{\s_n }{  \sqrt {\nu_n}}+\frac{\s_m }{\sqrt {\nu_m}}\Big)^2  { C_{\theta,\rho}\over
    \sqrt {\nu_n\over \nu_m}-1   }\bigg\}.
    \end{eqnarray}  
  We note that $$      \rho^{\nu_n-\nu_m}  +  {1\over (\theta\nu_n)^{3/2}}+ {1\over
 (\theta(\nu_n-\nu_m) )^{3/2}} 
$$ is less than $  {1  \over
 (\nu_n-\nu_m) ^{3/2}} $, up to a constant $C_{\theta,\rho,\a}$. Let    
 $\s(n,m) =\max \big(\frac{\s_n }{\sqrt{\nu_m}},\frac{\s_m }{\sqrt{\nu_m}},1\big)$.
  \vskip 3 pt  By    considering separately the cases $\nu_n-\nu_m\ge 1$,  $0<\nu_n-\nu_m< 1$, we thus get 
 \begin{eqnarray}\label{e11ca} |A |&\le & {\color{black}\s_n\prod_{j=m+1}^n\vartheta_j }  
\,  +\,
 C_{\theta,\rho,\a} \,{\s_n  \over
 (\nu_n-\nu_m) ^{3/2}} 
\cr & &  +C_{\theta,\rho,\a} \,\Big({\s_n \over \sqrt{\nu_n}}\Big)\bigg\{ { 1\over    \sqrt{{\nu_n\over \nu_m}}-1} 
     +\, \sqrt{\nu_m} \,\rho^{(1-1/\a)(\nu_n-\nu_m)}
 \cr & &\qq\qq\qq\qq+\Big(\frac{\s_n }{  \sqrt {\nu_n}}+\frac{\s_m }{\sqrt {\nu_m}}\Big)^2  { 1 \over
    \sqrt {\nu_n\over \nu_m}-1   }\bigg\}
 .
    \end{eqnarray}  
 
  \vskip 2 pt
Now by \eqref{tD}, 
%$\s^2\ge  \frac{ D^2   }{4} \t_Y$ 
$1  \le  \frac{2}{ D   }\,\frac{\s   }{\sqrt{\t_Y}} $, and so 
$$\s(n,m) \le \frac{C}{D}\, \max \big(\frac{\s_n }{\sqrt{\nu_n}},\frac{\s_m }{\sqrt {\nu_m}} \big).$$
Thus 
\begin{eqnarray}\label{e11ca.f} |A |&\le & {\color{black}\s_n\prod_{j=m+1}^n\vartheta_j }  
\,  +\,
 C_{\theta,\rho,\a}  \,\Big({\s_n \over \sqrt{\nu_n}}\Big)\,{\sqrt{\nu_n} \over
 (\nu_n-\nu_m) ^{3/2}} 
\cr & &  +C_{\theta,\rho,\a} \,\Big({\s_n \over \sqrt{\nu_n}}\Big)\bigg\{ { 1\over    \sqrt{{\nu_n\over \nu_m}}-1} 
     +\, \,{\sqrt{\nu_m} \over
 (\nu_n-\nu_m) ^{3/2}} 
 \cr & &\qq\qq\qq\qq+\Big( \frac{C}{D}\, \max \big(\frac{\s_n }{\sqrt{\nu_n}},\frac{\s_m }{\sqrt {\nu_m}} \big)\Big)^2  { 1 \over
    \sqrt {\nu_n\over \nu_m}-1   }\bigg\}
\cr &\le &   {\color{black}\s_n\prod_{j=m+1}^n\vartheta_j }  
   +C_{\theta,\rho,\a} \,\Big({\s_n \over \sqrt{\nu_n}}\Big)\bigg\{ { 1\over    \sqrt{{\nu_n\over \nu_m}}-1} 
     +\, \,{\sqrt{\nu_n} \over
 (\nu_n-\nu_m) ^{3/2}} \Big\}
 \cr & &  +C_{\theta,\rho,\a}D\, \Big( \frac{C}{D}\, \max \big(\frac{\s_n }{\sqrt{\nu_n}},\frac{\s_m }{\sqrt {\nu_n}} \big)\Big)^3 \,\bigg\{   { 1 \over
    \sqrt {\nu_n\over \nu_m}-1   } \bigg\}
  \cr&\le &   C_{\theta,\rho,\a} \,   \frac{C}{D^2}\, \max \big(\frac{\s_n }{\sqrt{\nu_n}},\frac{\s_m }{\sqrt {\nu_m}} \big)^3
  \cr & & \qq \times\bigg\{\, {\color{black}\nu_n^{1/2} \prod_{j=m+1}^n\vartheta_j }  
 %  {\nu_n^{1/2}  \over (\nu_n-\nu_m) ^{3/2}}
   +{ 1\over    \sqrt{{\nu_n\over \nu_m}}-1}+\,{\sqrt{\nu_n} \over
 (\nu_n-\nu_m) ^{3/2}} \bigg\}.
    \end{eqnarray}  
   
\vskip 3 pt Using   \eqref{basic}, \eqref{llt.iid.appl.}
we obtain (the value of $\a$ being irrelevant and $\theta,\rho$ depending only on $\e$, the value of which being irrelevant too)
  \begin{align}\label{e14a}
  \big|{\mathbb E\,}  & Y_nY_m\big|  \,\le \,   \frac{C}{D^2}\, \max \big(\frac{\s_n }{\sqrt{\nu_n}},\frac{\s_m }{\sqrt {\nu_m}} \big)^3\cr &\times \bigg\{\, {\color{black}\nu_n^{1/2} \prod_{j=m+1}^n\vartheta_j }  
   +{ 1\over    \sqrt{{\nu_n\over \nu_m}}-1}+\,{\sqrt{\nu_n} \over
 (\nu_n-\nu_m) ^{3/2}} \bigg\}
 \end{align}
if $n>m\ge 1 $.
%$$  e^{- \sum_{j=m+1}^n\log \frac{1}{\vartheta_j }}\le \frac{C_M}{(\sum_{j=m+1}^n\log \frac{1}{\vartheta_j })^M}\le C \frac{1}{\sum_{j=m+1}^n \vartheta_j  }
%$$
  This together with \eqref{Y},   proves \eqref{t2[asllt]enonce1}.

%{\color{blue}$$ ( \sum_{j=m+1}^nx_j) ^{3/2}\prod_{j=m+1}^nx_j   \le   1   $$}
\vskip 8  pt 
 
 Finally
 %  $\underline{\tau} =\inf_{j\ge 1} \t_{X_j}>0$, we pick $  \vartheta_j $ so that 
%$$ 0 < \vartheta_j \le \underline{\tau}\le \t_{X_j}, \qq \quad j=1,2,\ldots.$$and  note that  $$\log \frac{1}{\vartheta_j} \ge \log \frac{1}{\underline{\tau}} \ge \Big(\frac{\log \frac{1}{\underline{\tau}}}{\underline{\tau}}\Big)\, \vartheta_j . $$ Similarly 
if $ {\tau}=\sup_{j\ge 1} \t_{X_j}<1$, we pick $  \vartheta_j $ so that 
$$ 0 < \vartheta_j \le   \t_{X_j} \le    {\tau} <1, \qq \quad j=1,2,\ldots,$$
 and  note that  $$\log \frac{1}{\vartheta_j} \ge \log \frac{1}{ {\tau}} \ge \Big(\frac{\log \frac{1}{ {\tau}}}{ {\tau}}\Big)\, \vartheta_j , $$
Thus
%  in both cases letting $\tau$ be equal to $\underline{\tau}$  or $  \overline{\tau}$, we have
 $$   \prod_{j=m+1}^n\vartheta_j    
=e^{ -\sum_{j=m+1}^n \log \frac{1}{\vartheta_j }   } \le \frac{C(M)}{(\sum_{j=m+1}^n \log \frac{1}{\vartheta_j } )^M} \le \frac{C( {\tau},M)}{(\sum_{j=m+1}^n  \vartheta_j   )^M}=\frac{C( ,M)}{(\nu_n -\nu_m )^M}.
$$    

 Taking $M=3/2$ provides the bound 
  $$   \prod_{j=m+1}^n\vartheta_j    
\le \frac{C( {\tau} )}{(\nu_n -\nu_m )^{3/2}  }.
$$

By carrying it back   to   estimate \eqref{e14a}, 
we get 
\begin{equation}\label{e11ca.} \big|{\mathbb E\,}    Y_nY_m\big|\, \le        \,   \frac{C_{\tau,c }}{D^2}\, \max \big(\frac{\s_n }{\sqrt{\nu_n}},\frac{\s_m }{\sqrt {\nu_m}} \big)^3
 \bigg\{\,  { 1\over    \sqrt{{\nu_n\over \nu_m}}-1}+\,{\sqrt{\nu_n} \over
 (\nu_n-\nu_m) ^{3/2}} \bigg\},
    \end{equation}  which is estimate \eqref{t2[asllt]enonce2}. This completes the proof of Theorem \ref{t2[asllt]}.
\vskip 7 pt

 \subsection{Proof of Corollary    
 \ref{cort1}}\label{s3.1}
  Let $0<c<1$. Let $\nu_m\le c\,\nu_n$. Then
$$   {1\over    \sqrt{ \nu_n/\nu_m}-1   }\le \Big({1\over 1-\sqrt c} \Big) \, \sqrt{\nu_m\over \nu_n}= C_c \, \sqrt{\nu_m\over \nu_n} . $$
 Further
  $$     {\sqrt \nu_n   \over
 (\nu_n-\nu_m) ^{3/2}}=  \sqrt {\nu_n\over \nu_n-\nu_m}\     {  1\over
 (\nu_n-\nu_m)  }  =  {1\over{ (1- \nu_m/\nu_n)^{3/2} }}
\ {1 \over \nu_n}\le {1\over{ (1- c)^{3/2} }}
\ {1  \over \nu_n  }  . $$
By incorporating these estimates into \eqref{e11ca} 
 we get
\begin{eqnarray}\label{e11a1}\big|{\mathbb E\,}   Y_nY_m\big|  \,\le \,     \frac{C_{\tau,c}}{D^2}\, \max \big(\frac{\s_n }{\sqrt{\nu_n}},\frac{\s_m }{\sqrt {\nu_m}} \big)^3  \sqrt{\nu_m\over \nu_n}.\end{eqnarray}
This proves Corollary \ref{cort1}.

\vskip 20 pt
 %   Put for any positive integer $i$, $$ Z_i=\sum_{2^i\le  n<2^{i+1}}{\vartheta_nY_n\over   \s_n\sqrt{\nu_n}}    . $$
%We apply Theorem \ref{t2[asllt]} to estimate
%$$ |\E Z_iZ_j|, \qq \E Z_i^2.$$

  In the following Example,  we study    the correlation of the system of random variables (recalling  %that  $Y_n= \s_n  \big({\bf 1}_{\{S_n=\kappa_n\}}-{\mathbb P}\{S_n=\kappa_n\}\big) $, by 
       \eqref{Y})
 \beq\label{cor.syst} {\vartheta_nY_n\over  
 \s_n\sqrt{\nu_n}}  = {\vartheta_n   ({\bf 1}_{\{S_n=\kappa_n\}}-{\mathbb P}\{S_n=\kappa_n\}) \over  
  \sqrt{\nu_n}}   , \qq\quad n\ge 2.
\eeq

We consider the   case   
 $$\lim_{n\to\infty} \t_{X_n} =0 .$$

\subsection{Example}\label{s2..}
 Let  \beq \label{BE.data}    {\rm Var} (X_n)  =\t_{X_n}= \log  n  -   \log (n-1)   , \qq \quad n\ge 1.
\eeq
% Recall that by \eqref{basic0}, \beq\label{EYn2.bound}\E Y_n^2\,\le \, C\,\s_n.\eeq
We also choose 
$$  D=2, \qq\qq \t_n=\t_{X_n}, \qq n\ge 1.
$$    As  $\nu_n= \log n$, assumption \ref{nun} is satisfied.  By \eqref{BE.data}, 
\beq
   \s_n=\sqrt {\log n} = \sqrt{\nu_n} .
   \eeq   
   This   is compatible with \eqref{tDn},
   % ($\s_n\ge \sqrt {\log n}$) 
   since we assumed $D=2$.  We have $\frac{ \vartheta_n }{  
 \s_n\sqrt{\nu_n}}=\frac{ 1 }{  
n \log n}$, thus the series $\sum_{n\ge 1}\frac{ \vartheta_n }{  
 \s_n\sqrt{\nu_n}}$ is divergent. 
   \vskip 5 pt

 %%%%%%%%%%%%%%%%%
% \subsection{The case $\boldsymbol{ \t_{X_n}=\frac1n,\ \s_n= \sqrt{\nu_n}}$}\label{sub2..1}
       
In the case considered, 
\beq\label{cor.syst.ex}
   %{\vartheta_nY_n\over  \s_n\sqrt{\nu_n}} = 
   { Y_n\over  
  n \log n } = { {\bf 1}_{\{S_n=\kappa_n\}}-{\mathbb P}\{S_n=\kappa_n\}  \over  
 n\sqrt {\log n}}    , \qq\quad n\ge 2.
 \eeq
 %We first estimate $\E Z_i^2$. 

We    examine the magnitude of the  $L^2$-increments.
% (in the argument $n$), $n$ being not too far from $m$, 
 
\begin{lemma}\label{example2} There exists an absolute constant $C$ such that for any integers   $1<M_1<M_2$, 
 \ben\label{sm1m2.} &&\E\Big|\sum_{M_1\le m\le   M_2} {   Y_m  \over  
 m\log m  }\,\Big|^2\,\le \, C \,\sum_{M_1\le m\le   M_2}  {1\over   m\log m}
    .
\een
 \end{lemma}

\begin{proof}Let $b>1$ be some fixed real. 
We consider several cases.    \vskip 2pt {\tt Case\,1:}($m=n$) By using \eqref{basic0}, we have $\E  Y_n^2\le C\,\sqrt{  n}$,  so that
 $$ {\E  Y_m^2\over  m\log m} \,\le \,   {C\over   \sqrt{  m}\log m}\,\le \, C.$$

 \vskip 2pt {\tt Case\,2:}($m+1 \le n< m^b $) 
 %\cc
  By using \eqref{basic.corr.bd}, 
      % |{\mathbb E\,} Y_nY_m|     \le  C\, \Big(\frac {\s_n}{\s_{n-m}}+1\Big)= C\, \Big(\sqrt{{\log n\over \log (n-m)}} +1\Big).
 
\ben\label{t2[asllt]enonce2.cite..} 
\sum_{m+1 \le n< m^b}   { |\E Y_mY_n|\over  
n\log n}&    \le  &      \, C \sum_{m+1 \le n< m^b}
 \,  \Big({ 1\over   n\sqrt{\log n}  \sqrt{ \log (n-m)} }+{1\over  
n\log n}\Big)  
\cr &    \le  &      \, C_b \Big({ 1\over   \sqrt{\log m}} \sum_{1\le h \le m^b }
 \,  { 1\over   h \sqrt{ \log h} } + \big(\log\log m^b- \log\log m\big)\Big)
 \cr &    \le  &      \, C_b \Big({ 1\over   \sqrt{\log m}}\, \sqrt{\log m}+\log\frac{b \log m}{\log m}\Big) \le C_b. \een

\vskip 2pt {\tt Case\,3:}($n\ge m^b$) By Corollary \ref{c1[cor.est]}, \begin{align}\label{t2[asllt]enonce2.cite} 
|\E Y_nY_m|&\ =\  \s_n \s_m \, \Big|{\mathbb P}\{S_n=\k_n, S_m=\k_m\}- {\mathbb P}\{S_n=\k_n \}{\mathbb P}\{  S_m=\k_m\} \Big|
 \cr &  \, \le        \,   \frac{C_{\tau }}{D^2}\, \max \big(\frac{\s_n }{\sqrt{\nu_n}},\frac{\s_m }{\sqrt {\nu_m}} \big)^3
 \bigg\{\,  { 1\over    \sqrt{{\nu_n\over \nu_m}}-1}+\,{\sqrt{\nu_n} \over
 (\nu_n-\nu_m) ^{3/2}} \bigg\}
  \cr &  \, \le        \, C 
 \bigg\{\,  { 1\over    \sqrt{{\nu_n\over \nu_m}}-1}+\,{\sqrt{\nu_n} \over
 (\nu_n-\nu_m) ^{3/2}} \bigg\} \, \le        \, C 
 \,  { 1\over    \sqrt{{\nu_n\over \nu_m}}-1} .
\end{align}
Thus
\ben\sum_{n>m^b}   { |\E Y_mY_n|\over  
n\log n}&    \le   &     \, C \sum_{n>m^b}     \,  { 1\over   n(\log n)  \sqrt{{\log n\over \log m}}-1} .
\een We have ${\log n\over \log m} \ge (1-1/b) $. Thus
$$   {1\over    \sqrt{ {\log n\over \log m}}-1   }\le \Big({1\over 1-\sqrt c} \Big) \, \sqrt{{\log m\over \log n}}= C_c \, \sqrt{{\log m\over \log n}} . $$
So that
\beq\sum_{n>m^b}   { |\E Y_mY_n|\over  
n\log n}   \,  \le        \, C_b \sum_{n>m^b}     \,  { 1\over   n(\log n)  }\sqrt{{\log m\over \log n}}
\,  =         \, C_b  \sqrt{ \log m }\sum_{n>m^b}     \,  { 1\over   n(\log n)^{3/2}  }.
\eeq 
%n \log n(   \sqrt{{\log n\over \log m}}-1) }\le C_b\,  \sqrt{\log m}\sum_{n\ge m^b}\frac{ 1 }{  
%n \log^{3/2} n }\le C_b \, \frac{   \sqrt{\log m}  }{  \sqrt{\log m^b}  }\le C_b.\eeq 
 This sum is a special case of the sum 
$$\sum_{n>N}\frac{  \vartheta_n }{\nu_{n}^\d\nu_n },$$
and we can apply Lemma \ref{hlp} with $\d= \frac12$. Consequently (choosing $N= m^b$)
$$  \sqrt{ \log m^b }\sum_{n>m^b}     \,  { 1\over   n(\log n)^{3/2}  }\le  C_b.$$
We deduce
\beq\sum_{n>m^b}   { |\E Y_mY_n|\over  
n\log n}   \,  \le        \, C_b          \sqrt{ \log m }\sum_{n>m^b}     \,  { 1\over   n(\log n)^{3/2}  }\,  \le        \, C_b.
\eeq 
By combining,
\beq\sum_{n\ge m }   { |\E Y_mY_n|\over  
n\log n}   \,  \le           \, C_b.
\eeq
Now  noticing that
\ben &&\E\Big|\sum_{M_1\le m\le   M_2} {   Y_m  \over  
 m\log m  }\,\Big|^2
 \cr &\le &\sum_{M_1\le m\le   M_2} { \E  Y^2_m  \over  
 (m\log m )^2 } +2\sum_{M_1\le m \le    M_2}\frac{1}{  m\log m}\Big(\sum_{  m<n\le    M_2} {  | \E Y_mY_n | \over  
 n\log n   }\Big).
\een
  
\end{proof}\vskip 3 pt 
% Now we estimate $\E Z_iZ_j$, for $i$.$\log m\le c \log n$[] $ m\le    n^c $[]$\ 2^i\le  2^{cj}$. Let $0<c<1$. By Corollary \ref{cort1}, under Assumption \eqref{nun},  there   exists a constant$C_{c}  $  such that for all $1\le \nu_m\le c\nu_n$,$$\big|{\mathbb E\,}   Y_nY_m\big|  \,\le \,     C_{c,D} \,h(n)\,  \sqrt{\nu_m\over \nu_n}.$$

%\vskip 25 pt 

\section{\gsec  A GENERAL ASLLT WITH ALMOST SURE CONVERGENT  SERIES}\label{s4}
 \vskip 10 pt 
%%%%%%%%%%%%%%%%%
In this Section we derive from Theorem \ref{t2[asllt]} and Corollary \ref{cort1}, a sharp  almost sure local limit theorem for sums of independent random variables. This one is  stated in Theorem \ref{t1[asllt].}, assertion (iv),  and thereby extends the ASLLT proved in the i.i.d. case (Theorem \ref{t1[asllt]..}), to the independent non-identically distributed case. This result is deduced from new  general stronger forms  in (i), (ii) and (iii),     in which the almost sure convergence of  the naturally associated series, namely 
$$
 \sum_{k\ge 1}  
\sup
\bigg\{ \frac{\big|\sum_{2^j\le n<2^{j+1}}{     {\bf 1}_{\{S_n=\kappa_n\}}-{\mathbb P}\{S_n=\kappa_n\} 
  \over 
 \sqrt{\nu_n}}\big|}{\sum_{1\le n<2^{j} }{1\over \s_n\sqrt{\nu_n} }
}\,,\,   2^k\le\sum_{1\le n<2^{j} }{1\over \s_n\sqrt{\nu_n} }
<2^{k+1} \bigg\} , $$
is established.  For proving these results  {\it no} LLT is assumed; and thus  these    are  in that sense {\it disconnected} from LLT. However if the LLT is applicable, our assumptions can be weakened and the ASLLT  becomes stronger: compare   Claims (iii) and (iv) in Theorem \ref{t1[asllt].}.

%\medskip\par  
\subsection{Basic Assumption}
%Assumptions   in Theorem \ref{t2[asllt]} are used, in particular we assume \eqref{nun},  \begin{equation*}
%  \nu_n =\sum_{j=1}^n \t_j \, \uparrow \infty, \qq \quad n\to \infty.
% \end{equation*} 
 
 We   assume  throughout this Section that
 \begin{eqnarray*}    \hbox{\it The series $ \sum_{ 
n\ge 1}{\vartheta_n\over \s_n\sqrt{\nu_n} }$ is divergent.}
  \end{eqnarray*}

     \begin{remark}\label{miconv} 
  %Assumption \eqref{midiv} is used in the  definition   of the almost sure local limit property.  
   It will  be later proved   that  in the case this assumption is not satisfied, by  assuming only a little more than the convergence of the series   $ \sum_{ 
n\ge 1}{\vartheta_n\over \s_n\sqrt{\nu_n} }$,  a strong  convergence result holds, although   radically different from the almost sure local limit property. 
%the sequence 
% $$\sum_{n< 2^J}  {\vartheta_n({\bf 1}_{\{S_n=\kappa_n\}}-{\mathbb P}\{S_n=\kappa_n\})\over  
%  \sqrt{\nu_n}}\qq J=1,2,\ldots,$$ converges  almost surely, which is   strong  convergence result, but  radically different from the almost sure local limit property. %  See  Remark   \ref{opt.midiv} for details.
\end{remark}

\subsection{Introducing function $\boldsymbol h$}    \vskip 5 pt    
   Let  
   function  $h$   be    defined as follows:
\begin{eqnarray}\label{h}
  h(x)& =&  \displaystyle{\max_{1\le m\le x }}\Big( \frac{\s_m^2 }{ {\nu_m}}\Big), 
\qq \quad x\ge 1.
   \end{eqnarray}

 \begin{remark}\label{bdd.}By  inequality  \eqref{tDn}, $\s_n^2\,\ge  \,\frac{ D^2   }{4}\,\nu_n$, for $n\ge 1$, so that 
 $$h(x)\ge  { D^2   }/{4},\qq \quad x\ge 1.
$$
  Now if 
\begin{equation} \label{bdd}
   \s_n^2=\mathcal O(\nu_n)
   ,\end{equation}
 one can take $h(x)\equiv {\rm const.}$, and this shows that the basic assumption  is     fulfilled.  This follows from assertion  (i) of Th.\,162 in  \cite{HLP}.
 Condition \eqref{bdd} is obviously fulfilled in  the i.i.d. case. 
 \end{remark}   
  
 \subsection{A system of weights}

     \vskip 5 pt   
     Introduce a   system of weights associated with sequences $\{\t_n, n\ge 1\}$, $\{\s_n, n\ge 1\}$, and of key importance for controlling $\E Z_i^2$. 
     
     \vskip 2 pt Set  
\beq \label{omega}
\o(m) = \sum_{      \nu_m< \nu_n <2\nu_m} \, {  \vartheta_n  \over     \,(\s_n^2-\s_m^2)^{1/2}\,\sqrt{\nu_n }   }
%\sum_{      \nu_m< \nu_n <2\nu_m} \, {  \vartheta_n  \over     \sqrt{\nu_n } \,\s_{n-m} }   , 
\qq \quad m\ge 1.
\eeq
 %The link with $m_i$ follows from 
 %  \beq \label{omega.mi}
%\o(n)  \ge C\,\sum_{      \nu_n< \nu_m <2\nu_n} \, {  \vartheta_m  \over     \s_{m}\,\sqrt{\nu_m }  }\ge Cm_i   , \qq \hbox{if $\nu_n \asymp 2^i$}.
%\eeq 
 
%\subsection{Properties of $\boldsymbol \o$}  
  From \eqref{tDn} follows that
  % \begin{lemma}  \label{omega.lemma}We have for all $m\ge 1$,
%\vskip 2 pt {\rm (i)} 
\beq \label{omega1}
\o(m)\le   C_\t\,\sum_{      \nu_m< \nu_n <2\nu_m} \,    \frac {\vartheta_n}{\sqrt{\nu_{n } -\nu_{m} }\, \sqrt{\nu_{n }}  } .
\eeq
%\vskip 3 pt
%{\rm (ii)} If the sequence $\{\t_n, n\ge 1\}$ is decreasing, then 
%\beq \label{omega2}
%\o(m) \le C_\t\,  \sum_{ h\ge 1  \atop     \nu_m< \nu_{m+h} <2\nu_m} \,     \frac {\vartheta_{h}}{ \sqrt{\nu_{h}  }}   .
%\eeq
%\end{lemma}
%\begin{proof} %Assertion (i)
  
%Assertion (ii) is immediate.
%\end{proof}  
\vskip 5 pt    We begin with giving some examples for which the  following condition
    \beq \label{omega.star}
\o_X^*: =\sup_{m\ge 1} \,\o(m)<\,\infty
% \sup_{m\ge 1} \ \Big(\sum_{      \nu_m< \nu_n <2\nu_m} \, {  \vartheta_n  \over     \sqrt{\nu_n } \,\s_{n-m} } \Big)<\,\infty
,\eeq
is satisfied.

\begin{example}[i.i.d. case]\label{example3.}   In this  case ($\t_n=C$, $\nu_n=n$, $\s_n=\sqrt n$), we have
% the bound 
%\ben\label{}\o(m)=    \sum_{      \nu_m< \nu_n <2\nu_m} \, {  \vartheta_n  \over     \sqrt{\nu_n }  } \, \Big( \frac {1}{\s_{n-m}  } \Big)=\mathcal O(1).
%\een 
%Indeed   %\sum_{m+1 \le n< m^b}   { |\E Y_mY_n|\over  n\log n}
$$\o(m)\,=\,    C\,\sum_{       m<  n < 2m} \, {  1 \over     \sqrt{ n } \,  \sqrt{n-m}  }  \,\le\,{  C \over     \sqrt{ m }}\, \sum_{     h=1}^{m} \, {  1 \over       \sqrt{h}  }
    \,\le\,{  C \over     \sqrt{ m }}\, \sqrt{ m }\,=\, C
   . $$
Thus $\o_X^*$ is finite. 
\end{example}

 \begin{example}[Cram\'er's probabilistic model]\label{cramer} This well-known model  basically consists with  a sequence of independent random variables $\xi=\{\xi_j,j\ge 3\}$, and    associated partial sums  $S_n =\sum_{j=3}^n \xi_j$, each  $\xi_j$ 
 being  defined   by
\begin{equation}\label{C.xin}\P\{\xi_j= 1\}= \frac1{\log j}, \qq \quad \P\{\xi_j= 0\}= 1-\frac1{\log j}.\end{equation}
      Thus for $j\ge 3$, $\P\{\xi_j=k\}\wedge\P\{\xi_j=k+1 \}=\frac{1}{\log j }$, if $k=0$,    equals $0$ for $k\in \Z_*$; so the span is 1, and $  \t_{\xi_j}=\frac{1}{\log j}$.  Choose $ \t_j= \t_{\xi_j}$,  then  
   $\nu_n=\sum_{j=3}^n \frac{1}{\log j}$.  % Thus $  \s_n^2\sim\nu_n\sim \frac{n}{\log n}$, as $n\to \infty$. 
   We have $\E S_n=\sum_{j=3}^n    \frac{1}{\log j}$, $\s^2_n=\sum_{j=3}^n  (1-\frac{1}{\log j})(\frac{1}{\log j})$.
 Moreover $h(n)=\mathcal O(1)$.
 
 We have for $n>m$, $m$ large, using $\big( {x}/{\log x}\big)'=  {1}/{\log x}- {1}/{(\log x)^2}$,   $$ \s_n^2-\s_m^2 =\sum_{j=m+1}^n \big( \frac{1}{\log j}-\frac{1}{(\log j)^2}\big)\sim \int_{m}^n\big(\frac{x}{\log x}\big)'\dd x =\frac{n}{\log n}-\frac{m}{\log m}\ge (1-o(1))\frac{n-m}{\log n} . $$  Using also that $\nu_n\ge \frac{n}{\log n}-o(1)$,  $2\nu_m \sim \nu_{2m}$, we deduce 
\ben
\o(n) &\le &\sum_{\nu_m<\nu_n<2\o_m} \, {  \vartheta_n  \over     \,(\s_n^2-\s_m^2)^{1/2}\,\sqrt{\nu_n }   }\,\le\, C\sum_{\nu_m<\nu_n<2\o_m} \, \frac{1}{\log n}\,\sqrt{\frac{\log n}{n}}\,\sqrt{\frac{\log n}{n-m}}
\cr &\le & C\sum_{\nu_m<\nu_n<2\o_m}  \frac{1}{\sqrt{n-m}}\,\le\, C\int_m^n \frac{\dd x}{\sqrt{x(x-m)}}
 \,\le \, C\int_0^1 \frac{\dd u}{\sqrt{u(u+1)}}\, =C.
\een
So that $\o_\xi^*<\infty$.

By Proposition 3.1 in \cite{W1b}, the local limit theorem holds,
 \begin{eqnarray*}   \Big| \P \{S_n =\kappa \} -\frac{ 1 }{ \sqrt{2\pi}\s_n  }  e^{- \frac{(\k- m_n)^2}{    2  \s_n^2   } }    \Big|         & \le &    C\,
 \frac{(\log n)^{3/2}}{n}  ,
   \end{eqnarray*}
for all $\k\in\Z$ such that
$ |\k- m_n|  \le C \,\frac{n^{3/4}}{\log n}$.  \end{example}
In the next  Example  however  we can only show that $\o(m)=\mathcal O(\log\log m)$.    \begin{example}[Example \ref{s2..} continued]\label{example3..} Here    $\t_n=\frac1n$, $\nu_n=\log n$, $\s_n= \sqrt{\log n}$. By using Cauchy-Schwarz's inequality, 
\begin{eqnarray*} %\sum_{m+1 \le n< m^b}   { |\E Y_mY_n|\over  n\log n}
\o(m)&=&\sum_{      \nu_m< \nu_n <2\nu_m} \,  { 1  \over   n  \sqrt{(\log n) (\log  n-\log m)}}
\cr &  
  =  &        \sum_{m+1 \le n< m^2}
 \,   { 1\over   n\sqrt{\log n}  \sqrt{ \log (n-m)} }
 \cr &    \le &\Big(\int_m^{m^2}\frac{\dd x}{x\log x}\Big)^{1/2} \Big(\int_m^{m^2}\frac{\dd x}{x(\log x-\log m)} \Big)^{1/2} 
 \cr &    = &\sqrt{ \log 2}\, \Big(\int_1^{m }\frac{\dd t}{t \log t } \Big)^{1/2} 
 \cr &    = &\sqrt{(\log 2)\log\log m} .  
 \end{eqnarray*}  
 \end{example}
  \medskip \par
 
%In order to estimate $\o(m)$ in the general case, we need an additional  Lemma. 
The next lemma is of relevance  for   estimating the rectangle sums appearing in the first Step of  the proof of Theorem \ref{t1[asllt].}.
\begin{lemma}\label{hlp}  We have  the following estimates,
\vskip 2 pt {\rm (i)} Let $0<\d<1$. Then for all $N\ge1$,
\begin{eqnarray*} \sum_{n>N}\frac{\d \vartheta_n }{\nu_{n-1}^\d\nu_n } \le    \frac{1}{\nu_{N}^\d } .
\end{eqnarray*}
\vskip 2 pt {\rm (ii)} Further, for $M>N\ge 1$  such that $  \nu_N>e$,
\begin{eqnarray}\label{delta.0..}
\Big|\sum_{N<n\le M}\frac{  \vartheta_n }{ \nu_n }\Big| 
 &\le &    \big|\log \nu_M -\log \nu_N \big| .\end{eqnarray}
\end{lemma}
\begin{proof} (i) We proceed as in the proof of Th.\,162 in Hardy, Littlewood and P\'olya \cite{HLP}.
 By Th.\,41 in  \cite{HLP}, $\d \vartheta_n\nu_n^{\d-1}=\d\nu_n^{\d-1}(\nu_n-\nu_{n-1})\le \nu_n^{\d}- \nu_{n-1}^{\d}$. So that
\begin{eqnarray}\label{delta1}\sum_{n>N}\frac{\d \vartheta_n }{\nu_{n-1}^\d\nu_n }= \sum_{n>N}\frac{\d \vartheta_n\nu_n^{\d-1}}{\nu_{n-1}^\d\nu_n^\d}
&=&\sum_{n>N}\frac{\d\nu_n^{\d-1}(\nu_n-\nu_{n-1})}{\nu_{n-1}^\d\nu_n^\d}\le \sum_{n>N}\frac{\nu_n^{\d}- \nu_{n-1}^{\d}}{\nu_{n-1}^\d\nu_n^\d}
\cr &=&\sum_{n>N}\Big(\frac{1}{\nu_{n-1}^\d }-\frac{1}{\nu_{n}^\d } \Big)\, =\,  \frac{1}{\nu_{N}^\d }  .\end{eqnarray}

(ii) Now, at first for $M>N\ge 1$  such that $  \nu_N>e$, we have 
\begin{eqnarray}\label{delta.0}\sum_{N<n\le M}\frac{\d \vartheta_n }{\nu_{n-1}^\d\nu_n }&=& \sum_{N<n\le M}\frac{\d \vartheta_n\nu_n^{\d-1}}{\nu_{n-1}^\d\nu_n^\d}
=\sum_{N<n\le M}\frac{\d\nu_n^{\d-1}(\nu_n-\nu_{n-1})}{\nu_{n-1}^\d\nu_n^\d}  
\cr &\le & \sum_{N<n\le M}\frac{\nu_n^{\d}- \nu_{n-1}^{\d}}{\nu_{n-1}^\d\nu_n^\d}
= \sum_{N<n\le M}\Big(\frac{1}{\nu_{n-1}^\d }-\frac{1}{\nu_{n}^\d } \Big)
\cr&=&   \frac{1}{\nu_{N}^\d } -\frac{1}{\nu_{M}^\d }   \,=\,\d\,(\log \nu_M -\log \nu_N)\,  e^{-  \theta(N,M)} 
,
\end{eqnarray}
for some {\it positive} real $\theta(N,M)$ such that $\d\log \nu_N<\theta(N,M)<\d\log \nu_M$, where we used  Lagrange's Theorem in the last line of inequalities. It follows that  \begin{eqnarray}\label{delta.0.}\Big|\sum_{N<n\le M}\frac{  \d\vartheta_n }{\nu_{n-1}^\d\nu_n }\Big| 
 &\le &   \d\, \big(\log \nu_M -\log \nu_N\big)  .\end{eqnarray}
By dividing both sides by $\d$,   letting next $\d$ tend to 0 in the left-term,   we obtain
\begin{eqnarray}\label{delta.0..f}
\Big|\sum_{N<n\le M}\frac{  \vartheta_n }{ \nu_n }\Big| 
 &\le &    \big|\log \nu_M -\log \nu_N \big| .\end{eqnarray}
  \end{proof}
  
  %\begin{corollary}\label{omega.cor} Assume that the sequence $\{\t_n, n\ge 1\}$ is decreasing. Then we have the following estimate 
% \ben  \o(m) &\le&   C_\t\,\log \nu_{m}.\een \end{corollary}
  
%   \begin {proof} Let $\bar m$ be the larger positive integer such that $\nu_{\bar m} <2\nu_m$. 
%   If the sequence $\{\t_n, n\ge 1\}$ is decreasing, then by Lemma \ref{hlp}, assertion (ii),
%\ben  
%\o(m) &\le& C_\t\,  \sum_{ h\ge 1  \atop     \nu_m< \nu_{m+h} <2\nu_m} \,     \frac {\vartheta_{h}}{ \nu_{h}  }  \le C_\t\,  \sum_{ 1\le k\le \bar m -m  } \,     \frac {\vartheta_{h}}{ \nu_{h}  }  \,\le \, \big|\log \nu_{\bar m -m} -\log \nu_1 \big|
%\cr &  \le &   C_\t\,\log \nu_{\bar m -m}\,\le \, C_\t\,\log \nu_{m}.
%\een
 % \end{proof}
% \vskip 5 pt   \vskip 10 pt  
 
%\vskip 20 pt 
\subsection{Associated functions $\boldsymbol k$, $\boldsymbol M$, $\boldsymbol \Phi$   and sequence $  \mathcal M$}
% We pass to the main result of this Section. 
  
Let  
\beq M(t)= \sum_{1\le \nu_n<t }\,{\vartheta_m\over \s_m\sqrt{\nu_m} },  \qq t\ge 1
\eeq 
  and set for $x\ge M(1)$,
$ M^{-1}(x)=\sup\{ t\ge 1: M(t)\le x\}$.
       \vskip 3 pt              
 Put for any positive integer 
$i$,
\begin{equation}\label{mi}     m_i=\sum_{2^i\le \nu_n<2^{i+1}}{\vartheta_n\over \s_n\sqrt{\nu_n} }.  \end{equation}

 \vskip 5 pt  
 Let  the sequence $\mathcal M$ be defined as follows,
\begin{equation}\label{Mi}         \mathcal M=\big\{M_J , J\ge 1\big\} \qq {\rm where}\quad     M_J=M(2^J)
% \sum_{1\le \nu_n<2^J }{\vartheta_n\over \s_n\sqrt{\nu_n} }
=\sum_{1\le i<J}  m_i,\quad J\ge 1.\end{equation}
\vskip 2 pt   
\noi  The almost sure local limit theorem  in view,  will tightly rely upon     the asymptotic distribution  of   the sequence $\mathcal M$. 
 \vskip 3 pt It will be deduced from the study of  the   asymptotic almost sure behavior of the sequence 
\beq \label{zi} \mathcal Z= \{Z_i, i\ge 1\}\qq {\rm where\ for \ each} \ i, \quad Z_i=\sum_{2^i\le \nu_n<2^{i+1}}{\vartheta_nY_n\over  
 \s_n\sqrt{\nu_n}} , 
 \eeq
 recalling that $Y_n= \s_n  ({\bf 1}_{\{S_n=\kappa_n\}}-{\mathbb P}\{S_n=\kappa_n\})$. 
We will apply Theorem \ref{8.2} with the choices   $\xi_l=Z_l$ and  $u_l= m_l$, $l\ge 1$. We thus need assumptions \eqref{8.12}, \eqref{little.s} to be fulfilled.  
\vskip 3 pt Introduce function $k(x)$ 
\beq k(x) =h(x) \o(x), \qq\quad x\ge 1.
\eeq
\vskip 5 pt 
Now  let   
 the non-decreasing   function  $\Phi$ be      defined as follows:
\begin{eqnarray}\label{Phi}
     \Phi (x )  =k\circ M^{-1}(x).
 \end{eqnarray}%\begin{eqnarray}\label{Phi}
%     \Phi (x )  = \begin{cases}& \max\Big(h(2^{J+1})
 %\,  \k(2^{I+1}) 
%   \ ,\  
 %\displaystyle{ \max_{m\le 2^{J+1}}(1,\o(m))}\Big) ,\quad  \hbox{ if $x=M_J$, $J\ge %1$,}   \cr 
%  &
%\hbox{and $\Phi$ is piecewise linear.}
%\end{cases}
%  \end{eqnarray}
 %Let $M_J\in[M^l,M^{l+1}[$
 %{\color{blue}\begin{equation*}
 % \E \big|\sum_{i\le l\le j}\xi_l\big|^\g \le  \Phi(\sum_{  l\le j}u_l)\Psi(\sum_{i\le l\le j} u_l)  \qq (\forall 1\le i\le j<\infty),\end{equation*}
%\begin{equation*} s^\g=s_{M}^\g =\sum_{l\ge 1\atop [M^l,M^{l+1}[\cap \mathcal U\neq\emptyset} \frac{\Psi(M^{ l  })\,\Phi(M^{l })\,L(l)^{\e(\eta)}}{  M^{ \g l  }} \end{equation*}} 

%It will be necessary to have at disposal an estimate of $M_J$  in order to control the convergence of the series
%$$s^\g  =\sum_{l\ge 1\atop [M^l,M^{l+1}[\cap \mathcal U\neq\emptyset} \frac{\Psi(M^{ l  })\,\Phi(M^{l })\,L(l)^{\e(\eta)}}{  M^{ \g l  }} $$

%$$\Phi(M^{l })\le \Phi (M_J )=\max\big(h(2^{J+1})   \, ,\,    \max_{m\le 2^{J+1}}(1,\o(m)) \big).$$
  \vskip 10 pt

   \begin{example}\label{example4}
  Assume that $h(x)=\mathcal O(1)$, and  $\o_X^*=\mathcal O(1)$ (condition \eqref{omega.star}), then 
$$\Phi(x) =\mathcal O(1).$$
In Examples \eqref{example3.},   \eqref{cramer},    condition \eqref{omega.star} is satisfied. 

%\vskip 3 pt 
%(ii) If $h(x)=\mathcal O(1)$, then
%\beq  \Phi\big(M_J\big)\,\le\, \max_{m\le 2^{J+1}}(1,\o(m)).\eeq
%ZZZZ If the sequence $\{\t_n, n\ge 1\}$ is decreasing, by Corollary \ref{omega.cor},  
%$$\max_{m\le 2^{J+1}}(1,\o(m))\le  C_\t\,\log \nu_{2^{J+1}},$$ 
%and so we   have, 
%noticing that $\nu_n\le n$,
%$$\Phi\big(M_J\big)\,\le\,   C_\t\,\log \nu_{2^{J+1}}
% \,\le\,C_\t\, J
%.$$
%\vskip 2 pt (iii) If   for  some $0\le \a<1$,  $h(x)\le (\log x)^\a $,
 % then 
%$  h(2^{2^{ J  }+1})\le C\,2^{  \a J  }$. Further if   $\{\t_n, n\ge 1\}$ is decreasing,  \beq  \Phi\big(M_J\big)\,\le\, C_\t\, \max\big(2^{  \a J  } ,\log  \nu_{2^{J+1}}\big)\,\le\, C_\t\, \max\big(2^{  \a J  } ,J\big)=C_\t\,  2^{  \a J  }. \eeq
% \vskip 2 pt (iv) If  condition \ref{omega.star}
%is satisfied, then 
%\beq\Phi\big(M_J\big)  \le \, h(2^{J+1}).
%\eeq
%Moreover if $h(x) =\mathcal O(1)$
%FAUX, then in that case, as we will prove in Theorem \ref{asllt.quasi.ortho}, the sequence 
%$$\Big\{\frac{Z_i}{m_i},\, i\ge 1\Big\},$$  is  a quasi-orthogonal system, which is a strong and achieved result.
 \end{example}

\vskip 10 pt
\subsection{The key metrical estimate.} We   control the $L^2$-norm of increments $\sum_{I\le i < J  }    Z_i$ and prove the following   estimate.
\begin{theorem}\label{inc} Assume that 
\beq m_i=\mathcal O(1), \qq i\to\infty.\eeq
There exists a constant $C_X$ such that for any $J\ge I\ge 1$,
 \begin{eqnarray} \label{estim.A2}
\E\Big|\sum_{I\le i < J  }    Z_i\Big|^2  &\le  &   C_{X} \,h(2^{J+1})\,    \sum_{I\le i \le J }\, \Big(\sum_{2^{  i}  <\nu_m\le 2^{i+1} }  {\vartheta_m  \,\max(1,\o(m) ) \over \s_m  \sqrt{\nu_m }  }   \Big) .\end{eqnarray} 

 In particular,  \begin{eqnarray} \label{estim.A2.enonce}
\E\Big|\sum_{I\le i < J  }    Z_i\Big|^2  &\le  &    C_{X} \,\Phi\big(M_J\big) \,   \Big( \sum_{I\le i \le J }\, m_i\Big)
%\, \Big(\sum_{2^{  i}  <\nu_m\le 2^{i+1} }  {\vartheta_m    \over \s_m  \sqrt{\nu_m }  }   \Big)
    .\end{eqnarray}   
\end{theorem}  
 \begin{proof}Let $0<c<1$. By Corollary \ref{cort1}, and assumption \eqref{nun},  there   exists a constant
$C_{c}  $  such that for all $1\le \nu_m\le c\nu_n$,$$\big|{\mathbb E\,}   Y_nY_m\big|  \,\le \,     C_{c,D} \,h(n)\,  \sqrt{\nu_m\over \nu_n}.$$
 Recall that
  $$ {\mathbb E\,} Y_n^2=\s_n^2
{\mathbb P}\{S_n=\kappa_n\}\big(1-{\mathbb P}\{S_n=\kappa_n\} \big) ={\mathcal O}(\s_n)
 $$  by assumption \eqref{nun.ka}-(2). 
 %Thus by Cauchy-Schwarz's inequality, for all $n$ and $m$
% \begin{equation}\label{nun.ka2} |{\mathbb E\,}   Y_nY_m |\,\le \,C\sqrt{\s_n\s_m}. \end{equation}  

 \vskip 3 pt {\it Step 1.} We first bound the rectangle sums 
  \beq\label{rectangle.sum}\sum_{I\le i< j-1  < J }| \E Z_iZ_j|.
  \eeq
 Let $1\le i<j -1$.  Then
\begin{eqnarray}\label{rectangle.sum.est1} | \E Z_iZ_j|&\le &\sum_{2^i\le \nu_m<2^{i+1}\atop
2^j\le \nu_n<2^{j+1}}  {\vartheta_m\vartheta_n\over\s_m\sqrt{\nu_m}\s_n\sqrt{\nu_n}} \ |\E Y_nY_m|
\cr &\le &C_{D} \,h(2^{j+1})\, \sum_{2^i\le \nu_m<2^{i+1}\atop
2^j\le \nu_n<2^{j+1}}\ {\vartheta_m\vartheta_n\over \s_n\s_m\sqrt{\nu_m\nu_n}} \,   \sqrt{\nu_m\over \nu_n}
\cr & = &C_{D} \,h(2^{j+1})\, \sum_{2^i\le \nu_m<2^{i+1} }{\vartheta_m \over \s_m \sqrt{\nu_m}  } \Big(\sqrt{\nu_m} \sum_{ 
2^j\le \nu_n<2^{j+1}}{ \vartheta_n\over  \s_n \nu_n   }\Big).
\end{eqnarray} 
Noting that $\nu_n\ge \nu_m$ implies $n\ge m$  by assumption \eqref{nun2},
%since all $\t_u$ are positive,
 so that $\nu_n \ge 2^j$ and $\nu_m<2^{i+1}$, imply $n\ge m$, we have
\begin{eqnarray} \label{estim.A1}
& & \sum_{I\le i< j -1< J }| \E Z_iZ_j|
\cr &\le  &   C_{D} \,h(2^{J+1})\, \sum_{I\le i< j -1< J}\, \sum_{2^i\le \nu_m<2^{i+1} }{\vartheta_m\over \s_m \sqrt{\nu_m}  } \Big(\sqrt{\nu_m} \sum_{ 
2^j\le \nu_n<2^{j+1}}{\vartheta_n\over  \s_n \nu_n   }\Big)
 \cr &\le  &   C_{D} \,h(2^{J+1})\, \sum_{I\le i \le J }\, \sum_{2^i\le \nu_m<2^{i+1} }{\vartheta_m\over \s_m \sqrt{\nu_m}  } \Big(\sqrt{\nu_m} \sum_{ 
n\ge m }{\vartheta_n\over  \s_n \nu_n   }\Big)
\cr &\le  &   C_{D} \,h(2^{J+1})\, \Big(\max_{m\le 2^{J+1}}\sqrt{\nu_m} \sum_{ 
n\ge m }{\vartheta_n\over  \s_n \nu_n   }\Big) \Big(\sum_{I\le i \le J }\, m_i \Big)
    .\end{eqnarray} 
{\color{black} To control the mid-term,  we appeal to   Lemma \ref{hlp}.
By applying it with  $\d= \frac12$,  we get
\beq  \label{midterm}\nu_m^{1/2}  \sum_{ 
n\ge m }{\vartheta_n\over  \s_n \nu_n   }\,\le \, C \sqrt{\nu_m} \sum_{ 
n\ge m }{\vartheta_n\over    \nu_n^{3/2}   }\,\le \, C\,\nu_m^{1/2}\  \frac{1}{\nu_{m }^{1/2} } = C  .
\eeq
Therefore
\begin{eqnarray} \label{estim.A1f}
 \sum_{I\le i< j -1< J  }| \E Z_iZ_j|  &\le  &   C_{D} \,h(2^{J+1})\,   \Big(\sum_{I\le i \le J }\, m_i \Big)
%\cr &= & C_{c,D} \,h(2^{J+1})\,  \Big( \sum_{I\le i \le J }m_i \Big)
    .\end{eqnarray} }
     
    \vskip 3 pt {\it Step 2.}  Now we bound the square sum
    \beq\label{square.sum} \sum_{I\le i \le J }  \E Z_i^2.
    \eeq
   We have 
  \begin{eqnarray} \E |Z_i|^2&=& \E\Big|\sum_{2^i\le \nu_m<2^{i+1} }  {\vartheta_m Y_m\over \s_m \sqrt{\nu_m}  } \Big|^2
  \cr &=& 2 
   \sum_{2^i\le \nu_m<\nu_n<2^{i+1} }  {\vartheta_m \vartheta_n\E Y_mY_n\over \s_m\s_n \sqrt{\nu_m\nu_n}  }+ \sum_{2^i\le \nu_m<2^{i+1} }  {\vartheta_m \E Y^2_m\over \s^2_m  \nu_m   } 
   \cr &:=& I_1+ I_2.
  \end{eqnarray}
  At first 
  % (taking $\d=1$ in \eqref{delta} and letting $N+1$ be the smallest integer $m$ such that $\nu_m\ge 2^i$),
    we deduce from \eqref{basic0},
   \begin{eqnarray*}I_2\,=\, \sum_{2^i\le \nu_m<2^{i+1} }  {\vartheta_m \E Y^2_m\over \s^2_m  \nu_m   } \,\le \, C\sum_{2^i\le \nu_m<2^{i+1} }  {\vartheta_m  \over \s_m  \nu_m   } \,\le \, C\sum_{2^i\le \nu_m<2^{i+1} }  {\vartheta_m  \over \s_m  \sqrt{\nu_m}   }.
 %   \cr & \le  & C\sum_{2^i\le \nu_m<2^{i+1} }  {\vartheta_m  \over  \nu^2_m   }\le C\sum_{m>N}  {\vartheta_m  \over  \nu^2_m   }\le \, {C \over  \nu_N   }\le \, {C \over  2^i  }.
    \end{eqnarray*}
Now
 \begin{eqnarray}\label{I1bdd}
 |I_1|&\le &2 
   \sum_{2^i\le \nu_m<\nu_n<2^{i+1} }  {\vartheta_m \vartheta_n|\E Y_mY_n|\over \s_m\s_n \sqrt{\nu_m\nu_n}  }
    \end{eqnarray}
% $$=====$$
By \eqref{basic.corr.bd}
$$  |{\mathbb E\,} Y_mY_n|     \le  C_\t\, \Big(\frac {\s_n}{\sqrt{\s_{n}^2-\s_{m}^2 } } +1\Big) ,\qq\quad (n>m\ge 1)
$$
%\begin{align}
so that
% voir apres end{document} le calcul detaille en (1)
\begin{eqnarray}\label{l1bound}
&&   \sum_{2^i\le \nu_m<\nu_n<2^{i+1} }  {\vartheta_m \vartheta_n|\E Y_mY_n|\over \s_m\s_n \sqrt{\nu_m\nu_n}  }
\cr  &\le & C_\t\,\sum_{2^i\le \nu_m<\nu_n<2^{i+1} }  {\vartheta_m \vartheta_n  \over \s_m\s_n \sqrt{\nu_m\nu_n}  }\Big(\frac {\s_n}{\sqrt{\s_{n}^2-\s_{m}^2 }  } +1\Big)
  \cr  &\le & {\color{black}C_\t\,\sum_{2^{  i}  <\nu_m\le 2^{i+1} }  {\vartheta_m    \over \s_m  \sqrt{\nu_m }  }  \ \Big\{\sum_{      \nu_m< \nu_n <2^{i+1}} \, {  \vartheta_n  \over     \sqrt{\nu_n }  } \, \Big( \frac {1}{\sqrt{\s_{n}^2-\s_{m}^2 }  } \Big)}  \Big\}
\cr & &\ + C_\t\,\Big(\sum_{2^i\le \nu_m <2^{i+1} }  {\vartheta_m    \over \s_m  \sqrt{\nu_m }  }\Big)^2
\cr  &= &   C_\t\,\Big(\sum_{2^{  i}  <\nu_m\le 2^{i+1} }  {\vartheta_m  \,\o(m)  \over \s_m  \sqrt{\nu_m }  }   \Big)   
 \ + C_\t\,\Big(\sum_{2^i\le \nu_m <2^{i+1} }  {\vartheta_m    \over \s_m  \sqrt{\nu_m }  }\Big)^2
 \cr  &\le &   C_\t\,\Big(\sum_{2^{  i}  <\nu_m\le 2^{i+1} }  {\vartheta_m  \,\max(1,\o(m) ) \over \s_m  \sqrt{\nu_m }  }   \Big)   ,
\end{eqnarray}
 since $  \sum_{2^i\le \nu_m <2^{i+1} }  {\vartheta_m    \over \s_m  \sqrt{\nu_m }  }=\mathcal O(1)$ by assumption.  
 
  \vskip 3 pt 
 We therefore get
\begin{align}\label{sum.squared.}
 \E Z_i^2  &\,\le \, C_\t\,\Big(\sum_{2^{  i}  <\nu_m\le 2^{i+1} }  {\vartheta_m  \,\max(1,\o(m) ) \over \s_m  \sqrt{\nu_m }  }   \Big),
   \end{align} 
and   \begin{align}\label{sum.squared..}
 \sum_{I\le i \le J }  \E Z_i^2  &\,\le \, C_\t\,\sum_{I\le i \le J } \Big(\sum_{2^{  i}  <\nu_m\le 2^{i+1} }  {\vartheta_m  \,\max(1,\o(m) ) \over \s_m  \sqrt{\nu_m }  }   \Big).   \end{align}  
 
   Further  
\begin{eqnarray}\label{sum.i,i+1}\E| Z_iZ_{i+1}|\le \|Z_i\|_2\|Z_{i+1}\|_2&\le &\frac12 \sum_{h\in\{i,i+1\}}\|Z_h\|_2^2
\cr&\le &  C_\t\,\Big(\sum_{2^{  i}  <\nu_m\le 2^{i+1} }  {\vartheta_m  \,\max(1,\o(m) ) \over \s_m  \sqrt{\nu_m }  }   \Big) . 
\end{eqnarray}

%\begin{remark}[$M$-geometric blocks] voir aprÃšs end{document}
 %Consider the special cases.placed after \end{document} (1a)

 \vskip 3 pt {\it Step 3.}   
 %Recalling that $$m'_i =\max(1,\o(m) ) ,$$
By combining \eqref{sum.squared..} and \eqref{sum.i,i+1} with estimate \eqref{estim.A1f}
we get, 
\begin{eqnarray} \label{estim.A2f}
\E\Big|\sum_{I\le i < J  }    Z_i\Big|^2  &\le  &   C_{D} \,h(2^{J+1})\,    \sum_{I\le i \le J }\, \Big(\sum_{2^{  i}  <\nu_m\le 2^{i+1} }  {\vartheta_m  \,\max(1,\o(m) ) \over \s_m  \sqrt{\nu_m }  }   \Big)%\cr &= & C_{c,D} \,h(2^{J+1})\,  \Big( \sum_{I\le i \le J }m_i \Big)
\cr &\le  &   C_{D} \,h(2^{J+1})\,\Big(\max_{m\le 2^{J+1}}(1,\o(m) \Big)\,    \sum_{I\le i \le J }\, \Big(\sum_{2^{  i}  <\nu_m\le 2^{i+1} }  {\vartheta_m    \over \s_m  \sqrt{\nu_m }  }   \Big)
 \cr &=  &   C_{D} \,\Phi\big(M_J\big) \,   \Big( \sum_{I\le i \le J }\,  m_i   \Big)
    .\end{eqnarray}    
%   If the sequence $\{\t_n, n\ge 1\}$ is decreasing,   by Corollary \ref{omega.cor},  $\o(m)  \le    C_\t\,\log \nu_{m}$, so that in this case we   get \begin{eqnarray} \label{estim.A2.decrease}
%\E\Big|\sum_{I\le i < J  }    Z_i\Big|^2  &\le  &   C_{D} \,h(2^{J+1})\,    \sum_{I\le i \le J }\, \Big(\sum_{2^{  i}  <\nu_m\le 2^{i+1} }  {\vartheta_m  \,\log \nu_{m} \over \s_m  \sqrt{\nu_m }  }   \Big)%\cr &= & C_{c,D} \,h(2^{J+1})\,  \Big( \sum_{I\le i \le J }m_i \Big)
 %   .\end{eqnarray}  

At this stage we make a useful observation.
 
\begin{remark}[$R$-geometric blocks]  \label{R.geometric.blocks} 
Let $R>1$ and consider the sequence $\mathcal Z(R)=\{Z_{R,i},i\ge 1\}$ of sums of $R$-geometric blocks defined as follows,
\beq \label{zi.m} Z_{R,i}=\sum_{R^i\le \nu_n<R^{i+1}}{\vartheta_nY_n\over  
 \s_n\sqrt{\nu_n}} , \qq\qq i\ge 1.
 \eeq
Correspondingly, we note
\beq \label{mi.m}  m_{R,i}=\sum_{R^i\le \nu_n<R^{i+1}}{\vartheta_n\over \s_n\sqrt{\nu_n} }, \qq\qq i\ge 1,
 \eeq 
 and   $\mathcal M_R$ the sequence 
\begin{equation}\label{miR}         \mathcal M_R=\big\{M_{R,J} , J\ge 1\big\} \qq {\rm where}\quad     M_{R,J}=M(R^J)
=\sum_{1\le i<J}  m_{R,i},\quad J\ge 1.\end{equation}
For controlling rectangle sums \eqref{rectangle.sum} of $2$-geometric blocks, 
$$\sum_{I\le i< j-1  < J }| \E Z_iZ_j|,$$
we have used Corollary \ref{cort1}, and assumption \eqref{nun}.  Next we controlled separately $\E| Z_iZ_{i+1}|$ in \eqref{sum.i,i+1}, by using Cauchy-Schwarz's inequality and an elementary inequality. 
 If in place we had rectangle sums  of $R$-geometric blocks, with $1<R<2$, then   this can be done quite similarly by first
   controlling the modified rectangle sums
 \beq\label{R.rectangle.sum}\sum_{I\le i< j-\d(R)  < J }| \E Z_iZ_j|,
  \eeq
for some suitable integer $\d(R)>1$ depending on $R$ only.  Next we can control  $\E| Z_iZ_{i+h}|$, $h=1,\ldots, \d(R)$ correspondingly as before.  Finally we control the square sums of $R$-geometric blocks the same way as in {\it Step 2} for the sum \eqref{square.sum}.
%\beq \sum_{I\le i \le J }  \E Z_i^2.
 %   \eeq
%If $\mathcal Z(M)$ is  a quasi-orthogonal system, then for all $1<M'\le M$, $\mathcal Z(M')$ is  a quasi-orthogonal system
Consequently we control  by  using the same proof the sums of $R$-geometric blocks. For simplicity of the notation 
%and for the clarity of the exposition, we chosed to 
we only write it   for $2$-geometric blocks.\end{remark}

By Remark \ref{R.geometric.blocks}, we also have
\begin{eqnarray} \label{estim.A2.R}
\E\Big|\sum_{I\le i < J  }   Z_{R,i}\Big|^2  &\le  &  
    C_{D, R} \,\Phi\big(M(R^J)\big) \,   \Big( \sum_{I\le i \le J }\,  m_{R,i}  \Big)
    .\end{eqnarray}    \end{proof}

\vskip 15 pt
{
%\color{blue} 
\noi {\bf Subsequence  case.} Let $\mathcal N\subseteq \N$ be an increasing sequence of positive integers and consider the $\mathcal N$-restricted blocks
\beq \label{subseq.Z}Z_i^\mathcal N= \sum_{2^i\le  \nu_n<2^{i+1}\atop n\in \mathcal N }{\t_nY_n\over \s_n\sqrt{\nu_n}}, \qq\quad i\ge 1.
\eeq
Let also 
\beq \label{subseq.m}m_i^\mathcal N= \sum_{2^i\le  \nu_n<2^{i+1}\atop n\in \mathcal N }{\t_n \over \s_n\sqrt{\nu_n}}, \qq\quad i\ge 1,
\eeq
and define correspondingly $\mathcal M^\mathcal N$, and $\mathcal M_j^\mathcal N$, $j\ge 1$, as before.
\begin{theorem}[Subsequences]\label{inc.subseq} Assume that 
\beq m_i^\mathcal N=\mathcal O(1), \qq i\to\infty.\eeq
There exists a constant $C_X$ such that for any $J\ge I\ge 1$,
 \beq \label{estim.A2subseq}
\E\Big|\sum_{I\le i < J  }    Z_i^\mathcal N\Big|^2  \,\le  \,   C_{X} \,h(2^{J+1})\,    \sum_{I\le i \le J }\, \Big(\sum_{2^{  i}  <\nu_m\le 2^{i+1}\atop m\in \mathcal N }  {\vartheta_m  \,\max(1,\o(m) ) \over \s_m  \sqrt{\nu_m }  }   \Big) .\eeq 

 In particular,  
 \begin{eqnarray} \label{estim.A2.enonce.subseq}
\E\Big|\sum_{I\le i < J  }    Z_i^\mathcal N\Big|^2  &\le  &    C_{X} \,\Phi\big(M_J\big) \,   \Big( \sum_{I\le i \le J }\, m_i^\mathcal N\Big)
%\, \Big(\sum_{2^{  i}  <\nu_m\le 2^{i+1} }  {\vartheta_m    \over \s_m  \sqrt{\nu_m }  }   \Big)
    .\end{eqnarray}   
\end{theorem} 
 It is of matter to indicate that applications to corresponding subsequence-ASLLT's, will require the  following modification of  assumption \ref{midiv} to hold:   \begin{eqnarray}\label{midiv.subseq}
   \sum_{ 
n\in \mathcal N}{\vartheta_n\over \s_n\sqrt{\nu_n} } =\infty,  \end{eqnarray}
    also that $m_i^\mathcal N$ might tend to 0, even in the i.i.d. case, and (see \eqref{Y}) that $\kappa_n$ \underline{defines} $Y_n$.
   
   \begin{proof}  It is a simple adaptation of the previous proof. In \eqref{rectangle.sum.est1}, we bounded the  correlation $| \E Z_iZ_j|$ by the sum
$$\sum_{2^i\le \nu_m<2^{i+1}\atop
2^j\le \nu_n<2^{j+1}}  {\vartheta_m\vartheta_n\over\s_m\sqrt{\nu_m}\s_n\sqrt{\nu_n}} \ |\E Y_nY_m|
,$$
next $ \sum_{I\le i< j -1< J }| \E Z_iZ_j|$ by the sum
$$\sum_{I\le i< j -1< J } \sum_{2^i\le \nu_m<2^{i+1}\atop
2^j\le \nu_n<2^{j+1}}  {\vartheta_m\vartheta_n\over\s_m\sqrt{\nu_m}\s_n\sqrt{\nu_n}} \ |\E Y_nY_m|
.
$$
which is shown in \eqref{estim.A1} to be bounded    by
  $    \sum_{I\le i \le J }\, m_i  $.
Similarly, in \eqref{I1bdd},  we controlled $ |I_1|$ by the sum
$$\sum_{2^i\le \nu_m<\nu_n<2^{i+1} }  {\vartheta_m \vartheta_n|\E Y_mY_n|\over \s_m\s_n \sqrt{\nu_m\nu_n}  }
$$
with  terms $|\E Y_mY_n|$, which is  next bounded in \eqref{l1bound}, by
$$ C_\t\,\Big(\sum_{2^{  i}  <\nu_m\le 2^{i+1} }  {\vartheta_m  \,\max(1,\o(m) ) \over \s_m  \sqrt{\nu_m }  }   \Big).$$
The control of the sum $I_2$ raises no peculiar problem. It is clear that all this is transferable {\it ipso facto} to the subsequence case by pointing at each place where $\nu_n$ appears, that $n\in \mathcal N$, and next collecting the modified bounds.
\end{proof}}

% \vskip 10 pt

  %  \vskip 6 pt
    \subsection{The ASLLT} We    now are in  position to state the main result of this Section.

  %So is in particular the case if  condition \eqref{bdd} is satisfied.
 %  \medskip\par     Next put for any positive integer 
%$i$, \cc
 %\begin{equation}\label{mi.prime}     m'_i=\sum_{2^i\le m<2^{i+1} }{\vartheta_m\,\max(1,\o(m))\over \s_m\sqrt{\nu_m} } .\end{equation}

\begin{theorem} \label{t1[asllt].} 
 Let  $\{\t_n, n\ge 1\}$ be   chosen so that \eqref{unif.bound.varthetaj} holds. 
   Suppose that 
   %assumption
   \begin{eqnarray}\label{midiv}
   \hbox{\it The series $ \sum_{ 
n\ge 1}{\vartheta_n\over \s_n\sqrt{\nu_n} }$ is divergent.}
  \end{eqnarray} We also assume that   the following condition is fulfilled:  
\begin{equation}\label{s2} \hbox{\it The series
$  s^2:= \displaystyle{\sum_{l\ge 1\atop 
[2^l,2^{l+1}[\cap \mathcal M\neq\emptyset}}
\frac{\Phi(2^{ l  })\, \left(1+ \log \sharp\{[2^l,2^{l+1}[\cap \mathcal M\}\right)}{  2^{ l  }  }$\  is convergent.}
\end{equation}
  Let  $\k_n\in \mathcal L(nv_0,D)$, $n=1,2,\ldots$   be a  sequence of integers   such that    \eqref{nun.ka} holds.
    
   \vskip 4 pt  {\rm (i)} We have (recalling that  $M_j= \sum_{1\le n<2^j }{\vartheta_n\over \s_n\sqrt{\nu_n} }$),
   $$
\bigg\|
%\sum_{k\ge 1} %\sum_{1\le m<2^{j}=M_j
\sum_{k\ge 1\atop 
[2^k,2^{k+1}[\cap \mathcal M\neq\emptyset}
  \sup_{2^k\le M_j  
<2^{k+1}} 
\bigg\{ \frac{\big|\sum_{1\le n<2^{j}}{   \vartheta_n(  {\bf 1}_{\{S_n=\kappa_n\}}-{\mathbb P}\{S_n=\kappa_n\} )
  \over 
 \sqrt{\nu_n}}\big|}{M_j }
       \bigg\} \bigg\|_2\le K  s.$$
    
          Further    
%    \begin{equation}\label{Ya}
% \lim_{j\rightarrow\infty} \frac{\sum_{1\le n<2^{j}}  { \vartheta_n({\bf 1}_{\{S_n=\kappa_n\}}-  {\mathbb P}\{S_n=\kappa_n\})   \over  \sqrt{\nu_n}}}{ M_j  }      \, \buildrel{\rm a.s.}\over = 0 .
%\end{equation}
 \begin{equation}\label{Ya1}
 \lim_{j\rightarrow\infty} {1\over \sum_{1\le n<2^j}
 {\t_n\over \s_n\sqrt{\nu_n} } }\sum_{1\le n<2^{j}} { \t_n({\bf 1}_{\{S_n=\kappa_n\}}-  {\mathbb P}\{S_n=\kappa_n\} )  \over  \sqrt{\nu_n}}  \, \buildrel{\rm a.s.}\over = 0 .
\end{equation}
 {\rm (ii)} Assume that
 \begin{equation}\label{2}  \lim_{n\to \infty}  \s_n {\mathbb P}\{S_n=\k_n\}=
 \g, \qq\qq (0<\g<\infty)
  .\end{equation}
Then   we have
 \begin{equation}\label{2.mean} 
   \lim_{j\rightarrow\infty}  {1\over \displaystyle{\sum_{1\le n<2^j}{\vartheta_n\over \s_n\sqrt{\nu_n} }}}\ \sum_{1\le n<2^j}{  \vartheta_n {\bf 1}_{\{S_n=\kappa_n\}}  \over   \sqrt{\nu_n}} \, \buildrel{\rm a.s.}\over =\g
     .
\end{equation}
%In particular if for any sequence $\k_n$ such that
 %\begin{equation}\label{2..}  \lim_{n\to \infty} { \k_n-a_n   \over    \s_n  }= \k ,
%\end{equation}   we have 
%\begin{equation}\label{2...}  \lim_{n\to \infty}  \s_n {\mathbb P}\{S_n=\k_n\}=
 %{D\over  \sqrt{ 2\pi} }e^{-{ \k ^2\over  2   } }
 %,\end{equation}

 \vskip 4 pt  {\rm (iii)} Assume that   
 \begin{equation}  \lim_{n\to \infty}  \s_n {\mathbb P}\{S_n=\k_n\}={D\over  \sqrt{ 2\pi} }e^{-
{ \k ^2\over  2   } },\end{equation}
and that $M(t)$ is slowly varying near infinity. 
 Then   the   ASLLT holds,
\begin{equation} \label{Ya.}
   \lim_{N\rightarrow\infty}  {1\over \displaystyle{\sum_{1\le n<N}{\vartheta_n\over \s_n\sqrt{\nu_n} }}}\ \sum_{1\le n<N}{ \vartheta_n  {\bf 1}_{\{S_n=\kappa_n\}}  \over   \sqrt{\nu_n}} \, \buildrel{\rm a.s.}\over ={D\over  \sqrt{ 2\pi} }e^{-{ \k ^2\over  2   } }
   .\end{equation}
\vskip 4 pt  {\rm (iv)} In particular, if the LLT is applicable to the sequence $X$, then for any sequence $\{\kappa_n, n\ge 1\}$ such that 
$$\lim_{n\to \infty}\frac{\k_n-a_n}{\s_n}=\kappa. $$ 
In   addition if $M(t)$ is slowly varying near infinity, then 
 \eqref{Ya.} holds.
  \vskip 4 pt  {\rm (v)}
  {\rm (Optimality of Assumption \eqref{midiv})} Assume that   $\Phi(x) = \mathcal O(1)$ and \begin{eqnarray}\label{miconv1}
   \sum_{n=1}^\infty {\t_n(\log\log n)^2\over \s_n\sqrt{\nu_n} }
<\infty  \,.\end{eqnarray}
Then     a     strong but  different almost sure convergence result takes place, namely that the sequence 
$$\sum_{n< 2^J}  {{\bf 1}_{\{S_n=\kappa_n\}}-{\mathbb P}\{S_n=\kappa_n\}\over  
  \sqrt{\nu_n}}\qq J=1,2,\ldots,$$ 
  converges  almost surely.

\end{theorem}
 
 %    %In contrast with the proof of Theorem \ref{t2[asllt]}, the proof is much shorter.

\begin{proof}%[Proof of Theorem \ref{t1[asllt].}] 
  By Theorem \ref{inc} the sequence $\mathcal Z$ satisfies  condition \eqref{8.12}  with the choice $\g=2$, $u_l=m_l$,  $\Psi(x)=Cx$ for some suitable  finite, positive constant $C$, and $\Phi(x)$ defined in \eqref{Phi}. In view of assumption \eqref{s2}, 
 %  the series
%$$  s^2:= \displaystyle{\sum_{l\ge 1\atop 
%[2^l,2^{l+1}[\cap \mathcal M\neq\emptyset}}
%\frac{\Phi(2^{ l  })\, \left(1+ \log \sharp\{[2^l,2^{l+1}[\cap \mathcal M\}\right)}{  2^{ l  }  },$$   is convergent, and so 
Theorem \ref{8.2} is in force. 
\vskip 2 pt
Therefore $${\mathcal S}^2
=\sum_{k\ge 1}
\sup\bigg\{   \frac{\big|\sum_{1\le \nu_n<2^{j+1}}{\t_nY_n\over 
 \s_n\sqrt{\nu_n}}\big|}{\sum_{1\le \nu_n<2^{j} }{\t_n\over \s_n\sqrt{\nu_n} }
}\,:\,   2^k\le\sum_{1\le \nu_n<2^{j} }{\t_n\over \s_n\sqrt{\nu_n} }
<2^{k+1}  \bigg\}^2,$$
satisfies  
$$ \|{\mathcal S}\|_2 \le K  s ,\qq \hbox{and  
}\qq
 \lim_{j\rightarrow\infty} {1\over \sum_{1\le \nu_n<2^{j} }{\t_n\over \s_n\sqrt{\nu_n} } }\sum_{1\le i<j}\sum_{2^i\le \nu_n<2^{i+1}}{\t_nY_n\over 
 \s_n\sqrt{\nu_n}} \, \buildrel{\rm a.s.}\over = 0.$$
 Recalling \eqref{Y}, $Y_n= \s_n  ({\bf 1}_{\{S_n=\kappa_n\}}-
 {\mathbb P}\{S_n=\kappa_n\} )$,   the latter means 
 \begin{equation}\label{Ya1.}
 \lim_{j\rightarrow\infty} {1\over \sum_{1\le \nu_n<2^{j} }{\t_n\over \s_n\sqrt{\nu_n} } }\sum_{1\le \nu_n<2^j}    { \t_n({\bf 1}_{\{S_n=\kappa_n\}}-  {\mathbb P}
 \{S_n=\kappa_n\})  \over  \sqrt{\nu_n}}  \, \buildrel{\rm a.s.}\over = 0 .
\end{equation}
 In particular, \eqref{2.mean} is immediate. 
Now if 
%\eqref{llt1} 
 \begin{equation}  \lim_{n\to \infty}  \s_n {\mathbb P}\{S_n=\k_n\}={D\over  \sqrt{ 2\pi} }e^{-
{ \k ^2\over  2   } },\end{equation}
 then\begin{equation}\label{Yb}
   \lim_{j\rightarrow\infty} {1\over \sum_{1\le \nu_n<2^{j} }{\t_n\over \s_n\sqrt{\nu_n} }}\sum_{1\le \nu_n<2^{j}}{\t_n   {\bf 1}_{\{S_n=\kappa_n\}}  \over   \sqrt{\nu_n}} \, \buildrel{\rm a.s.}\over =
  {D\over  \sqrt{ 2\pi} }e^{-
{ \k ^2\over  2   } } .
\end{equation}
This proves assertions (i) and (ii)  of the Theorem. 
  \vskip 2 pt 
 As to assertions (iii)-(iv),  we use  %from $\|{\mathcal S}\|_2 \le K  s$ and the definition of $Y_n$, and 
  the fact (see Remark \ref{R.geometric.blocks} and estimate \eqref{estim.A2.R}) that the proof given remains true if instead of sieving by the geometric sequence $2^i, i\ge 1$, we sieve with the geometric sequence $R^i, i\ge 1$, where $R>1$ is arbitrary; that is \eqref{Ya1}, \eqref{2.mean} are true for these sequences. 
 Next using  \eqref{llt1} and  that $M(t)$ is slowly varying near infinity allows to conclude to \eqref{Ya.}.

   \vskip 5 pt  
 (v) % (Optimality of assumption \eqref{midiv}) %\label{opt.midiv} 
  The attentive reader will have certainly   observed that assumption \eqref{midiv} is not required in the proof of Theorem \ref{inc},  but only in the application of Proposition \ref{8.2}. Consider the case where $\sum_{j=1}^\infty m_j <\infty$ and assume that $\Phi(x) = \mathcal O(1)$.  See Example  \ref{example4}-(i). In that case the above inequality takes the simpler form
 \begin{eqnarray*}
       \E\Big|\sum_{I\le i\le J} Z_i\Big|^2 &\le &C_{c,D}  \Big(\sum_{I\le i \le J}m_i \Big) .\end{eqnarray*} 
 Assume only a little more, namely that   $\sum_{i=1}^\infty  m_i(\log i)^2<\infty$, which by   definition of $m_i$ in \eqref{mi} means that 
\begin{eqnarray*} 
   \sum_{n=1}^\infty {\t_n(\log\log n)^2\over \s_n\sqrt{\nu_n} }
<\infty  \, ,\end{eqnarray*}
which is \eqref{miconv1}. We use Remark \ref{U.bounded}, and more precisely,   implication \eqref{8.12.U.bounded.as.cv}, thank to which it follows that the series $\sum_{l=1}^\infty Z_l $, namely the sequence 
 $$  \sum_{n< 2^J} {Y_n\over  
 \s_n\sqrt{\nu_n}}\qq J=1,2,\ldots,$$
  converges almost surely. Therefore as $Y_n= \s_n  ({\bf 1}_{\{S_n=\kappa_n\}}-{\mathbb P}\{S_n=\kappa_n\})$, the sequence 
$$\sum_{n< 2^J}  {{\bf 1}_{\{S_n=\kappa_n\}}-{\mathbb P}\{S_n=\kappa_n\}\over  
  \sqrt{\nu_n}}\qq J=1,2,\ldots,$$ 
  converges  almost surely. Consequently, if assumption \eqref{miconv1} is fulfilled,    a     strong but  different almost sure convergence result takes place.
 \end{proof}

 \subsection{The quasi-orthogonal case}We prove the following 
  \begin{theorem}\label{asllt.quasi.ortho}
 Assume that  $h(x)  =\mathcal O(1)$, $m_i\asymp (1)$ and further  that condition \eqref{omega.star}
 is satisfied. 
 %\cc 
 %    \beq \label{omega.star}
%\o_\t^*: =\sup_{m\ge 1} \,\o(m)<\,\infty.
%\eeq
%Then
%  the sequence $\{\frac{Z_i}{m_i},i\ge 1\}$  is  a quasi-orthogonal system. 
 % \vskip 5 pt  
Then  the sequence $\{ {Z_i} ,i\ge 1\}$  is  a quasi-orthogonal system.
\end{theorem}

 \begin{proof}
 %[Proof of Theorem \ref{asllt.quasi.ortho}]
  Let $i$, $j$ be positive integers such that either $j>i+1$ or $j<i-1$. Obviously
 \begin{eqnarray*} | \E Z_iZ_j|&\le &\sum_{2^i\le \nu_m<2^{i+1}\atop
2^j\le \nu_n<2^{j+1}}  {\vartheta_m\vartheta_n\over\s_m\sqrt{\nu_m}\s_n\sqrt{\nu_n}} \ |\E Y_nY_m|
\cr &\le &C_{D}  \, \sum_{2^i\le \nu_m<2^{i+1}\atop
2^j\le \nu_n<2^{j+1}}\ {\vartheta_m\vartheta_n\over \s_n\s_m\sqrt{\nu_m\nu_n}} \,   \sqrt{\nu_m\over \nu_n}
\cr & \le  &C_{D} \,  2^{-(j-i)/2}\Big( \sum_{2^i\le \nu_m<2^{i+1} }{\vartheta_m \over \s_m \sqrt{\nu_m}  }\Big) \Big(  \sum_{ 
2^j\le \nu_n<2^{j+1}}{ \vartheta_n\over  \s_n \nu_n   }\Big)  
\cr & =  &C_{D}  \, 2^{-(j-i)/2} m_im_j.
 \end{eqnarray*}
  Therefore,
 \begin{eqnarray*} \sum_{|j-i|>1}\frac{| \E Z_iZ_j|}{ m_im_j}&\le & C_{D} \, H.
 \end{eqnarray*}
    In view of \eqref{sum.squared.} and the assumptions made,
 \begin{align*}
 \E Z_i^2  &\,\le \, C_\t\,\Big(\sum_{2^{  i}  <\nu_m\le 2^{i+1} }  {\vartheta_m    \over \s_m  \sqrt{\nu_m }  }   \Big)\, = C_\t\,m_i.
   \end{align*}
Since  $m_i\asymp (1)$ we have,
%$$  (iid \,case)\qq m_i=\sum_{2^i\le m<2^{i+1} }{\vartheta_m\over \s_m\sqrt{\nu_m} }=\sum_{2^i\le m<2^{i+1} }\frac{ 1}{ \log m }={\rm const.}$$
  \begin{align*}
 \E \big(\frac{Z_i}{m_i}\big)^2  &\,\le \,  \frac{C_\t}{m_i}\le C_\t.  \end{align*}
 % if $m_i\ge m(\t,\s)>0$.
 %{\color{blue} ICI IL Y AVAIT UN PROBLEME, SAUF SI $m_i\asymp C$}   By  \eqref{sum.i,i+1} 
%  $\E| Z_iZ_{i+1}|\le   C_\t\,m_i $,  it follows that and  inequality \eqref{qfbaa} in Lemma \ref{qfb}, it follows that 
Consequently
\begin{equation}\label{ZiZj.ortho.cond} \sup_{i\ge 1}\  \sum_{j\ge 1} | \E Z_iZ_j|  \le   C_{D} \, H.   \end{equation}
 
 This shows by using    criterion    \eqref{quasi.o.criterion}, that   the sequence $\{ {Z_i} ,i\ge 1\}$  is  a quasi-orthogonal system.

\end{proof}

 \vskip 5 pt 
 
 Note that (see Example  \ref{example4}-(i)) in the i.i.d. case, $h(x)  \equiv C $ (by Remark \ref{bdd.}), and condition \eqref{omega.star} holds. Further $$  m_i=\sum_{2^i\le \nu_m<2^{i+1} }{\vartheta_m\over \s_m\sqrt{\nu_m} }= \, \sum_{2^i\le \nu m<2^{i+1} }\frac{ \nu}{   m }={\rm const.}$$
Therefore we have
\begin{corollary}[i.i.d. case]\label{iid.qos} The corresponding sequence  $\{ {Z_i} ,i\ge 1\}$  is  a quasi-orthogonal system.
\end{corollary}
This is no longer true when passing to subsequences! See Theorem \ref{inc.subseq}  and comments hereafter. \vskip 3 pt 

\vskip 3 pt 
Another consequence is 
   \begin{corollary}[a.s. convergent series and asllt]\label{qos.iid}   Let $b>3/2$. Then the series
\begin{eqnarray}\label{series.asllt}
  \sum_{j>1} {1\over   j^{1/2} (\log j)^{ b}}\ \Big( \sum_{m= 2^{j-1}}^{2^j -1}\frac {    {\bf 1}_{\{S_m=\kappa_m\}}-{\mathbb P}\{S_m=\kappa_m\} }{\sqrt m} \Big),\end{eqnarray}
converges   almost surely.   
  \vskip 3 pt
  Also,
 \begin{eqnarray}\label{series.asllt1}\lim_{N\to\infty}
  {1\over   N^{1/2} (\log N)^{ b} }\sum_{j=1}^{2^N} \frac {    {\bf 1}_{\{S_m=\kappa_m\}}-{\mathbb P}\{S_m=\kappa_m\} }{\sqrt m}= 0,  
  \end{eqnarray}
   almost surely.   
 \vskip 2 pt  Further,  
$$ \lim_{ N\to \infty}{1\over    \log N } \sum_{ n\le
N}  {  1 \over \sqrt n} {\bf 1}_{\{S_n=\kappa_n\}} \buildrel{a.s.}\over {=}{D\over
\sqrt{ 2\pi}\s}e^{-  {\k^2/ ( 2\s^2 ) } },$$
  {\it for any  sequence of integers $\k_n\in \mathcal L(nv_0,D)$, $n=1,2,\ldots$      such that \eqref{eq2} holds.} 
  \end{corollary}
%\begin{remark}
    The first claim (\cite{W3}) cannot be derived from  G\'al--Koksma's criterion (\cite{PS}, p.\,134), unlike to the second one. The third one is Theorem \ref{t1[asllt]..}.%\end{remark}

 \begin{proof} It is very short. 
As Rademacher--Menchov's Theorem applies to quasi-orthogonal systems, the   series
\begin{eqnarray*}
  \sum_j {Z_j\over   j^{1/2} (\log j)^{ b} } \end{eqnarray*}
 converges   almost surely if $b>3/2$. 
  By Kronecker's Lemma,
 ${1\over   N^{1/2} (\log N)^{ b} }\sum_{j=1}^N Z_j \ \to 0,
 $ as $N$ tends to infinity, almost surely, whence the second claim.  
%The result  follows from the definition of $Z_i$, see \eqref{Zi}.
%$${1\over   N^{1/2} (\log N)^{ b} }\sum_{j=1}^N Z_j  ={1\over   N^{1/2} (\log N)^{ b} } \sum_{  1\le n<2^{N+1} }{Y_n\over n}$$
%$$={1\over   N^{1/2} (\log N)^{ b} } \sum_{  1\le n<2^{N+1} }\Big(\frac{{\bf 1}_{\{S_n=\kappa_n\}}-\P\{S_n=\kappa_n\}}{\sqrt n } \Big).   $$
  By arguing as in  Remark \ref{R.geometric.blocks}, this remains true along any $R$-geometric sequence, $R>1$. Now   by  Gnedenko's Theorem \eqref{llt.iid},  if
$ \k_n \in \mathcal L(nv_0,D)$ is a sequence which   verifies condition (\ref{eq2}), then \eqref{llt1} holds. The proof is achieved by applying  the   following well-known elementary Lemma.\begin{lemma}Let $\{f_n,n\ge 1\}$ be a sequence of nonnegative
numbers.   If for each $r>1$, the  averages
$  r^{-k}\sum_{n\le r^k}f_n$
converges, as $k\to \infty$.
Then for each $r>1$ this limit is the same, call it $L$, and we have
$\lim_{N\to\infty}N^{-1}\sum_{n\le N}f_n=L$.
\end{lemma}
 %\begin{equation}\label{llt1a}  \lim_{n\to \infty}  \sqrt n {\mathbb P}\{S_n=\k_n\}={D\over  \sqrt{ 2\pi}\s}e^{-
%{ \k ^2\over  2   \s^2} } .
%\end{equation}
   \end{proof}

\vskip 5 pt We pass to another example.
\begin{corollary}[Cram\'er model] Let   $\xi=\{\xi_j, j\ge 1\}$ be the sequence defined in Example \ref{cramer}. Then the corresponding sequence  $\{ {Z_i} ,i\ge 1\}$  is  a quasi-orthogonal system.
\end{corollary}
\begin{proof}Recall  that $\t_m=\frac{1}{\log m}$, and  $\s_m^2\sim\nu_m\sim \frac{m}{\log m}$, as $m\to \infty$. Thus $h(x)  \equiv C $, further 
$$  m_i\sim \sum_{2^i\le \nu_m<2^{i+1} }{  \log m \over (\log m) m }\sim \sum_{2^i\le\frac{m}{\log m}<2^{i+1} }{1\over m } $$$$\sim  \sum_{i2^i\le m<(i+1) 2^{i+1} }{1\over m }\sim  \log (i+1) + (i+1) \log 2-  \log i  -i \log 2\sim {\rm const.} $$
\end{proof}

 % \begin{corollary}[{\bf i.i.d. case}]  placÃ© aprÃšs end{document}

%change le 22.05, voir apres end.document
\section{\gsec  REVISITING THE ASLLT IN THE I.\,I.\,D. 
%NDEPENDENT    IDENTICALLY DISTRIBUTED 
CASE
}\label{sub2..1}
The goal of this Section is to prove that in the i.\,i.\,d. case, an almost sure local limit theorem with explicit speed of convergence (Theorem \ref{asllt.speed.iid}) expressed with an a.s. convergent series. 
Let $X$ be an  $\mathcal L(v_0,D)$-valued square integrable  random variable, with maximal span $D$, $\m ={\mathbb E\,} X$,
$\s^2={\rm Var} (X)>0$. Assume that $\t_{X}>0$ and choose $\t =\t_{X }$.  Let  
$S_n=X_1+\ldots +X_n$,     $ X_k   $  being independent copies of
$X$.  
  Thus $\nu_n= \t n$, $\s_n= \s  \sqrt{  n}$ and assumptions \ref{nun} and \ref{midiv} are satisfied. 
%  As $\frac{ \vartheta_n }{  \s_n\sqrt{\nu_n}}=\frac{ \t }{    \,n }$,   the series $\sum_{n\ge 1}\frac{ \vartheta_n }{   \s_n\sqrt{\nu_n}}$ is divergent.   
 \vskip 1 pt 
 Let  $\k_j\in \mathcal L(jv_0,D)$, $j=1,2,\ldots$   be   such that assumption \eqref{nun.ka.iid} is satisfied. 
 %   by \eqref{basic0},
%  \begin{equation}\label{basic0}{\mathbb E\,} Y_n^2=\s_n^2 \,{\mathbb P}\{S_n=\kappa_n\}\big(1-{\mathbb P}\{S_n=\kappa_n\} \big) ={\mathcal O}(\s_n).\end{equation}
%$\E Y_n^2 \le  
%C \s_n=
%C \s \sqrt n$.
%As $Y_n= \s\sqrt n  \big({\bf 1}_{\{S_n=\kappa_n\}}-{\mathbb P}\{S_n=\kappa_n\}\big)$, 
The system of random variables   \eqref{cor.syst} writes: 
 \beq\label{cor.syst.} {\vartheta_nY_n\over  
 \s_n\sqrt{\nu_n}}  = {\t Y_n\over  
 n}  = \t\ { {\bf 1}_{\{S_n=\kappa_n\}}-{\mathbb P}\{S_n=\kappa_n\}  \over  
  \sqrt n}    , \qq\quad n\ge 1,
\eeq
 \beq\label{Zi}Z_i  = \sum_{m= 2^{i-1}}^{2^i -1}\frac {Y_m}{m}=\s \sum_{m= 2^{i-1}}^{2^i -1}\frac {    {\bf 1}_{\{S_m=\kappa_m\}}-{\mathbb P}\{S_m=\kappa_m\} }{\sqrt m}, \qq\quad i\ge 1.
\eeq

A natural question is whether it is possible in   \eqref{series.asllt} to substitute to ${\mathbb P}\{S_m=\kappa_m\}$,  its limit ${D\over  \sqrt{ 2\pi}\s}e^{-
{ \k ^2/2   \s^2} }$.    Let \begin{equation}\label{llt1delta.n} \d_n=\Big| \s\sqrt n {\mathbb P}\{S_n=\k_n\}-{D\over  \sqrt{ 2\pi} }e^{-
{ (\k_n-n\m) ^2\over  2 n  \s^2} }\Big|.
\end{equation} By Gnedenko's Theorem \eqref{llt.iid}, $\d_n=o(1)$.  Recall that   $\t_X$ is defined in \eqref{vartheta}.  
\begin{theorem}\label{asllt.speed.iid}  Assume that  $\t_X >0$.    
 Let $  \k_n= n\m  +   \k\sqrt n(1+\e_n)$, $\e_n\to\, 0$.
\vskip 3 pt {\rm (i)} We have 
\begin{eqnarray*}
\frac{\sum_{  1\le n<2^{N+1} }  \frac{1}{\sqrt n } {\bf 1}_{\{S_n=\kappa_n\}}}{   \sum_{  1\le n<2^{N+1} }\frac{1}{  n } }=  {D\over  \s \sqrt{ 2\pi} }e^{-
{  \k  ^2\over  2   \s^2} }  \ +\mathcal O\Big({1\over   N    } \sum_{  1\le n<2^{N+1} }\frac{\d_n+\e_n}{n } \Big)
  + o\Big({     (\log N)^{ b }\over  N^{1/2} }\Big).
\end{eqnarray*}

{\rm (ii)} Assume that 
\begin{eqnarray*}\sum_j {1\over   j^{1/2} (\log j)^{ b} }\,\sum_{2^j\le n<2^{j+1}}\frac{\d_n+\e_n}{n }<\infty.
\end{eqnarray*}
This is fulfilled if $\max(\d_n,\e_n)=\mathcal O\big( \log^{-1/2}n \big)$. Then
the series 
$$\sum_j {1\over   j^{1/2} (\log j)^{ b} } \bigg\{ \sum_{2^j\le n<2^{j+1}}\frac{1}{\sqrt n } {\bf 1}_{\{S_n=\kappa_n\}}  
   - \Big({D \log 2\over  \s \sqrt{ 2\pi} }e^{-
{  \k  ^2\over  2   \s^2} }\Big)\bigg\}$$
converges   almost surely. 
\end{theorem}
\begin{proof}
%Put $Y_n= \sqrt n \big({\bf 1}_{\{S_n=\kappa_n\}}-{\mathbb P}\{S_n=\kappa_n\} \big)  $.
%Then there exists a constant
%$C $  such that for all $1\le m<n$
%\begin{eqnarray} \label{C1}\big|{\mathbb E\,}   Y_nY_m\big|
% &\le &  C  \Big\{     {   1\over
% \sqrt{n\over m}-1   } +      {n^{1/2} \over
% (n-m) ^{3/2}} \Big\} .
%\end{eqnarray}
 %  And   given  $0<c<1$, that there exists a constant
%$C  $  such that for all $1\le m\le cn$,
%\begin{eqnarray}\label{C2} \big|{\mathbb E\,}   Y_nY_m\big|
% &\le &  C  \, \sqrt{{m\over n}}. \end{eqnarray}

 We have 
 %\begin{equation}\label{llt1a}  \lim_{n\to \infty}  \sqrt n {\mathbb P}\{S_n=\k_n\}={D\over  \sqrt{ 2\pi}\s}e^{-
%{ \k ^2\over  2   \s^2} } .
%\end{equation}
 %As
 $$\Big| e^{-
{ (\k_n-n\m) ^2\over  2 n  \s^2} }
- e^{-{  \k  ^2\over  2   \s^2} }
\Big|\le  \frac{1}{2\s^2}\big|{ (\k_n-n\m) ^2\over    n   } -    \k  ^2 \big|= \frac{\k  ^2 }{2\s^2}\big|  (1+\e_n) ^2 -     1\big| \le  \frac{3\k  ^2  }{2\s^2}\,\e_n.$$
Thus \begin{eqnarray*}
& & \Big| \s\sqrt n {\mathbb P}\{S_n=\k_n\}-{D\over  \sqrt{ 2\pi} }e^{-
{  \k  ^2\over  2   \s^2} } \Big|
\cr &\le &\Big| \s\sqrt n {\mathbb P}\{S_n=\k_n\}-{D\over  \sqrt{ 2\pi} }e^{-
{  \k  ^2\over  2   \s^2} }e^{-
{ (\k_n-n\m) ^2\over  2 n  \s^2} }\Big| +   \frac{3D\k  ^2  }{2\sqrt{ 2\pi} \s^2}\,\e_n
\ \le \ \d_n + \frac{3\k  ^2  }{2\s^2}\,\e_n.
\end{eqnarray*}
 
So that
\begin{eqnarray*}\frac{Y_n}{n}&=&\frac{1}{n}\Big( \sqrt n \Big\{{\bf 1}_{\{S_n=\kappa_n\}}-{\mathbb P}\{S_n=\kappa_n\} \Big\}\Big) \,=\,
\frac{1}{\sqrt n } {\bf 1}_{\{S_n=\kappa_n\}}  -\frac{\sqrt n{\mathbb P}\{S_n=\kappa_n\}}{n } 
\cr &=&  \frac{1}{\sqrt n } {\bf 1}_{\{S_n=\kappa_n\}}  -  \Big({D\over  \s \sqrt{ 2\pi} }e^{-
{  \k  ^2\over  2   \s^2} }\Big)\, \frac{1}{n }+\mathcal O\big(\frac{\d_n+\e_n}{n }\big).
 \end{eqnarray*}

%   Put for any positive integer $j$
%$$ Z_j=\sum_{2^j\le n<2^{j+1}}{Y_n\over n} . $$
%By   the local limit theorem,  
 % $ {\mathbb E\,} Y_n^2=n
%{\mathbb P}\{S_n=\kappa_n\}\big(1-{\mathbb P}\{S_n=\kappa_n\} \big) ={\mathcal O}(\sqrt n).
 %$
%This and estimate  \eqref{C2}  imply that    $\{Z_j, j\ge 1\}$ is a quasi-orthogonal system.
%As Rademacher--Menchov's Theorem applies to quasi-orthogonal systems, the   series
%By Theorem \ref{qos.iid}, 
%\begin{eqnarray*}
 % \sum_j {Z_j\over   j^{1/2} (\log j)^{ b} } \end{eqnarray*}
 % converges   almost surely if $b>3/2$. 
% \vskip 3 pt 

%\vskip 3 pt 
(i) By Theorem \ref{qos.iid},  if $b>3/2$,
%By Kronecker's Lemma,
$${1\over   N^{1/2} (\log N)^{ b} }\sum_{j=1}^N Z_j \ \to 0,
$$ as $N$ tends to infinity, almost surely. 
But  
$${1\over   N^{1/2} (\log N)^{ b} }\sum_{j=1}^N Z_j  ={1\over   N^{1/2} (\log N)^{ b} } \sum_{  1\le n<2^{N+1} }{Y_n\over n}$$
$$={1\over   N^{1/2} (\log N)^{ b} } \sum_{  1\le n<2^{N+1} }\Big(\frac{1}{\sqrt n } {\bf 1}_{\{S_n=\kappa_n\}}  -  \Big({D\over  \s \sqrt{ 2\pi} }e^{-
{  \k  ^2\over  2   \s^2} }\Big)\, \frac{1}{n }+\mathcal O\big(\frac{\d_n+\e_n}{n }\big)\Big)$$
 $$={1\over   N^{1/2} (\log N)^{ b} } \Big(  \sum_{  1\le n<2^{N+1} }\frac{1}{  n }\Big)\bigg\{\frac{\sum_{  1\le n<2^{N+1} }\frac{1}{\sqrt n } {\bf 1}_{\{S_n=\kappa_n\}}}{   \sum_{  1\le n<2^{N+1} }\frac{1}{  n } }   -  \Big({D\over  \s \sqrt{ 2\pi} }e^{-
{  \k  ^2\over  2   \s^2} }\Big)\bigg\} $$ $$+\mathcal O\Big({1\over   N^{1/2} (\log N)^{ b} } \sum_{  1\le n<2^{N+1} }\frac{\d_n+\e_n}{n } \Big)\ \to 0,
$$ as $N$ tends to infinity, almost surely. 
 Therefore ($ \sum_{  1\le n<2^{N+1} }\frac{1}{  n }\sim C N$),
\begin{eqnarray*}
\frac{\sum_{  1\le n<2^{N+1} }  \frac{1}{\sqrt n } {\bf 1}_{\{S_n=\kappa_n\}}}{   \sum_{  1\le n<2^{N+1} }\frac{1}{  n } }=  {D\over  \s \sqrt{ 2\pi} }e^{-
{  \k  ^2\over  2   \s^2} }  \ +\mathcal O\Big({1\over   N    } \sum_{  1\le n<2^{N+1} }\frac{\d_n+\e_n}{n } \Big)
  + o\Big({     (\log N)^{ b }\over  N^{1/2} }\Big).
\end{eqnarray*}
 \vskip 3 pt 
\vskip 3 pt 
(ii) Now assume that 
\begin{eqnarray*}\sum_j {1\over   j^{1/2} (\log j)^{ b} }\,\sum_{2^j\le n<2^{j+1}}\frac{\d_n+\e_n}{n }<\infty.
\end{eqnarray*}
 
Then the series 
$$\sum_j {1\over   j^{1/2} (\log j)^{ b} } \bigg\{ \sum_{2^j\le n<2^{j+1}}\frac{1}{\sqrt n } {\bf 1}_{\{S_n=\kappa_n\}}  
   - \Big({D \log 2\over  \s \sqrt{ 2\pi} }e^{-
{  \k  ^2\over  2   \s^2} }\Big)\bigg\}$$
converges   almost surely. 

\vskip 3 pt
Indeed, first the series
$$\sum_j {1\over   j^{1/2} (\log j)^{ b} } \bigg\{ \sum_{2^j\le n<2^{j+1}}\frac{1}{\sqrt n } {\bf 1}_{\{S_n=\kappa_n\}}  
   - \Big(\sum_{2^j\le n<2^{j+1}}\frac{1}{n }\Big) \Big({D\over  \s \sqrt{ 2\pi} }e^{-
{  \k  ^2\over  2   \s^2} }\Big)\bigg\}$$
converges   almost surely, since  
\begin{eqnarray*}
& &\sum_j {Z_j\over   j^{1/2} (\log j)^{ b} } \,=\, \sum_j {1\over   j^{1/2} (\log j)^{ b} }\sum_{2^j\le n<2^{j+1}}{Y_n\over n}
\cr &=& \sum_j {1\over   j^{1/2} (\log j)^{ b} }\bigg\{ \sum_{2^j\le n<2^{j+1}}\frac{1}{\sqrt n } {\bf 1}_{\{S_n=\kappa_n\}}   
   - \Big(\sum_{2^j\le n<2^{j+1}}\frac{1}{n }\Big) \Big({D\over  \s \sqrt{ 2\pi} }e^{-
{  \k  ^2\over  2   \s^2} }\Big)\bigg\}
 \cr & & \quad  +\mathcal O\Big(\sum_j {1\over   j^{1/2} (\log j)^{ b} }\,\sum_{2^j\le n<2^{j+1}}\frac{\d_n+\e_n}{n }\Big). \end{eqnarray*}  
Besides 
$$ \sum_{2^j\le n<2^{j+1}}\frac{1}{ n }=\log 2+ \mathcal O\big(2^{-j}\big),$$
so that the simplified series 
$$\sum_j {1\over   j^{1/2} (\log j)^{ b} } \bigg\{ \sum_{2^j\le n<2^{j+1}}\frac{1}{\sqrt n } {\bf 1}_{\{S_n=\kappa_n\}}  
   - \Big({D \log 2\over  \s \sqrt{ 2\pi} }e^{-
{  \k  ^2\over  2   \s^2} }\Big)\bigg\}$$
converges   almost surely. 

\end{proof}

%%%%%%%%%%%%%%%%%
%\section{\gsec   I.\,I.\,D. SEQUENCES WITH ARITHMETICAL WEIGHTS
%}\label{sub2....1}

  \section{\gsec CONCLUDING REMARKS}\label{s6}
In a previous attempt, for proving the ASLLT (Theorem \ref{t1[asllt].}), we considered slightly different block sums, 
$$ Z_i=\sum_{2^i\le n<2^{i+1}}{Y_n\over  
 \s_n\sqrt{\nu_n}} , $$
 and operated differently in order the control their  covariances; in particular   the diagonal case was not treated separately. 
   We bounded $| \E Z_iZ_j|$ from above, by proceeding as follows: for $1\le i\le j $, \begin{eqnarray*} | \E Z_iZ_j|&\le &\sum_{{2^i\le m<2^{i+1}\atop
2^j\le n<2^{j+1}}\atop 1\le \nu_m\le c\nu_n}  {|\E Y_nY_m|\over\s_m\sqrt{\nu_m}\s_n\sqrt{\nu_n}}+ \sum_{{2^i\le m<2^{i+1}\atop
2^j\le n<2^{j+1}}\atop   \nu_m> c\nu_n}  {|\E Y_nY_m|\over\s_m\sqrt{\nu_m}\s_n\sqrt{\nu_n}}.\end{eqnarray*}
 the first  sum can be estimated efficiently, the estimation of the second imposed restrictive assumptions. We have not tried any further, and swiched to the one implemented in this paper. Maybe the preceding approach can still be improved, at the price of some new argument, similar to the ones used here. We don't believe it would bring more.

% \begin{problem}Let $\mathcal N$ be some   increasing  infinite   sequence of positive integers. It is an interesting question to know how the ASLLT extends   when  ${(k_n-a_n)/ b_n}\to
%\kappa$   along $\mathcal N$ only.  
%\end{problem}A typical instance when the Bernoulli part method does not applies,  is provided by the following example. 
%Let $\{k_j, j\ge 1\}$ be  an increasing sequence of positive integers. Consider the sums
%$$B_\nu=
%\b_1+\ldots +\b_\nu, \qq \quad \nu=1, 2,\ldots$$
%where  $\b_j$ are independent binomial random variables  defined by
%\begin{eqnarray}\label{betai}{\mathbb P}\{
%\b_j=0\}= \t_j,\qq
%{\mathbb P}\{\b_j=k_j\}  =1-\t_j,\end{eqnarray}
%  with $0<\t_j<1$ for each $j$, and $k_j$ are increasing positive weights.
 %   The problem obviously amounts to study weighted  sums of Bernoulli independent random variables. Informations on these sums are required and are generally unknown.

 \vskip 9 pt 

%\noi {\bf Acknowledgments.}  

\end{document}